\documentclass[12pt]{amsart}
\usepackage{graphicx}
\usepackage{amsmath}
\usepackage{amsfonts}
\usepackage{amssymb}

\usepackage{setspace}
\usepackage{datetime}
\usepackage[colorlinks,
            linkcolor=black,
            anchorcolor=blue,
            citecolor=black]{hyperref}
\usepackage{geometry}
\allowdisplaybreaks[4]

\newtheorem{thm}{Theorem}[section]
\newtheorem{cor}[thm]{Corollary}

\newtheorem{lem}[thm]{Lemma}

\theoremstyle{definition}

\theoremstyle{remark}
\newtheorem{rem}[thm]{Remark}
\theoremstyle{conclusion}

\theoremstyle{question}
\numberwithin{equation}{section}

\newcommand\tbbint{{-\mkern -16mu\int}}

\newcommand\dbbint{{-\mkern -19mu\int}}

\newcommand\bbint{
{\mathchoice{\dbbint}{\tbbint}{\tbbint}{\tbbint}}
}
\newcommand{\lr}{\left(}
\newcommand{\rr}{\right)}
\newcommand{\R}{\mathbb{R}}
\newcommand{\C}{\mathbb{C}}

\newcommand{\be}{\begin{equation}}
\newcommand{\ee}{\end{equation}}

\newcommand{\la}{\lambda}

\newcommand{\var}{\varepsilon}

\begin{document}
\title[Classification results of $3$-D $\&$ $4$-D conformally invariant systems]
{Classification of solutions to $3$-D and $4$-D mixed order conformally invariant systems with critical and exponential growth}

\author{Wei Dai$^*$, Lixiu Duan$^\dag$, Rong Zhang$^\ddag$}

\address[Wei Dai]{$^*$ School of Mathematical Sciences, Beihang University (BUAA), Beijing 100191, P. R. China, and Key Laboratory of Mathematics, Informatics and Behavioral Semantics, Ministry of Education, Beijing 100191, P. R. China}
\email{weidai@buaa.edu.cn}

\address[Lixiu Duan]{$^\dag$ School of Mathematical Sciences, Beihang University (BUAA), Beijing 100191, P. R. China}
\email{lixiuduan@buaa.edu.cn}

\address[Rong Zhang]{$^\ddag$ HLM, Academy of Mathematics and Systems Science, Chinese Academy of Sciences, Beijing 100190, P. R. China}
\email{zhangrong@amss.ac.cn}

\thanks{\noindent Wei Dai and Lixiu Duan are supported by the NNSF of China (No. 12222102 and No. 11971049), the National Key R$\&$D Program of China (2022ZD0116401) and the Fundamental Research Funds for the Central Universities. Rong Zhang is supported by the National Funded Postdoctoral Program of China (No. GZC20232913).}

\begin{abstract}
In this paper, without any assumption on $v$ and under the extremely mild assumption $u(x)= O(|x|^{K})$ as $|x|\rightarrow+\infty$ for some $K\gg1$ arbitrarily large, we classify solutions of the following conformally invariant system with mixed order and exponentially increasing nonlinearity in $\mathbb{R}^{3}$:
$$
\begin{cases}
\ (-\Delta)^{\frac{1}{2}} u=v^{4} ,&x\in \mathbb{R}^{3},\\
\ -\Delta v=e^{pw} ,&x\in \mathbb{R}^{3},\\
\ (-\Delta)^{\frac{3}{2}} w=u^{3} ,&x\in \mathbb{R}^{3},
\end{cases}
$$
where $p>0$, $u,v\geq0$, $w(x)=o(|x|^{2})$ at $\infty$ and $u$ satisfies the finite total curvature condition $\int_{\mathbb{R}^{3}}u^{3}(x)\mathrm{d}x<+\infty$. Moreover, under the extremely mild assumption that \emph{either} $u(x)$ or $v(x)=O(|x|^{K})$ as $|x|\rightarrow+\infty$ for some $K\gg1$ arbitrarily large \emph{or} $\int_{\mathbb{R}^{4}}e^{\Lambda pw(y)}\mathrm{d}y<+\infty$ for some $\Lambda>0$, we also prove classification of solutions to the conformally invariant system with mixed order and exponentially increasing nonlinearity in $\mathbb{R}^{4}$:
\begin{align*}
\begin{cases}
\ (-\Delta)^{\frac{1}{2}} u=e^{pw} ,&x\in \mathbb{R}^{4},\\
\ -\Delta v=u^2 ,&x\in \mathbb{R}^{4},\\
\ (-\Delta)^{2} w=v^{4} ,&x\in \mathbb{R}^{4},
\end{cases}
\end{align*}
where $p>0$, $u,v\geq0$, $w(x)=o(|x|^{2})$ at $\infty$ and $v$ satisfies the finite total curvature condition $\int_{\mathbb{R}^{4}}v^{4}(x)\mathrm{d}x<+\infty$. The key ingredients are deriving the integral representation formulae and crucial asymptotic behaviors of solutions $(u,v,w)$ and calculating the explicit value of the total curvature. When $p=\frac{5}{2}$, these systems are closely related to single conformally invariant equations $(-\Delta)^{\frac{1}{2}}u=u^{\frac{n+1}{n-1}}$, $-\Delta v=v^{\frac{n+2}{n-2}}$ in $\mathbb{R}^{n}$ with $n=3,4$, $(-\Delta)^{\frac{3}{2}}w=2e^{3w}$ in $\mathbb{R}^{3}$ and $(-\Delta)^{2}w=6e^{4w}$ in $\mathbb{R}^{4}$, which have been quite extensively studied (cf. \cite{CY,CLL,CLZ,Lin,WX,Z} etc).
\end{abstract}
\maketitle {\small {\bf Keywords:} Classification of solutions; Conformally invariant; Systems with mixed order; Exponentially increasing nonlinearity; Method of moving spheres; Fractional Laplacians. \\

}

\section{Introduction}

\subsection{Conformally invariant systems with mixed order and critical and exponential growth in $\mathbb{R}^{3}$}
In this paper, we first investigate the following conformally invariant system with mixed order and critical and exponential growth in $\mathbb{R}^{3}$:
\begin{equation}\label{a1}
\begin{cases}
\ (-\Delta)^{\frac{1}{2}} u=v^{4} ,&x\in \mathbb{R}^{3},\\
\ -\Delta v=e^{pw} ,&x\in \mathbb{R}^{3},\\
\ (-\Delta)^{\frac{3}{2}} w=u^{3} ,&x\in \mathbb{R}^{3},
\end{cases}
\end{equation}
where $p\in(0,+\infty)$, $u,v\geq0$, $w(x)=o(|x|^{2})$ as $|x|\rightarrow+\infty$ and $u$ satisfies the finite total curvature condition $\int_{\mathbb{R}^{3}}u^{3}(x)\mathrm{d}x<+\infty$.

\smallskip

We assume $(u,v,w)$ is a pair of classical solution to the $3$-D system \eqref{a1} in the sense that $u\in C^{1,\epsilon}_{loc}(\mathbb{R}^{3})\cap \mathcal{L}_{1}(\mathbb{R}^{3})$ with arbitrarily small $\epsilon>0$, $v\in C^{2}(\mathbb{R}^{3})$, and $w\in C_{loc}^{3,\epsilon}(\mathbb{R}^{3})$ with arbitrarily small $\epsilon>0$ such that $w$ or $\Delta w\in\mathcal{L}_{1}(\mathbb{R}^{3})$. The square root of the Laplacian $(-\Delta)^{\frac{1}{2}}$ is a particular case of general fractional Laplacians $(-\Delta)^{\frac{\alpha}{2}}$ with $\alpha=1$. In $\mathbb{R}^{n}$ with $n\geq1$, for any $u\in C^{[\alpha],\{\alpha\}+\epsilon}_{loc}(\mathbb{R}^{n})\cap\mathcal{L}_{\alpha}(\mathbb{R}^{n})$, the nonlocal operator $(-\Delta)^{\frac{\alpha}{2}}$ ($0<\alpha<2$) is defined by (see \cite{CT,CG,CLL,CLM,DQ,S})
\begin{equation}\label{nonlocal defn}
  (-\Delta)^{\frac{\alpha}{2}}u(x)=C_{n,\alpha} \, P.V.\int_{\mathbb{R}^n}\frac{u(x)-u(y)}{|x-y|^{n+\alpha}}\mathrm{d}y:=C_{n,\alpha}\lim_{\varepsilon\rightarrow0}\int_{|y-x|\geq\varepsilon}\frac{u(x)-u(y)}{|x-y|^{n+\alpha}}\mathrm{d}y,
\end{equation}
where $[\alpha]$ denotes the integer part of $\alpha$, $\{\alpha\}:=\alpha-[\alpha]$, the constant $C_{n,\alpha}=\left(\int_{\mathbb{R}^{n}} \frac{1-\cos \left(2 \pi \zeta_{1}\right)}{|\zeta|^{n+\alpha}} d \zeta\right)^{-1}$ and the (slowly increasing) function space
\begin{equation}\label{0-1-space}
  \mathcal{L}_{\alpha}(\mathbb{R}^{n}):=\left\{u: \mathbb{R}^{n}\rightarrow\mathbb{R}\,\big|\,\int_{\mathbb{R}^{n}}\frac{|u(x)|}{1+|x|^{n+\alpha}}\mathrm{d}x<+\infty\right\}.
\end{equation}

\smallskip

The fractional Laplacians $(-\Delta)^{\frac{\alpha}{2}}$ can also be defined equivalently (see \cite{CLM}) by Caffarelli and Silvestre's extension method (see \cite{CS}) for $u\in C^{[\alpha],\{\alpha\}+\epsilon}_{loc}(\mathbb{R}^{n})\cap\mathcal{L}_{\alpha}(\mathbb{R}^{n})$. For instance, the square root of the Laplacian $(-\Delta)^{\frac{1}{2}}$ can be defined equivalently for any $u\in C^{1,\epsilon}_{loc}(\mathbb{R}^{n})\cap\mathcal{L}_{1}(\mathbb{R}^{n})$ by
\begin{equation}\label{extension}
  (-\Delta)^{\frac{1}{2}}u(x):=-C_{n}\lim_{y\rightarrow0+}\frac{\partial U(x,y)}{\partial y}
  =-C_{n}\lim_{y\rightarrow0+}\int_{\mathbb{R}^{n}}\frac{|x-\xi|^{2}-ny^{2}}{\big(|x-\xi|^{2}+y^{2}\big)^{\frac{n+3}{2}}}u(\xi)d\xi,
\end{equation}
where $U(x,y)$ is the harmonic extension of $u(x)$ in $\mathbb{R}^{n+1}_{+}=\{(x,y)| \, x\in\mathbb{R}^{n}, \, y\geq0\}$. The definition \eqref{nonlocal defn} of the fractional Laplacian $(-\Delta)^{\frac{\alpha}{2}}$ can also be extended further to distributions in the space $\mathcal{L}_{\alpha}(\mathbb{R}^{n})$ by
\begin{equation}\label{distribution}
  \left\langle(-\Delta)^{\frac{\alpha}{2}}u,\phi\right\rangle=\int_{\mathbb{R}^n}u(x)(-\Delta)^{\frac{\alpha}{2}}\phi(x) \mathrm{d}x, \qquad \forall\phi\in C^{\infty}_0(\mathbb{R}^n).
\end{equation}

\smallskip

Throughout this paper, we define $(-\Delta)^{\frac{1}{2}}u$ and $(-\Delta)^{\frac{3}{2}}w:=(-\Delta)(-\Delta)^{\frac{1}{2}}w=(-\Delta)^{\frac{1}{2}}(-\Delta)w$ by definition \eqref{nonlocal defn} and its equivalent definition \eqref{extension} for $u\in C^{1,\epsilon}_{loc}(\mathbb{R}^{3})\cap\mathcal{L}_{1}(\mathbb{R}^{3})$ and $w\in C_{loc}^{3,\epsilon}(\mathbb{R}^{3})$ such that $w$ or $\Delta w\in\mathcal{L}_{1}(\mathbb{R}^{3})$, where $\epsilon>0$ is arbitrarily small. Due to the nonlocal virtue of $(-\Delta)^{\frac{1}{2}}$, we need the assumption $u\in C^{1,\epsilon}_{loc}(\mathbb{R}^{3})$ and $w\in C_{loc}^{3,\epsilon}(\mathbb{R}^{3})$ with arbitrarily small $\epsilon>0$ (merely $u\in C^{1}$ and $w\in C^{3}$ are not enough) to guarantee that $(-\Delta)^{\frac{1}{2}}u\in C(\mathbb{R}^{3})$ and $(-\Delta)^{\frac{3}{2}}w\in C(\mathbb{R}^{3})$ (see \cite{CLM,S}), and hence $(u,v,w)$ is a pair of classical solution to the $3$-D system \eqref{a1} in the sense that $(-\Delta)^{\frac{1}{2}}u$ and $(-\Delta)^{\frac{3}{2}}w$ is pointwise well defined and continuous in the whole space $\mathbb{R}^{3}$.

\smallskip

The fractional Laplacian $(-\Delta)^{\frac{\alpha}{2}}$ is a nonlocal integral-differential operator. It can be used to model diverse physical phenomena, such as anomalous diffusion and quasi-geostrophic flows, turbulence and water waves, molecular dynamics, and relativistic quantum mechanics of stars (see \cite{CV,Co} and the references therein). It also has various applications in conformal geometry, probability and finance (see \cite{Be,CT,CG} and the references therein). In particular, $(-\Delta)^{\frac{\alpha}{2}}$ with $0<\alpha<2$ can also be understood as the infinitesimal generator of a stable L\'{e}vy process (see \cite{Be}).

\medskip

Consider fractional order or higher order geometrically interesting conformally invariant equation of the form
\begin{equation}\label{GPDE}
  (-\Delta)^{\frac{\alpha}{2}}u=u^{\frac{n+\alpha}{n-\alpha}} \qquad \text{in} \,\, \mathbb{R}^{n},
\end{equation}
where $n\geq1$ and $\alpha\in(0,n)\bigcup(n,+\infty)$. We say \eqref{GPDE} has subcritical or super-critical order if $\alpha<n$ or $\alpha>n$ respectively. In the special case $\alpha=2<n$, equation \eqref{GPDE} is the the well-known Yamabe problem. In higher order case that $\alpha\geq2$ is an even integer, equation \eqref{GPDE} arises from the conformal metric problems, prescribing $Q$-curvature problems, conformally covariant Paneitz operators and GJMS operators and so on ... (see e.g. \cite{CY,CL,CL1,CL2,FKT,GJMS,Juhl,Lin,Li,N,P,WX,Xu,Z} and the references therein). In the fractional order or fractional higher order case that $\alpha\in(0,n)\setminus 2\mathbb{N}$, conformally invariant equation \eqref{GPDE} is closely related to the fractional $Q$-curvature problems and the study of fractional conformally covariant Paneitz and GJMS operators and so on ... (cf. e.g. \cite{CC,CG,JLX1} and the references therein). In particular, the special cases $\alpha=1,2,3,4$ concerned in our paper are related to mean curvature, scalar curvature, fractional $Q$-curvature and $Q$-curvature problems respectively.

\smallskip

The quantitative and qualitative properties of solutions to conformally invariant equations \eqref{GPDE} have been extensively studied. In the subcritical order case $\alpha\in(0,n)$, for classification results of positive classical solutions to equation \eqref{GPDE}, please see Gidas, Ni and Nirenberg \cite{GNN1} and Caffarelli, Gidas and Spruck \cite{CGS} for $\alpha=2$, Lin \cite{Lin} for $\alpha=4$, Wei and Xu \cite{WX} for even integer $\alpha\in(0,n)$, Chen, Li and Li \cite{CLL}, Chen, Li and Zhang \cite{CLZ} and Jin, Li and Xiong \cite{JLX} for $0<\alpha<2$, Dai and Qin \cite{DQ} for $\alpha=3$, and Cao, Dai and Qin \cite{CDQ0} for any real number $\alpha\in(0,n)$. In \cite{CLO}, by developing the method of moving planes in integral forms, Chen, Li and Ou classified all the positive $L^{\frac{2n}{n-\alpha}}_{loc}$ solutions to the equivalent integral equation of the PDE \eqref{GPDE} for general $\alpha\in(0,n)$, as a consequence, they obtained the classification results for positive weak solutions to PDE \eqref{GPDE}. In the super-critical order cases, for classification results of positive classical solutions to equation \eqref{GPDE} and related IE with negative exponents, please refer to \cite{DF,Li,N,Xu} and the references therein.

\medskip

The case $n=\alpha$ is called the limiting case (or the critical order case). When $n=\alpha=2$, by using the method of moving planes, Chen and Li \cite{CL} classified all the $C^{2}$ smooth solutions with finite total curvature of the Liouville equation
\begin{equation}\label{0-1}\\\begin{cases}
-\Delta u(x)=e^{2u(x)},  \qquad  x\in\mathbb{R}^{2}, \\ \\
\int_{\mathbb{R}^{2}}e^{2u(x)}\mathrm{d}x<+\infty.
\end{cases}\end{equation}
They proved that there exists some point $x_{0}\in\mathbb{R}^{2}$ and some $\lambda>0$ such that
$$u(x)=\ln\left[\frac{2\lambda}{1+\lambda^{2}|x-x_{0}|^{2}}\right].$$
Equations of type \eqref{0-1} arise from a variety of situations, such as from prescribing Gaussian curvature in geometry and from combustion theory in physics. For conformally invariant systems of nonlinear PDEs of Liouville type, please see \cite{CK}.

\smallskip

Let $g_{\mathbf{S}^{2}}$ be the standard metric on the unit $2$-sphere $\mathbf{S}^{2}$. If we consider the conformal metric $\hat{g}:=e^{2w}g_{\mathbf{S}^{2}}$, then the Gaussian curvature $K_{\hat{g}}$ satisfies the following PDE:
\begin{equation}\label{g1}
  \Delta_{g_{\mathbf{S}^{2}}}w+K_{\hat{g}}e^{2w}=1 \qquad \text{on} \,\, \mathbf{S}^{2},
\end{equation}
where $\Delta_{g_{\mathbf{S}^{2}}}$ denotes the Laplace-Beltrami operator with respect to the standard metric $g_{\mathbf{S}^{2}}$ on the sphere $\mathbf{S}^{2}$. In particular, if we take $K_{\hat{g}}\equiv1$ in \eqref{a1}, then $\hat{g}:=e^{2w}g_{\mathbf{S}^{2}}$ is the pull back of the standard metric $g_{\mathbf{S}^{2}}$ through some conformal transformation $\phi$ (i.e., $\hat{g}$ is isometric to $g_{\mathbf{S}^{2}}$). Through the stereographic projection $\pi$ from $\mathbf{S}^{2}$ to $\mathbb{R}^{2}$, one can see that equation \eqref{0-1} on $\mathbb{R}^{2}$ is equivalent to the equation \eqref{g1} on $\mathbf{S}^{2}$ with $K_{\hat{g}}\equiv1$.

\smallskip

In general, on $(\mathbf{S}^{n},g_{\mathbf{S}^{n}})$, if we change the standard metric $g_{\mathbf{S}^{n}}$ to its conformal metric $\hat{g}:=e^{nw}g_{\mathbf{S}^{n}}$ for some smooth function $w$ on the $n$-sphere $\mathbf{S}^{n}$, since the GJMS operator $P_{n,g}$ is conformlly covariant (c.f. \cite{GJMS,P}), it turns out that there exists some scalar curvature quantity $Q_{n,g}$ of order $n$ such that
\begin{equation}\label{g3}
  -P_{n,g_{\mathbf{S}^{n}}}(w)+Q_{n,\hat{g}}e^{nw}=Q_{n,g_{\mathbf{S}^{n}}} \qquad \text{on} \,\, \mathbf{S}^{n},
\end{equation}
where the explicit formula for $P_{n,g_{\mathbf{S}^{n}}}$ on $\mathbf{S}^{n}$ with general integer $n\in\mathbb{N}^{+}$ is given by (c.f. \cite{Juhl}):
\begin{equation}\label{GJMS-2}
P_{n,g_{\mathbf{S}^{n}}}\left(\cdot\right)=\prod_{k=1}^{\frac{n}{2}}\left[-\Delta_{g_{\mathbf{S}^{n}}}+\left(\frac{n}{2}-k\right)\left(\frac{n}{2}+k-1\right)\right]\left(\cdot\right), \qquad \text{if} \,\, n \,\, \text{is even},
\end{equation}
\begin{equation}\label{GJMS-1}
P_{n,g_{\mathbf{S}^{n}}}\left(\cdot\right)=\left[-\Delta_{g_{\mathbf{S}^{n}}}+\frac{(n-1)^{2}}{4}\right]^{\frac{1}{2}}
\prod_{k=1}^{\frac{n-1}{2}}\left[-\Delta_{g_{\mathbf{S}^{n}}}+\frac{(n-1)^{2}}{4}-k^{2}\right]\left(\cdot\right), \qquad \text{if} \,\, n \,\, \text{is odd}.
\end{equation}
When the metric $\hat{g}$ is isometric to the standard metric $g_{\mathbf{S}^{n}}$, then $Q_{n,\hat{g}}=Q_{n,g_{\mathbf{S}^{n}}}=(n-1)!$, and hence \eqref{g3} becomes
\begin{equation}\label{g4}
  -P_{n,g_{\mathbf{S}^{n}}}(w)+(n-1)!e^{nw}=(n-1)! \qquad \text{on} \,\, \mathbf{S}^{n}.
\end{equation}
We can reformulate the equation \eqref{g4} on $\mathbb{R}^{n}$ by applying the stereographic projection. Let us denote by $\pi: \, \mathbf{S}^{n}\rightarrow \mathbb{R}^{n}$ the stereographic projection which maps the south pole on $\mathbf{S}^{n}$ to $\infty$. That is, for any $\zeta=(\zeta_{1},\cdots,\zeta_{n+1})\in\mathbf{S}^{n}\subset\mathbb{R}^{n+1}$ and $x=\pi(\zeta)=(x_{1},\cdots,x_{n})\in\mathbb{R}^{n}$, then it holds $\zeta_{k}=\frac{2x_{k}}{1+|x|^{2}}$ for $1\leq k\leq n$ and $\zeta_{n+1}=\frac{1-|x|^{2}}{1+|x|^{2}}$. Suppose $w$ is a smooth function on $\mathbf{S}^{n}$, define the function $u(x):=\phi(x)+w(\zeta)$ for any $x\in\mathbb{R}^{n}$, where $\zeta:=\pi^{-1}(x)$ and $\phi(x):=\ln\left[\frac{2}{1+|x|^{2}}\right]=\ln\left|J_{\pi^{-1}}\right|$. Since the GJMS operator $P_{n,g_{\mathbf{S}^{n}}}$ is the pull back under $\pi$ of the operator $(-\Delta)^{\frac{n}{2}}$ on $\mathbb{R}^{n}$ (see Theorem 3.3 in \cite{Branson1}), $w$ satisfies the equation \eqref{g4} on $\mathbf{S}^{n}$ if and only if the function $u$ satisfies
\begin{equation}\label{g5}
  (-\Delta)^{\frac{n}{2}}u=(n-1)!e^{nu} \qquad \text{in} \,\, \mathbb{R}^{n}.
\end{equation}

\smallskip

In \cite{CY}, for general integer $n$, Chang and Yang classified the $C^{n}$ smooth solutions to the critical order equations \eqref{g5} under decay conditions near infinity
\begin{equation}\label{g0}
  u(x)=\ln\left[\frac{2}{1+|x|^2}\right]+w\left(\zeta(x)\right)
\end{equation}
for some smooth function $w$ defined on $\mathbf{S}^n$. When $n=\alpha=4$, Lin \cite{Lin} proved the classification results for all the $C^{4}$ smooth solutions of
\begin{equation}\label{0-2}\\\begin{cases}
\Delta^{2}u(x)=6e^{4u(x)}, \,\,\,\,\,\,\,\, x\in\mathbb{R}^{4}, \\ \\
\int_{\mathbb{R}^{4}}e^{4u(x)}\mathrm{d}x<+\infty, \,\,\,\,\,\, u(x)=o\left(|x|^{2}\right) \,\,\,\, \text{as} \,\,\,\, |x|\rightarrow+\infty.
\end{cases}\end{equation}
When $n=\alpha$ is an even integer, Wei and Xu \cite{WX} classified the $C^{n}$ smooth solutions of \eqref{g5} with finite total curvature $\int_{\mathbb{R}^{n}}e^{nu(x)}\mathrm{d}x<+\infty$ under the assumption $u(x)=o\left(|x|^{2}\right)$ as $|x|\rightarrow+\infty$. Zhu \cite{Z} classified all the classical solutions with finite total curvature of the problem
\begin{equation}\label{0-4}\\\begin{cases}
(-\Delta)^{\frac{3}{2}}u(x)=2e^{3u(x)}, \,\,\,\,\,\,\,\, x\in\mathbb{R}^{3}, \\ \\
\int_{\mathbb{R}^{3}}e^{3u(x)}\mathrm{d}x<+\infty, \,\,\,\,\,\, u(x)=o(|x|^{2}) \,\,\,\, \text{as} \,\,\,\, |x|\rightarrow+\infty.
\end{cases}\end{equation}
The equation \eqref{0-4} can also be regarded as the following system with mixed order:
\begin{equation}\label{0-4s}\\\begin{cases}
(-\Delta)^{\frac{1}{2}}u(x)=2e^{3v(x)}, \,\,\,\,\,\,\,\, x\in\mathbb{R}^{3}, \\
-\Delta v(x)=u(x), \,\,\,\,\,\,\,\, x\in\mathbb{R}^{3}, \\
\int_{\mathbb{R}^{3}}e^{3v(x)}\mathrm{d}x<+\infty, \,\,\,\,\,\, v(x)=o\left(|x|^{2}\right) \,\,\,\, \text{as} \,\,\,\, |x|\rightarrow+\infty.
\end{cases}\end{equation}
Recently, Dai and Qin \cite{DQ} classified $(u,v)\in \left(C^{1,\epsilon}_{loc}(\mathbb{R}^{2})\cap \mathcal{L}_{1}(\mathbb{R}^{2})\right)\times C^{2}(\mathbb{R}^{4})$ arbitrarily small $\epsilon>0$ to the following planar mixed order conformally invariant system:
\begin{equation}\label{PDE-2d}\\\begin{cases}
(-\Delta)^{\frac{1}{2}}u(x)=e^{pv(x)}, \,\,\quad  u(x)\geq0, \qquad x\in\mathbb{R}^{2}, \\
-\Delta v(x)=u^{4}(x), \qquad x\in\mathbb{R}^{2}, \\
\int_{\mathbb{R}^{2}}u^{4}(x)\mathrm{d}x<+\infty, \qquad u(x)=O\left(|x|^{K}\right) \quad \text{as} \,\,\,\, |x|\rightarrow+\infty,
\end{cases}\end{equation}
where $p\in(0,+\infty)$ and $K\gg1$ is arbitrarily large. For more classification results on mixed order conformally invariant systems, please c.f. \cite{Yu,GP,GP1,HN}. For more literatures on the quantitative and qualitative properties of solutions to fractional order or higher order conformally invariant PDE and IE problems, please refer to \cite{BKN,BF,CT,C,CL,CL2,DQ0,Fall,FK,FKT,FLS,JLX1,LZ1,LZ} and the references therein.

\medskip

In this paper, by using the method of moving spheres, we classify all the classical solutions $(u,v,w)$ to the conformally invariant $3$-D system \eqref{a1} with mixed order and critical and exponential growth. One can observe that, if we assume the relationship $u=v^{2}=e^{w}$ in $\mathbb{R}^{3}$ between $u$, $v$ and $w$ in the following three $3$-D conformally invariant equations:
\begin{equation}\label{g6}
  (-\Delta)^{\frac{1}{2}}u=u^{2}, \quad -\Delta v=v^{5} \quad \text{and} \quad (-\Delta)^{\frac{3}{2}}w=e^{3w} \quad \text{in} \,\, \mathbb{R}^{3}
\end{equation}
with the finite total curvature $\int_{\mathbb{R}^{3}}e^{3w(x)}\mathrm{d}x<+\infty$, the resulting system is
\begin{equation}\label{PDE+}\\\begin{cases}
(-\Delta)^{\frac{1}{2}}u(x)=v^{4}(x), \qquad x\in\mathbb{R}^{3}, \\
-\Delta v(x)=e^{\frac{5}{2}w(x)}, \qquad x\in\mathbb{R}^{3}, \\
(-\Delta)^{\frac{3}{2}}w=u^{3}(x), \qquad x\in\mathbb{R}^{3}
\end{cases}\end{equation}
with the finite total curvature $\int_{\mathbb{R}^{3}}u^{3}(x)\mathrm{d}x<+\infty$, i.e., system \eqref{a1} with $p=\frac{5}{2}$. For more literatures on the classification of solutions and Liouville type theorems for various PDE and IE problems via the methods of moving planes or spheres and the method of scaling spheres, please refer to \cite{CGS,CDQ0,CY,CL,CL1,CL0,CL2,CLL,CLO,CLZ,DHL,DLQ,DQ,DQ0,DGZ,DZ,FKT,GNN1,JLX,JLX1,Lin,Li,LZ1,LZ,NN,Pa,Serrin,WX,Xu,Yu,Z} and the references therein.

\smallskip

Our classification result for $3$-D system \eqref{a1} is the following theorem.
\begin{thm} \label{thm0}
Assume $p>0$ and $(u,v,w)$ is a pair of classical solutions to the $3$-D system \eqref{a1} such that $u,v\geq0$, $\int_{\mathbb{R}^{3}}u^{3}(x)\mathrm{d}x<+\infty$ and $w(x)=o(|x|^{2})$ as $|x|\rightarrow+\infty$. Suppose that there exists some $K\gg1$ arbitrarily large such that $u(x)= O(|x|^{K})$ as $|x|\rightarrow+\infty$, then $(u,v,w)$ must take the unique form:
\begin{equation*} u(x)=\frac{2(\frac{5}{p})^{\frac{1}{3}}\mu}{1+\mu^{2}|x-x_{0}|^{2}},\ \,\,\, v(x)=\left(\frac{2(\frac{5}{p})^{\frac{1}{6}}\mu}{1+\mu^{2}|x-x_{0}|^{2}}\right)
^{\frac{1}{2}},\ \,\,\, w(x)=\frac{5}{2p}\ln\left(\frac{(\frac{405}{p})^{\frac{1}{15}}\mu}
{1+\mu^{2}|x-x_{0}|^{2}}\right)
\end{equation*}
for some $\mu>0$ and some $x_{0}\in\mathbb{R}^{3}$, and
$$\int_{\mathbb{R}^{3}}u^{3}(x)\mathrm{d}x=\frac{10\pi^2}{p},\ \quad
\int_{\mathbb{R}^{3}}v^{4}(x)\mathrm{d}x=\lr\frac{5}{p}\rr^{\frac{1}{3}}\frac{4\pi^2}{\mu},\ \quad
\int_{\mathbb{R}^{3}}e^{pw(x)}\mathrm{d}x=\lr\frac{5}{p}\rr^{\frac{1}{6}}\frac{4\pi}{\mu}.$$
\end{thm}
\begin{rem}\label{rem0}
One should note that, in Theorem \ref{thm0}, we do not need any assumption on $v$. The finite total curvature condition $\int_{\mathbb{R}^{3}}u^{3}(x)\mathrm{d}x<+\infty$ and $w(x)=o(|x|^{2})$ at $\infty$ is necessary for the classification result of higher order conformally invariant equations or systems (see e.g. \cite{CY,Lin,WX,Yu,Z}). The assumption ``$u(x)=O\left(|x|^{K}\right)$ at $\infty$ for some $K\gg1$ arbitrarily large" is an extremely mild condition. In fact, the necessary condition for us to define $(-\Delta)^{\frac{1}{2}}u$ is $u\in\mathcal{L}_{1}$ (i.e., $\frac{u}{1+|x|^{4}}\in L^{1}(\mathbb{R}^{3})$), which already indicates that $u$ grows slowly and must has strictly less than linear growth at $\infty$ in the sense of integral. In addition to Theorem \ref{thm0}, by direct calculations, one can also find that
\begin{equation}\label{g13}
  \int_{\mathbb{R}^{3}}v^{6}(x)\mathrm{d}x=2\left(\frac{5}{p}\right)^{\frac{1}{2}}\pi^2 \qquad \text{and} \qquad \int_{\mathbb{R}^{3}}e^{\frac{6}{5}pw(x)}\mathrm{d}x=\left(\frac{405}{p}\right)^{\frac{1}{5}}\frac{\pi^{2}}{4}.
\end{equation}
\end{rem}

\smallskip

We would like to mention some key ideas and main ingredients in our proof of Theorem \ref{thm0}.

\smallskip

One should note that the $3$-D system \eqref{a1} has higher degree of nonlinearity and complexity than a single equation or systems involving only two components $(u,v)$. We need to overcome the mutual restrictions between $(u,v,w)$ and derive the precise asymptotic behaviors at $\infty$ and the integral representation formulae of $u$, $v$ and $w$.

\medskip

First, from the finite total curvature condition $\int_{\mathbb{R}^{3}}u^{3}(x)\mathrm{d}x<+\infty$, we can derive the integral representation formula for $u$ (see Lemmas \ref{le1} and \ref{le2}), that is,
\begin{equation}\label{e1-3d}
  u(x)=\frac{1}{2\pi^{2}}\int_{\mathbb{R}^{3}}\frac{v^{4}(y)}{|x-y|^{2}}\mathrm{d}y, \qquad v(x)=\frac{1}{4\pi}\int_{\mathbb{R}^{3}}\frac{e^{pw(y)}}{|x-y|}\mathrm{d}y
\end{equation}
and hence $|x|^{-2}v^{4}\in L^{1}(\mathbb{R}^{3})$ and $|x|^{-1}e^{pw}\in L^{1}(\mathbb{R}^{3})$. Then, from the finite total curvature condition $\int_{\mathbb{R}^{3}}u^{3}(x)\mathrm{d}x<+\infty$ and $w(x)=o(|x|^{2})$ at $\infty$, we can derive the integral representation formula for $\Delta w$ (see Lemmas \ref{lem1}, \ref{lem2} and \ref{cor13}), that is,
\begin{equation}\label{e2+}
  \Delta w(x)=-\frac{1}{2\pi^{2}}\int_{\mathbb{R}^{3}}\frac{u^{3}(y)}{|x-y|^{2}}\mathrm{d}y.
\end{equation}
From the finite total curvature condition $u^{3}\in L^{1}(\mathbb{R}^{3})$, $w(x)=o(|x|^2)$ at $\infty$ and the assumption $u(x)=O\left(|x|^{K}\right)$ at $\infty$ for some $K\gg1$ arbitrarily large, by using the $\exp^{L}+L\ln L$ inequality from \cite{DQ2} to estimate integrals with logarithmic singularity, we get the asymptotic property $\lim\limits_{|x|\rightarrow+\infty}\frac{k(x)}{\ln|x|}=-\alpha$, where $k(x):=\frac{1}{2\pi^{2}}\int_{\mathbb{R}^{3}}\ln\left[\frac{|y|}{|x-y|}\right]u^{3}(y)\mathrm{d}y$ and $\alpha:=\frac{1}{2\pi^{2}}\int_{\mathbb{R}^{3}}u^{3}(x)\mathrm{d}x$. Based on these properties, by Liouville type results in Corollary 2.10 in \cite{DQ2} deriving from Lemma 3.3 in Lin \cite{Lin} (see Lemma \ref{lemma214}), we can deduce from $|x|^{-1}e^{pw}\in L^{1}(\mathbb{R}^{3})$ the following integral representation formula for $w$:
\begin{equation}\label{e2-3d}
  w(x)=\frac{1}{2\pi^{2}}\int_{\mathbb{R}^{3}}\ln\left[\frac{|y|}{|x-y|}\right]u^{3}(y)\mathrm{d}y+C_{0}
\end{equation}
for some constant $C_{0}\in\mathbb{R}$, and hence the crucial asymptotic behavior $\lim\limits_{|x|\rightarrow+\infty}\frac{w(x)}{\ln|x|}=-\alpha$ (see Lemma \ref{le3}). As a consequence, we derive that the total curvature $\alpha:=\frac{1}{2\pi^{2}}\int_{\mathbb{R}^{3}}u^{3}(x)\mathrm{d}x\geq\frac{2}{p}$. Moreover, if $\alpha>\frac{3}{p}$, we can get $\gamma:=\frac{1}{2\pi}\int_{\mathbb{R}^3}e^{pw(x)}\mathrm{d}x<+\infty$ and the asymptotic behavior $\lim\limits_{|x|\rightarrow+\infty}|x|v(x)=\gamma$, furthermore, $\beta:=\frac{1}{2\pi^2}\int_{\mathbb{R}^{3}}v^{4}(x)\mathrm{d}x<+\infty$ and the asymptotic behavior $\lim\limits_{|x|\rightarrow+\infty}|x|^{2}u(x)=\beta$ (see Corollary \ref{le4}).

\smallskip

Next, by making use of these properties, we can apply the method of moving spheres to the IE system for $(u,v,w)$ consisting of \eqref{e1-3d} and \eqref{e2-3d}. For any $x_{0}\in\mathbb{R}^{3}$, we first prove that, $u_{x_{0},\lambda}\leq u$, $v_{x_{0},\lambda}\leq v$ and $w_{x_{0},\lambda}\leq w$ in $B_{\lambda}(x_{0})\setminus\{x_{0}\}$ for $\lambda\in(0,+\infty)$ sufficiently large if $\alpha\geq\frac{5}{p}$, $u_{x_{0},\lambda}\geq u$, $v_{x_{0},\lambda}\geq v$ and $w_{x_{0},\lambda}\geq w$ in $B_{\lambda}(x_{0})\setminus\{x_{0}\}$ for $\lambda\in(0,+\infty)$ sufficiently small if $\frac{2}{p}<\alpha\leq\frac{5}{p}$ (see \eqref{r1} for definitions of the Kelvin transforms $u_{x_{0},\lambda}$, $v_{x_{0},\lambda}$ and $w_{x_{0},\lambda}$). Then, for any $x_{0}\in\mathbb{R}^{3}$, we show the limiting radius $\lambda_{x_{0}}=0$ if $\alpha>\frac{5}{p}$ and $\lambda_{x_{0}}=+\infty$ if $\alpha<\frac{5}{p}$ (see \eqref{c11} and \eqref{c10} for definitions of the limiting radius $\lambda_{x_{0}}$), and hence derive a contradiction from Lemma 11.2 in \cite{LZ1} (see Lemma \ref{le10}), the finite total curvature condition and the system \eqref{a1}. Finally, we must have $\alpha=\frac{5}{p}$ and hence $u_{x_{0},\lambda}\equiv u$, $v_{x_{0},\lambda}\equiv v$ and $w_{x_{0},\lambda}\equiv w$ in $\mathbb{R}^{3}\setminus\{x_{0}\}$ for any $x_{0}\in\mathbb{R}^{3}$ and some $\lambda>0$ depending on $x_{0}$. As a consequence, Lemma 11.1 in \cite{LZ1} (see Lemma \ref{le10}) and the asymptotic properties of $(u,v,w)$ yield the desired classification results in Theorem \ref{thm0}.

\subsection{Conformally invariant systems with mixed order and critical and exponential growth in $\mathbb{R}^{4}$.}

We are also concerned with the following conformally invariant system with mixed order and critical and exponential growth in $\mathbb{R}^{4}$:
\begin{align}\label{eq41}
\begin{cases}
\ (-\Delta)^{\frac{1}{2}} u=e^{pw} ,&x\in \mathbb{R}^{4},\\
\ -\Delta v=u^2 ,&x\in \mathbb{R}^{4},\\
\ (-\Delta)^{2} w=v^{4} ,&x\in \mathbb{R}^{4},
\end{cases}
\end{align}
where $p>0$, $u,v\geq0$, $w(x)=o(|x|^{2})$ as $|x|\rightarrow+\infty$ and $v$ satisfies the finite total curvature condition $\int_{\mathbb{R}^{4}}v^{4}(x)\mathrm{d}x<+\infty$. We assume $(u,v,w)$ is a pair of classical solution to the $4$-D system \eqref{eq41} in the sense that $u\in C^{1,\epsilon}_{loc}(\mathbb{R}^{4})\cap \mathcal{L}_{1}(\mathbb{R}^{4})$ with arbitrarily small $\epsilon>0$, $v\in C^{2}(\mathbb{R}^{4})$ and $w\in C^{4}(\mathbb{R}^{4})$.

\medskip

Recently, Yu \cite{Yu} classified $(u,v)\in C^{2}(\mathbb{R}^{4})\times C^{4}(\mathbb{R}^{4})$ to the following conformally invariant system
\begin{equation}\label{Yu-s}\\\begin{cases}
-\Delta u(x)=e^{3v(x)}, \,\,\quad  u(x)\geq0, \qquad  x\in\mathbb{R}^{4}, \\
\Delta^{2}v(x)=u^{4}(x), \qquad\,\,\,  x\in\mathbb{R}^{4}, \\
\int_{\mathbb{R}^{4}}u^{4}(x)\mathrm{d}x<+\infty, \quad \int_{\mathbb{R}^{4}}e^{3v(x)}\mathrm{d}x<+\infty, \,\,\,\,\,\, v(x)=o\left(|x|^{2}\right) \,\,\,\, \text{as} \,\,\,\, |x|\rightarrow+\infty.
\end{cases}\end{equation}
For more classification results on mixed order conformally invariant systems, please c.f. \cite{DQ2,GP,GP1,HN}.

\medskip

The $4$-D system \eqref{eq41} has higher degree of nonlinearity and complexity than a single equation or systems involving only two components $(u,v)$. By using the method of moving spheres, we classify all the classical solutions $(u,v,w)$ to the conformally invariant $4$-D system \eqref{eq41} with mixed order and critical and exponential growth. One can observe that, if we assume the relationship $u^{2}=v^{3}=e^{3w}$ in $\mathbb{R}^{4}$ between $u$, $v$ and $w$ in the following three $4$-D conformally invariant equations:
\begin{equation}\label{g6}
  (-\Delta)^{\frac{1}{2}}u=u^{\frac{5}{3}}, \quad -\Delta v=v^{3} \quad \text{and} \quad (-\Delta)^{2}w=e^{4w} \quad \text{in} \,\, \mathbb{R}^{4}
\end{equation}
with the finite total curvature $\int_{\mathbb{R}^{4}}e^{4w(x)}\mathrm{d}x<+\infty$, the resulting system is the $4$-D system \eqref{eq41} with $p=\frac{5}{2}$ satisfying the finite total curvature condition $\int_{\mathbb{R}^{4}}v^{4}(x)\mathrm{d}x<+\infty$. 

\medskip

Our classification result for $4$-D system \eqref{eq41} is the following theorem.
\begin{thm} \label{thm1}
Assume $p>0$ and $(u,v,w)$ is a pair of classical solutions to the $4$-D system \eqref{eq41} such that $u,v\geq0$, $\int_{\mathbb{R}^{4}}v^{4}(x)\mathrm{d}x<+\infty$ and $w(x)=o(|x|^{2})$ as $|x|\rightarrow+\infty$. Suppose that \\
either $(H_1)$ $u=O(|x|^{K})$ or $v=O(|x|^{K})$ as $|x|\rightarrow+\infty$ for some $K\gg1$ arbitrarily large, \\
or $(H_2)$ $\int_{\mathbb{R}^{4}}e^{\Lambda pw(y)}\mathrm{d}y<+\infty$ for some $\Lambda>0$, \\
then $(u,v,w)$ must take the unique form
\begin{align*}\label{eq42}
u(x)=\left(\frac{2\left(\frac{160}{p}\right)^{\frac{1}{12}}\mu}{1+\mu^{2}|x-x_{0}|^{2}}\right)^
\frac{3}{2}, \,\,\,\,\,\, v(x)=\frac{2\left(\frac{30}{p}\right)^{\frac{1}{4}}\mu}{1+\mu^{2}|x-x_{0}|^{2}}, \,\,\,\,\,\, w(x)=\frac{5}{2p}\ln\left(\frac{4\left(\frac{2^{9}3^85}{p}\right)^{\frac{1}{20}}\mu}
{1+\mu^{2}|x-x_{0}|^{2}}\right),
\end{align*}
where $x_{0}$ is a fixed point in $ \mathbb{R}^{4}$ and $\mu>0$ and
$$\int_{\mathbb{R}^{4}}u^{2}(x)\mathrm{d}x=\lr\frac{10}{p}\rr^{\frac{1}{4}}
\frac{8\pi^2}{\mu}, \,\,\,\,\,\,
\int_{\mathbb{R}^{4}}v^{4}(x)\mathrm{d}x=\frac{40\pi^2}{p}, \,\,\,\,\,\,
\int_{\mathbb{R}^{4}}e^{pw(x)}\mathrm{d}x=\lr\frac{10}{p}\rr^{\frac{1}{8}}
\frac{16\pi^2}{\mu^{\frac{3}{2}}}.$$
\end{thm}
\begin{rem}\label{rem1}
One should note that, in Theorem \ref{thm1}, if we assume $v=O(|x|^{K})$ as $|x|\rightarrow+\infty$ for some $K\gg1$ arbitrarily large in the hypothesis $(H_1)$ or assume hypothesis $(H_2)$, then no assumption on $u$ will be needed. The finite total curvature condition $\int_{\mathbb{R}^{4}}v^{4}(x)\mathrm{d}x<+\infty$ and $w(x)=o(|x|^{2})$ at $\infty$ is necessary for the classification result of higher order conformally invariant equations or systems (see e.g. \cite{CY,Lin,WX,Yu,Z}). The assumption ``either $(H_1)$ or $(H_2)$" in Theorem \ref{thm1} is much weaker than ``$v=O(|x|^{K})$ as $|x|\rightarrow+\infty$ for some $K\gg1$ arbitrarily large". The hypothesis ``$u=O(|x|^{K})$ or $v=O(|x|^{K})$ as $|x|\rightarrow+\infty$ for some $K\gg1$ arbitrarily large" in $(H_1)$ is an extremely mild condition. In fact, the necessary condition for us to define $(-\Delta)^{\frac{1}{2}}u$ is $u\in\mathcal{L}_{1}$ (i.e., $\frac{u}{1+|x|^{5}}\in L^{1}(\mathbb{R}^{4})$), which already indicates that $u$ grows slowly and must has strictly less than linear growth at $\infty$ in the sense of integral. In addition to Theorem \ref{thm1}, by direct calculations, one can also find that
\begin{equation}\label{g13}
  \int_{\mathbb{R}^{4}}u^{\frac{8}{3}}(x)\mathrm{d}x=\left(\frac{20}{p}\right)^{\frac{1}{3}}\frac{8\pi^2}{3} \qquad \text{and} \qquad \int_{\mathbb{R}^{4}}e^{\frac{8}{5}pw(x)}\mathrm{d}x=128\left(\frac{2160}{p}\right)^{\frac{1}{5}}\pi^{2}.
\end{equation}
\end{rem}

\smallskip

We would like to mention some key ideas and main ingredients in our proof of Theorem \ref{thm1}.

\smallskip

First, we need to overcome the mutual restrictions between $(u,v,w)$ and derive the precise asymptotic behaviors at $\infty$ and the integral representation formulae of $u$, $v$ and $w$. To this end, from the finite total curvature condition $\int_{\mathbb{R}^{4}}v^{4}(x)\mathrm{d}x<+\infty$, we can derive the integral representation formula for $u$ (see Lemmas \ref{lemm212} and \ref{le51}), that is,
\begin{equation}\label{e1-4d}
  u(x)=\frac{1}{4\pi^{2}}\int_{\mathbb{R}^{4}}\frac{e^{pw(y)}}{|x-y|^{3}}\mathrm{d}y, \qquad v(x)=\frac{1}{4\pi^{2}}\int_{\mathbb{R}^{4}}\frac{u^2(y)}{|x-y|^{2}}\mathrm{d}y,
\end{equation}
and hence $|x|^{-2}u^{2}\in L^{1}(\mathbb{R}^{4})$ and $|x|^{-3}e^{pw}\in L^{1}(\mathbb{R}^{4})$. Then, from the finite total curvature condition $\int_{\mathbb{R}^{4}}v^{4}(x)\mathrm{d}x<+\infty$ and $w(x)=o(|x|^{2})$ at $\infty$, we can derive the integral representation formula for $\Delta w$ (see Lemmas \ref{lem59}, \ref{lem510} and \ref{lem45}), that is,
\begin{equation}\label{e2+4}
  \Delta w(x)=-\frac{1}{4\pi^{2}}\int_{\mathbb{R}^{4}}\frac{v^{4}(y)}{|x-y|^{2}}\mathrm{d}y.
\end{equation}
Then, the main key ingredient is, by exploiting a bi-harmonic version (in Lemma 2.3 of \cite{Lin}) of the Brezis-Merle's estimate in \cite{BF} and elliptic estimates, we can deduce from either the assumption $(H_{1})$ or the assumption $(H_{2})$ that $v(x)=O\left(|x|^{2K}\right)$ at $\infty$ for some $K\gg1$ arbitrarily large (see Lemma \ref{lem46}). Consequently, by using the $\exp^{L}+L\ln L$ inequality from \cite{DQ2}, we can get the asymptotic property $\lim\limits_{|x|\rightarrow+\infty}\frac{q(x)}{\ln|x|}=\alpha$, where $q(x):=\frac{1}{8\pi^{2}}\int_{\mathbb{R}^{4}}\ln\left[\frac{|x-y|}{|y|}\right]v^{4}(y)\mathrm{d}y$ and $\alpha:=\frac{1}{8\pi^{2}}\int_{\mathbb{R}^{4}}v^{4}(x)\mathrm{d}x$. Based on these properties, by Liouville type results in Corollary 2.10 in \cite{DQ2} deriving from Lemma 3.3 in Lin \cite{Lin} (see Lemma \ref{lemma214}), we can deduce from $|x|^{-3}e^{pw}\in L^{1}(\mathbb{R}^{4})$ the following integral representation formula for $w$:
\begin{equation}\label{e2-4d}
  w(x)=\frac{1}{8\pi^{2}}\int_{\mathbb{R}^{4}}\ln\left[\frac{|y|}{|x-y|}\right]v^{4}(y)\mathrm{d}y+C_{0}
\end{equation}
for some constant $C_{0}\in\mathbb{R}$, and hence the crucial asymptotic behavior $\lim\limits_{|x|\rightarrow+\infty}\frac{w(x)}{\ln|x|}=-\alpha$ (see Lemma \ref{lem515}). As a consequence, we derive that the total curvature $\alpha:=\frac{1}{8\pi^{2}}\int_{\mathbb{R}^{4}}v^{4}(x)\mathrm{d}x\geq\frac{1}{p}$. Moreover, if $\alpha>\frac{4}{p}$, we can get $\gamma:=\frac{1}{4\pi^{2}}\int_{\mathbb{R}^{4}}e^{pw(x)}\mathrm{d}x<+\infty$ and the asymptotic behavior $\lim\limits_{|x|\rightarrow+\infty}|x|^{3}u(x)=\gamma$, furthermore, $\beta:=\frac{1}{4\pi^2}\int_{\mathbb{R}^{4}}u^{2}(x)\mathrm{d}x<+\infty$ and the asymptotic behavior $\lim\limits_{|x|\rightarrow+\infty}|x|^{2}v(x)=\beta$ (see Corollary \ref{lem58}). Next, similar to the $3$-D system \eqref{a1}, by making use of these properties, we can apply the method of moving spheres to the IE system for $(u,v,w)$ consisting of \eqref{e1-4d} and \eqref{e2-4d}. The key point is to show $\alpha=\frac{5}{p}$ through contradiction arguments combining with Lemma 11.2 in \cite{LZ1} (see Lemma \ref{le10}), the finite total curvature condition and the $4$-D system \eqref{eq41}.

\smallskip

The rest of our paper is organized as follows. In section 2, we will carry out our proof of Theorem \ref{thm0}. Section 3 is devoted to proving our Theorem \ref{thm1}.

\smallskip

In what follows, we will use $C$ to denote a general positive constant that may depend on $p$ and $\Lambda$, and whose value may differ from line to line.

\section{Proof of Theorem \ref{thm0}}

In this section, we classify the classical solutions $(u,v,w)$ to the $3$-D system \eqref{a1} and hence carry out our proof of Theorem \ref{thm0}.

\smallskip

Suppose that $(u,v,w)$ is a pair of classical solution to the system \eqref{a1} with $u,v\geq0$.

\smallskip

We first prove the following integral representation formula for $u$.
\begin{lem} \label{le1}
Suppose that $\int_{\mathbb{R}^{3}}u^{3}(x)\mathrm{d}x<+\infty$. Then we have, for any $x\in \mathbb{R}^{3}$,
\begin{equation}\label{l1}
u(x)=\frac{1}{2\pi^2}\int_{\mathbb{R}^{3}}\frac{v^{4}(y)}{|x-y|^{2}}\mathrm{d}y.
\end{equation}
Furthermore, $u >0$ in $\mathbb{R}^{3}$, $u(x) \geq \frac{C}{|x|^{2}}$
for some constant $C> 0$ and $|x|$ large enough, and
\begin{equation}\label{eq22a}
\int_{\mathbb{R}^{3}}\frac{v^{4}(x)}{|x|^{2}}\mathrm{d}x<+\infty.
\end{equation}
\end{lem}

\begin{proof}[{\sl Proof of Lemma~\ref{le1}}]
For arbitrary $R > 0$, we define
\begin{equation}\label{l3}
\ u_{R}(x)=\int_{B_{R}(0)}G_{R}(x,y)v^{4}(y)\mathrm{d}y,
\end{equation}
where the Green function $G_{R}(x,y)$ of $(-\Delta)^{\frac{1}{2}}$ (see \cite{K}) on the ball $B_{R}(0)$ satisfies
$$
\begin{cases}
\ (-\Delta)^{\frac{1}{2}}G_{R}(x,y)=\delta(x-y),& x,y\in B_{R}(0),\\
\ \qquad\quad G_{R}(x,y)=0,& x\ \text{or}\ y\ \in\mathbb{R}^{3}\backslash B_{R}(0)
\end{cases}
$$
with
$$ G_{R}(x,y)=\frac{C}{|x-y|^{2}}\int_{0}^{\frac{t_{R}}{d_{R}}}\frac{1}{(1+z)^{\frac{3}{2}}\sqrt{z}}\mathrm{d}z,  \ x,y\in B_{R}(0),$$
 and
$$d_{R}=\frac{|x-y|^{2}}{R^{2}},\,\,\,\, t_{R}=\left(1-\frac{|x|^{2}}{R^{2}}\right)\left(1-\frac{|y|^{2}}{R^{2}}\right),\,\,\,\,
C=\frac{1}{2\pi^2}\left(\int_{0}^{+\infty}\frac{1}{(1+z)^{\frac{3}{2}}\sqrt{z}}\mathrm{d}z\right)^{-1}.$$
From this, we can deduce that $ u_{R}\in C(\mathbb{R}^{3})\cap \mathcal{L}_{1}(\mathbb{R}^{3})\cap C_{loc}^{1,\var}(B_{R}(0))$ solves
\begin{equation*}
\begin{cases}
\ (-\Delta)^{\frac{1}{2}} u_{R}(x)=v^{4}(x) ,&x\in B_{R}(0),\\
\ u_{R}(x)=0,&x\in\mathbb{R}^{3}\backslash B_{R}(0).
\end{cases}
\end{equation*}
Let $\varphi_{R}(x)=u(x)-u_{R}(x)\in C(\mathbb{R}^{3})\cap \mathcal{L}_{1}(\mathbb{R}^{3})\cap C_{loc}^{1,\var}(B_{R}(0)).$ A direct computation derives that
\begin{equation*}
\begin{cases}
\ (-\Delta)^{\frac{1}{2}} \varphi_{R}(x)=0 ,&x\in B_{R}(0),\\
\ \varphi_{R}(x)\geq0,&x\in\mathbb{R}^{3}\backslash B_{R}(0).
\end{cases}
\end{equation*}

Now we need the following maximum principle for fractional Laplacians.
\begin{lem}[Maximum principle, \textbf{\cite{CLL,S}}]\label{lem52}
Let $\Omega$ be a bounded domain in $\mathbb{R}^{n}$, $n\geq2$ and $0<\alpha<2$.  Assume that $u\in C_{loc}^{[\alpha],\{\alpha\}+\var}(\Omega)\cap \mathcal{L}_{\alpha}(\Omega)$ with arbitrarily small $\varepsilon>0$ and is lower semi-continuous on $\overline{\Omega}$. If
$$
\begin{cases}
\ (-\Delta)^{\frac{\alpha}{2}} u(x)\geq0 & \text{in}\ \Omega,\\
\ \qquad u\geq0 & \text{in}\ \mathbb{R}^{n}\backslash\Omega,
\end{cases}
$$
then
$$u\geq0 \qquad \text{in}\,\, \R^n.$$
Moreover, if $u=0$ at some point in $\Omega$, then $u=0$ a.e. in $\R^n$. These conclusions also apply to the unbounded domain $\Omega$ if we further assume
$$\liminf_{|x|\to\infty}u(x)\geq0.$$
\end{lem}
By the maximum principle of the fractional Laplacian in Lemma \ref{lem52}, it follows that
\begin{equation}\label{l6}
\varphi_{R}(x)\geq0,\ x\in\mathbb{R}^{3}.
\end{equation}
Let $R\rightarrow+\infty$ in \eqref{l6}, for each fixed $x\in\mathbb{R}^{3}$, we obtain
\begin{equation}\label{l7}
u(x)\geq \frac{1}{2\pi^2}\int_{\mathbb{R}^{3}}\frac{v^{4}(y)}{|x-y|^{2}}\mathrm{d}y:=\tilde{u}(x).
\end{equation}
Taking $x=0$ in \eqref{l7}, we find $v$ satisfies the following integrability
\begin{equation*}
\int_{\mathbb{R}^{3}}\frac{v^{4}(y)}{|y|^{2}}\mathrm{d}y\leq 2\pi^2 u(0)<+\infty.
\end{equation*}
Note that $\tilde{u}\in \mathcal{L}_{1}(\mathbb{R}^{3})\cap C_{loc}^{1,\var}(\mathbb{R}^{3})$ is a solution of
\begin{equation}\label{l9}
(-\Delta)^{\frac{1}{2}} \tilde{u}(x)=v^{4}(x),\ x\in\mathbb{R}^{3}.
\end{equation}
Let $\varphi(x)=u(x)-\tilde{u}(x)\in \mathcal{L}_{1}(\mathbb{R}^{3})\cap C_{loc}^{1,\var}(\mathbb{R}^{3})$, from system \eqref{a1}, \eqref{l7} and \eqref{l9}, it holds
\begin{equation*}
\begin{cases}
\ (-\Delta)^{\frac{1}{2}} \varphi(x)=0 ,&x\in \mathbb{R}^{3},\\
\ \varphi(x)\geq0 ,&x\in \mathbb{R}^{3}.
\end{cases}
\end{equation*}

Next, we need Liouville theorem for $\alpha$-harmonic function in $\R^n$ with $n\geq2$.
\begin{lem}[Liouville theorem, \cite{BKN}]\label{lem53}
Let $0<\alpha<2$ and $n\geq2$. Assume that $u$ is a strong solution of
$$
\begin{cases}
\ (-\Delta)^{\frac{\alpha}{2}} u(x)=0,& x\in\ \mathbb{R}^{n},\\
\ \quad\quad u\geq0,& x\in\ \mathbb{R}^{n}.
\end{cases}
$$
Then $u\equiv C\geq0$.
\end{lem}
By the Liouville theorem of the fractional Laplacion in Lemma \ref{lem53}, there exists nonnegative constant $C$, such that $\varphi(x)\equiv C$. Thus, we have proved that
\begin{equation}\label{l11}
u(x)= \frac{1}{2\pi^2}\int_{\mathbb{R}^{3}}\frac{v^{4}(y)}{|x-y|^{2}}\mathrm{d}y+C>C\geq0.
\end{equation}
Moreover, by the finite total curvature condition $\int_{\mathbb{R}^{3}}u^{3}(x)\mathrm{d}x<+\infty$, we have
\begin{equation*}
\int_{\mathbb{R}^{3}}C^{3}\mathrm{d}x<\int_{\mathbb{R}^{3}}u^{3}(x)\mathrm{d}x<+\infty,
\end{equation*}
which indicates that $C=0$. Thus, we arrived at the integral representation formula \eqref{l1} of $u$. From \eqref{l1}, we get, for any $|x|$ sufficiently
large,
\begin{equation*}
\begin{aligned}
u(x)&= \frac{1}{2\pi^2}\int_{\mathbb{R}^{3}}\frac{v^{4}(y)}{|x-y|^{2}}\mathrm{d}y\geq\frac{1}{2\pi^2}\int_{1\leq|y|<\frac{|x|}{2}}\frac{v^{4}(y)}{|x-y|^{2}}\mathrm{d}y\\
&\geq \frac{2}{9\pi^2|x|^{2}}\int_{1\leq|y|<\frac{|x|}{2}}\frac{v^{4}(y)}{|y|^{2}}\mathrm{d}y\geq \frac{1}{9\pi^2|x|^{2}}\int_{|y|\geq1}\frac{v^{4}(y)}{|y|^{2}}\mathrm{d}y\\
&:=\frac{C}{|x|^{2}}.
\end{aligned}
\end{equation*}
This finishes our proof of Lemma \ref{le1}.
\end{proof}

Then we are going to prove the integral representation formula of $v$.
\begin{lem}\label{le2}
Suppose that $\int_{\mathbb{R}^{3}}u^{3}(x)\mathrm{d}x<+\infty$. Then we have, for any $x\in \mathbb{R}^{3}$,
\begin{equation}\label{lll1}
v(x)=\frac{1}{4\pi}\int_{\mathbb{R}^{3}}\frac{e^{pw(y)}}{|x-y|}\mathrm{d}y.
\end{equation}
Furthermore, $v >0$ in $\mathbb{R}^{3}$, $v(x) \geq \frac{C}{|x|}$
for some constant $C > 0$ and $|x|$ large enough, and
\begin{equation}\label{29}
\int_{\mathbb{R}^{3}}\frac{e^{pw(x)}}{|x|}\mathrm{d}x<+\infty.
\end{equation}
\end{lem}

\begin{proof}[{\sl Proof of Lemma~\ref{le2}}]
By maximum principle and Liouville theorem for $-\Delta$, using similar arguments as \eqref{l3}--\eqref{l11} in the proof of Lemma \ref{le1}, we can derive that
\begin{equation}\label{z1}
v(x)=\frac{1}{4\pi}\int_{\mathbb{R}^{3}}\frac{e^{pw(y)}}{|x-y|}\mathrm{d}y+C>C\geq0.
\end{equation}
From $\int_{\mathbb{R}^{3}}\frac{v^{4}(x)}{|x|^{2}}\mathrm{d}x<+\infty$ in \eqref{eq22a}, we can infer immediately that $C=0$. Therefore, $v$ satisfies the integral equation \eqref{lll1}.
Taking $x=0$ in \eqref{z1}, we obtain
$$\int_{\mathbb{R}^{3}}\frac{e^{pw(y)}}{|y|}\mathrm{d}y\leq4\pi v(0)<+\infty.$$
Furthermore, from the integral equation \eqref{lll1} of $v$, we obtain, for any $|x|$ sufficiently large,
\begin{equation*}
\begin{aligned}
v(x)&= \frac{1}{4\pi}\int_{\mathbb{R}^{3}}\frac{e^{pw(y)}}{|x-y|}\mathrm{d}y \geq\frac{1}{4\pi}\int_{1\leq|y|<\frac{|x|}{2}}\frac{e^{pw(y)}}{|x-y|}\mathrm{d}y\\
&\geq \frac{1}{9\pi|x|}\int_{1\leq|y|<\frac{|x|}{2}}\frac{e^{pw(y)}}{|y|}\mathrm{d}y\geq \frac{1}{18\pi|x|}\int_{|y|\geq1}\frac{e^{pw(y)}}{|y|}\mathrm{d}y\\
&:=\frac{C}{|x|},
\end{aligned}
\end{equation*}
which finishes the proof of Lemma \ref{le2}.
\end{proof}

Now, we are to show the integral representation formula of $w$. To this end, define
\begin{equation}\label{zz7}
\alpha:=\frac{1}{2\pi^2}\int_{\mathbb{R}^{3}}u^{3}(x)\mathrm{d}x
\end{equation}
and
\begin{equation}\label{dd1}
k(x):=\frac{1}{2\pi^2}\int_{\mathbb{R}^{3}}\ln\left(\frac{|y|}{|x-y|}\right)u^{3}(y)\mathrm{d}y.
\end{equation}
Since $\int_{\mathbb{R}^{3}}u^{3}(x)\mathrm{d}x<+\infty$, then it is easy to see that $\alpha$ and $k(x)$ are well-defined. Furthermore, $k(x)$ satisfies
\begin{equation}\label{214}
-\Delta k(x)=\frac{1}{2\pi^2}\int_{\mathbb{R}^{3}}\frac{u^{3}(y)}{|x-y|^{2}}\mathrm{d}y
\end{equation}
and
\begin{equation}\label{dd2}
(-\Delta)^{\frac{3}{2}} k(x)=u^3(x).
\end{equation}

We can derive the following upper bound estimate on $k(x)$.
\begin{lem}\label{lem1}
Suppose $\int_{\mathbb{R}^{3}}u^{3}(x)\mathrm{d}x<+\infty$. Then there is a constant $C$ such that
\begin{equation}\label{dd3}
-k(x)\leq\alpha\ln|x|+C.
\end{equation}
\end{lem}
\begin{proof}[{\sl Proof of Lemma~\ref{lem1}}]
We divide the integral region $\R^3$ into two parts: $\R^3=A_1\cup A_2$, where
$$A_1=\left\{y||y-x|\leq\frac{|x|}{2}\right\},\ A_2=\left\{y||y-x|\geq\frac{|x|}{2}\right\}.$$
For $y\in A_1$, we have $|y|\geq|x|-|x-y|\geq\frac{|x|}{2}\geq|x-y|$, which indicates
$$\ln\left(\frac{|x-y|}{|y|}\right)\leq0.$$

For $y\in A_2\cap\{|y|\geq2\}$, we have $|x-y|\leq|x|+|y|\leq|x||y|$; for $y\in A_2\cap\{|y|\leq2\}$, we have $\ln|x-y|\leq\ln|x|.$ Thus, we get
\begin{align*}
-k(x)&\leq\frac{1}{2\pi^2}\int_{A_2}\ln\left(\frac{|x-y|}{|y|}\right)u^3(y)\mathrm{d}y\\
&\leq\frac{1}{2\pi^2}\int_{A_2\cap\{|y|\geq2\}}\ln\left(\frac{|x-y|}{|y|}\right)u^3(y)\mathrm{d}y
+\frac{1}{2\pi^2}\int_{A_2\cap\{|y|\leq2\}}\ln\left(\frac{|x-y|}{|y|}\right)u^3(y)\mathrm{d}y\\
&\leq\frac{\ln|x|}{2\pi^2}\int_{A_2\cap\{|y|\geq2\}}u^3(y)\mathrm{d}y
+C\\&\leq\alpha\ln|x|+C.
\end{align*}
This finishes the proof of Lemma \ref{lem1}.
\end{proof}

Next, we will obtain the integral representation formula for $\Delta w(x)$.
\begin{lem} \label{lem2}
Suppose $\int_{\mathbb{R}^{3}}u^{3}(x)\mathrm{d}x<+\infty$, then $\Delta w(x)$ can be represented by
\begin{equation}\label{q13d}
\Delta w(x)=-\frac{1}{2\pi^2}\int_{\mathbb{R}^{3}}\frac{u^{3}(y)}{|x-y|^{2}}\mathrm{d}y-C
\end{equation}
for some constant $C\geq0$.
\end{lem}
\begin{proof}[{\sl Proof of Lemma~\ref{lem2}}]
Let $m(x)=w(x)-k(x)$, from the equation \eqref{a1} and \eqref{dd2}, we have
$$(-\Delta)^{\frac{3}{2}} m(x)=0,$$ which implies
$(-\Delta)^{2} m(x)=0,$ i.e., $\Delta m$ is a harmonic function in $\R^3$. From the mean value theorem of harmonic function, for any $x_0\in\R^3$ and $r>0$, we derive
\begin{align}\label{dd219}
\Delta m(x_0)&=\frac{3}{4\pi r^3}\int_{|y-x_0|\leq r}\Delta m(y)\mathrm{d}y\\
&=\frac{3}{4\pi r^3}\int_{|y-x_0|= r}\frac{\partial m}{\partial r}\mathrm{d}\sigma\nonumber.
\end{align}
Integrating the equation \eqref{dd219} from $0$ to $r$, we obtain
\begin{align*}
\frac{r^2}{6}\Delta m(x_0)&=\bbint_{|y-x_0|= r}m(y)\mathrm{d}\sigma-m(x_0).
\end{align*}
Using the Jensen's inequality, we have
\begin{align}\label{dd221}
e^{\frac{pr^2}{6}\Delta m(x_0)}&=e^{-pm(x_0)}e^{\bbint_{|y-x_0|= r}pm(y)\mathrm{d}\sigma}\\
&\leq e^{-pm(x_0)}\bbint_{|y-x_0|= r}e^{pm(y)}\mathrm{d}\sigma\nonumber.
\end{align}
By the inequality \eqref{dd3}, we have $m(x)=w(x)-k(x)\leq w(x)+\alpha\ln|x|+C$. Combining this with \eqref{dd221} and the formula \eqref{29} of $|x|^{-1}e^{pw(x)}\in L^1(\R^3)$, we get
$$r^{1-p\alpha}e^{\frac{pr^2}{6}\Delta m(x_0)}\in L_{r}^1\left([1,+\infty)\right).$$
Thus $\Delta m(x_0)\leq0$ for all $x_0\in\R^3$. By Liouville Theorem, $\Delta m(x)=-C$ in $\R^3$ for some constant $C\geq0$, which combines with \eqref{214} yield \eqref{dd3}. This completes the proof of Lemma \ref{lem2}.
\end{proof}

Now we need to assume $w(x)=o(|x|^{2})$ at $\infty$, and we can get the precise integral representation formula for $\Delta w(x)$.
\begin{lem}\label{cor13}
Assume $\int_{\mathbb{R}^{3}}u^{3}(x)\mathrm{d}x<+\infty$ and $w(x)=o(|x|^{2})$ at $\infty$. Then we have,
\begin{equation}\label{dd220}
\Delta w(x)=-\frac{1}{2\pi^2}\int_{\mathbb{R}^{3}}\frac{u^{3}(y)}{|x-y|^{2}}\mathrm{d}y.
\end{equation}
\end{lem}
\begin{proof}[{\sl Proof of Lemma~\ref{cor13}}]
First, from the Lemma \ref{lem2}, we have
\begin{equation*}
\Delta w(x)=-\frac{1}{2\pi^2}\int_{\mathbb{R}^{3}}\frac{u^{3}(y)}{|x-y|^{2}}\mathrm{d}y-C.
\end{equation*}
If $C=0$, we directly get \eqref{dd220}. If $C>0$, take $\var\in(0,\frac{C}{12})$ and $R_{0}$ large enough such that
\begin{align*}
\Delta w(x)\leq-C<0 \,\,\,\, \text{in} \ \R^3 \qquad  \text{and} \qquad
w(y)+\var|y|^2\geq\frac{\var}{2}|y|^2\geq0 \,\,\,\, \text{for}\  |y|\geq R_0.
\end{align*}
Define
$$n(y)=w(y)+\var|y|^2+B(|y|^{-1}-R_0^{-1}).$$
Thus
\begin{align}\label{ddz221}
\Delta n(y)=\Delta w(y)+6\var<-\frac{C}{2}<0
\end{align}
for $|y|\geq R_0$. Furthermore, we get
\begin{align*}
&\lim_{|y|\to\infty}n(y)=+\infty \quad\text{for any}\ B>0,\\
&\lim_{B\to+\infty}n(y)=-\infty \quad\text{for any}\ y\in\R^3\backslash\overline{B_{R_0}(0)}.
\end{align*}
Thus, we can take $B$ sufficiently large such that $\inf\limits_{|y|\geq R_0}n(y)$ is obtained by some $y_0\in\R^3$ with $|y_0|>R_0$. By the $maximum\ principle$ to \eqref{ddz221}, we derive a contradiction which implies $C=0$. This finishes the proof of Lemma \ref{cor13}.
\end{proof}

From Lemma \ref{cor13}, we get the following integral representation formula and asymptotic property for $w$.
\begin{lem} \label{le3}
Assume $\int_{\mathbb{R}^{3}}u^{3}(x)\mathrm{d}x<+\infty$, $u=O(|x|^{K})$ at $\infty$ for some $K\gg1$ arbitrarily large and $w(x)=o(|x|^{2})$ at $\infty$. Then we have
\begin{equation}\label{z4}
w(x)=\frac{1}{2\pi^2}\int_{\mathbb{R}^{3}}\ln\left(\frac{|y|}{|x-y|}\right)u^{3}(y)\mathrm{d}y+C_0,
\end{equation}
where $C_0\in \mathbb{R}$ is a constant. Moreover,
\begin{equation}\label{d4}
\lim_{|x|\rightarrow+\infty}\frac{w(x)}{\ln|x|}=-\alpha.
\end{equation}
\end{lem}
\begin{proof}[{\sl Proof of Lemma~\ref{le3}}]
We will first prove the following asymptotic property:
\begin{equation}\label{d5}
\lim_{|x|\rightarrow+\infty}\frac{k(x)}{\ln|x|}=-\alpha.
\end{equation}
To this end, we only need to show that
\begin{equation*}
\lim_{|x|\rightarrow+\infty}\int_{\mathbb{R}^{3}}\frac{\ln|x-y|-\ln|y|-\ln|x|}{\ln|x|}u^{3}(y)\mathrm{d}y=0.
\end{equation*}

Now we need the $\exp^{L}+L\ln L$ inequality established by Dai and Qin in Lemma 2.6 of \cite{DQ2}.
\begin{lem}[\textbf{\cite{DQ2}}]\label{lem212}
Suppose $n\geq1$ and $\Omega\subseteq\R^n$ is a bounded or unbounded domain. Assume $f\in\exp^L(\Omega)$ and $g\in L\ln L(\Omega)$, then we have $fg\in L^1(\Omega)$ and
\begin{align*}
\int_{\Omega}|f(x)g(x)|\mathrm{d}x&\leq\int_{\Omega}(e^{|f(x)|}-|f(x)|-1)\mathrm{d}x+\int_{\Omega}|g(x)|\ln(|g(x)|+1)\mathrm{d}x\nonumber\\
&=:\|f\|_{exp^L(\Omega)}+\|g\|_{L\ln L(\Omega)},
\end{align*}
where the space $\exp^L(\Omega):=\{f|f:\Omega\to\C \ \text{measurable} ,\ \int_\Omega(e^{|f(x)|}-|f(x)|-1)\mathrm{d}x<+\infty\}$ and $L\ln L(\Omega):=\{g|g:\Omega\to\C\ \text{measurable}, \ \int_{\Omega}|g(x)|\ln(|g(x)|+1)\mathrm{d}x<+\infty\}.$
\end{lem}
From Lemma \ref{lem212}, we get
\begin{equation}\label{d7}
\begin{aligned}
\int_{B_{1}(x)}\ln\left(\frac{1}{|x-y|}\right)u^{3}(y)\mathrm{d}y&\leq\int_{B_{1}(x)}\frac{1}{|x-y|}\mathrm{d}y+\int_{B_{1}(x)}u^{3}(y)\ln(u^{3}(y)+1)\mathrm{d}y\\
&\leq2\pi+\left(\max_{|x-y|\leq1}\ln(u^{3}(y)+1)\right)\int_{B_{1}(x)}u^{3}(y)\mathrm{d}y.
\end{aligned}
\end{equation}
Thus, from \eqref{d7} and the conditions $\int_{\mathbb{R}^{3}}u^{3}(x)\mathrm{d}x<+\infty$, $u=O(|x|^{K})$ at $\infty$ for some $K\gg1$ arbitrarily large, we obtain, for any $|x|\geq e^{2}$ large enough,
\begin{equation}\label{d8}
\begin{aligned}
&\quad\Big|\int_{\mathbb{R}^{3}}\frac{\ln|x-y|-\ln|y|-\ln|x|}{\ln|x|}u^{3}(y)\mathrm{d}y\Big|\\
\leq&3\int_{B_{1}(x)}u^{3}(y)\mathrm{d}y+\frac{2\pi}{\ln|x|}+
\frac{O(3K\ln|x|)}{\ln|x|}\int_{B_{1}(x)}u^{3}(y)\mathrm{d}y\\
&+\displaystyle\frac{\max_{|y|\leq\ln|x|}|\ln\frac{|x-y|}{|x|}|}{\ln|x|}
\int_{|y|\leq\ln|x|}u^{3}(y)\mathrm{d}y+\frac{1}{\ln|x|}\int_{|y|<\ln|x|}|\ln|y||u^{3}(y)\mathrm{d}y\\
&+\sup_{|x-y|\geq1,|y|\geq\ln|x|}\frac{\ln|x-y|-\ln|y|-\ln|x|}{\ln|x|}\int_{|y|
\geq\ln|x|}u^{3}(y)\mathrm{d}y\\
\leq&o_{|x|}(1)+\frac{2\pi}{\ln|x|}
+\frac{1}{\ln|x|}\int_{|y|<1}\ln\left(\frac{1}{|y|}\right)u^{3}(y)\mathrm{d}y\\
&\quad+\frac{\ln(\ln|x|)}{\ln|x|}\int_{R^{3}}u^{3}(y)\mathrm{d}y
+\left(2+\frac{\ln2}{\ln|x|}\right)\int_{|y|\geq\ln|x|}u^{3}(y)\mathrm{d}y=o_{|x|}(1).
\end{aligned}
\end{equation}
The last inequality is due to $\frac{1}{2|x|^2}\leq\frac{|x-y|}{|x||y|}\leq\frac{1}{|x|}+\frac{1}{|y|}<1$ for any $|y-x|\geq1$ and $|y|\geq\ln|x|$. Letting $|x|\rightarrow+\infty$ in \eqref{d8}, we have
$$\lim_{|x|\rightarrow+\infty}\int_{\mathbb{R}^{3}}\frac{\ln|x-y|-\ln|y|-\ln|x|}{\ln|x|}u^{3}(y)\mathrm{d}y=0.$$
This proves the asymptotic property \eqref{d5}.

Now, we start to show \eqref{z4}. We need the following Corollary 2.10 in \cite{DQ2}, which is a direct consequence of Lemma 3.3 in \cite{Lin}.
\begin{lem}[Corollary 2.10 in \cite{DQ2}]\label{lemma214}
Assume $n\geq2$. Suppose that $w$ is a harmonic function in $\R^N$. Then we have\\
(i)If $w^+=O(|x|^2)$ at $\infty$, then $w$ is a polynomial of degree at most 2.\\
(ii)If $w^+=o(|x|^2)$ at $\infty$, then $w$ is a polynomial of degree at most 1.\\
(iii)If $w^+=o(|x|)$ at $\infty$, then $w\equiv C$ in $\R^N$ for some constant $C$.
\end{lem}
By the formula \eqref{214} and \eqref{dd220}, we obtain $\Delta(w(x)-k(x))=0$. From the asymptotic property \eqref{d5} and $w(x)=o(|x|^{2})$ at $\infty$, we can immediately derive by Lemma \ref{lemma214} (ii) that, for some constants $a_i\in \mathbb{R}$ ($i=0,1,2,3$),
\begin{equation*}
w(x)-k(x)=\sum_{i=1}^{3}a_ix_i+a_0,\qquad \forall\ x\in\mathbb{R}^{3}.
\end{equation*}
Since \eqref{29} implies $|x|^{-1}e^{pw(x)}=|x|^{-1}e^{pk(x)}e^{pa_0}e^{p\sum_{i=1}^{3}a_ix_i}\in L^1(\R^3)$, we infer from the asymptotic property \eqref{d5} that $a_1=a_2=a_3=0$. Hence the integral representation formula \eqref{z4} for $w$ holds. The asymptotic property \eqref{d4} of $w$ follows immediately from \eqref{z4} and \eqref{d5}. This completes our proof of Lemma \ref{le3}.
\end{proof}

As a consequence of Lemma \ref{le3}, we have the following Corollary.
\begin{cor} \label{le4}
Assume $\int_{\mathbb{R}^{3}}u^{3}(x)\mathrm{d}x<+\infty$ and $u=O(|x|^{K})$ at $\infty$ for some $K\gg1$ arbitrarily large and $w(x)=o(|x|^{2})$ at $\infty$. Then we have, for arbitrarily small $\delta>0$,
\begin{equation}\label{e1}
\lim_{|x|\rightarrow+\infty}\frac{e^{pw(x)}}{|x|^{-\alpha p-\delta}}=+\infty,\
\lim_{|x|\rightarrow+\infty}\frac{e^{pw(x)}}{|x|^{-\alpha p+\delta}}=0.
\end{equation}
Consequently, $\alpha=\frac{1}{2\pi^2}\int_{\mathbb{R}^{3}}u^{3}(x)\mathrm{d}x\geq\frac{2}{p}$. Furthermore, if $\alpha>\frac{3}{p}$, then
\begin{equation}\label{z6}
\gamma:=\frac{1}{2\pi}\int_{\mathbb{R}^{3}}e^{pw(x)}\mathrm{d}x<+\infty
\end{equation}
and
\begin{equation}\label{e3}
\lim_{|x|\rightarrow+\infty}|x|v(x)=\gamma,
\end{equation}
in addition,
\begin{equation}\label{234eq1}
\beta:=\frac{1}{2\pi^2}\int_{\mathbb{R}^{3}}v^{4}(x)\mathrm{d}x<+\infty\end{equation}
and
\begin{equation}\label{e4}
\lim_{|x|\rightarrow+\infty}|x|^{2}u(x)=\beta.
\end{equation}
\end{cor}
\begin{proof}[{\sl Proof of Corollary~\ref{le4}}]
The asymptotic property \eqref{d4} implies that
$$w(x)=-\alpha\ln|x|+o(\ln|x|),\quad \text{as} \ |x|\rightarrow+\infty,$$
where $\alpha$ is defined in \eqref{zz7}. Therefore, we obtain
$$e^{pw(x)}=|x|^{-\alpha p}e^{o(\ln|x|)},\quad \text{as} \ |x|\rightarrow+\infty.$$
Therefore, for arbitrarily small $\delta>0$,
$$\lim_{|x|\rightarrow+\infty}\frac{e^{pw(x)}}{|x|^{-\alpha p-\delta}}=+\infty\quad \text{and}\quad
\lim_{|x|\rightarrow+\infty}\frac{e^{pw(x)}}{|x|^{-\alpha p+\delta}}=0.$$

From \eqref{29} and \eqref{e1}, one can easily infer that $\alpha \geq\frac{2}{p}$. Furthermore, if we assume $\alpha >\frac{3}{p}$, from the asymptotic property \eqref{e1}, it follows immediately that
$\gamma:=\frac{1}{2\pi}\int_{\mathbb{R}^{3}}e^{pw(x)}\mathrm{d}x<+\infty$.

Next, we prove the asymptotic property \eqref{e3} of $v$. In \eqref{e1}, let $\delta=\frac{\alpha p-3}{2}$, then there exists a $R_0\geq1$ sufficiently large such that
\begin{equation}\label{234eq}
e^{pw(x)}\leq|x|^{-\frac{\alpha p+3}{2}},\qquad\forall \,\,|x|\geq R_0.
\end{equation}
From the integral representation formula \eqref{lll1} of $v$, we only need to show
\begin{equation*}
\lim_{|x|\rightarrow+\infty}\int_{\R^3}\frac{|x|-|x-y|}{|x-y|}e^{pw(y)}\mathrm{d}y=0.
\end{equation*}
Indeed, from \eqref{z6} and \eqref{234eq}, for any $|x|\geq R_0$, we have
\begin{align*}
&\quad\left|\int_{\R^3}\frac{|x|-|x-y|}{|x-y|}e^{pw(y)}\mathrm{d}y\right|\\
&\leq\int_{|y-x|<\frac{|x|}{2}}\frac{1}{|x-y||y|^{\frac{\alpha p+1}{2}}}\mathrm{d}y
+3\int_{\{|y-x|\geq\frac{|x|}{2}\}\cap\{|y|\geq\frac{|x|}{2}\}}
e^{pw(y)}\mathrm{d}y\\
&\quad+\frac{2}{|x|}\int_{|y|<R_0}|y|e^{pw(y)}\mathrm{d}y
+\frac{2}{|x|}\int_{R_0\leq |y|<\frac{|x|}{2}}\frac{1}{|y|^{\frac{\alpha p+1}{2}}}\mathrm{d}y\\
&\leq\frac{2^{\frac{\alpha p+3}{2}}}{(\alpha p+1)|x|^{\frac{\alpha p-3}{2}}}+o_{|x|}(1)+\frac{2}{|x|}\int_{|y|<R_0}|y|e^{pw(y)}\mathrm{d}y+\frac{8\pi}{|x|}
\xi(x)=o_{|x|}(1),
\end{align*}
where $\xi(x)=\frac{2}{5-\alpha p}\left(\frac{|x|}{2}\right)^{\frac{5-\alpha p}{2}}$ if $3<\alpha p<5$, $\xi(x)=\ln\left(\frac{|x|}{2}\right)$ if $\alpha p=5$, and $\xi(x)=\frac{2}{\alpha p-5}R_0^{\frac{\alpha p-5}{2}}$ if $\alpha p>5$. Thus we obtain the asymptotic property \eqref{e3} for $v$. From \eqref{e3}, we get
\begin{equation}\label{237eq}
v(x)\leq\frac{C}{|x|},\qquad\forall \,\, |x|\geq R_0.
\end{equation}
Thus we have $\beta:=\frac{1}{2\pi^2}\int_{\mathbb{R}^{3}}v^{4}(x)\mathrm{d}x<+\infty$.

Finally, we derive the asymptotic property \eqref{e4} of $u$. By \eqref{l1}, we have
\begin{equation}\label{e5}
|x|^{2}u(x)-\beta=\frac{1}{2\pi^2}\int_{\mathbb{R}^{3}}\frac{|x|^{2}-|x-y|^{2}}{|x-y|^{2}}v^{4}(y)\mathrm{d}y.
\end{equation}
Therefore, we only need to show that the right hand side of \eqref{e5} goes to $0$ as $|x|\rightarrow+\infty$.
According to \eqref{234eq1} and \eqref{237eq}, for any $|x|\geq R_0$, we obtain
\begin{align*}
&\quad\left|\int_{\R^3}\frac{|x|^{2}-|x-y|^{2}}{|x-y|^{2}}v^{4}(y)\mathrm{d}y\right|\\
&\leq\int_{|y-x|<\frac{|x|}{2}}\frac{1}{|x-y|^2|y|^{2}}\mathrm{d}y
+5\int_{\{|y-x|\geq\frac{|x|}{2}\}\cap\{|y|\geq\frac{|x|}{2}\}}
v^{4}(y)\mathrm{d}y\nonumber\\
&\quad+\frac{4}{|x|^2}\int_{|y|<R_0}|y|^2v^{4}(y)\mathrm{d}y
+\frac{4}{|x|^2}\int_{R_0\leq |y|<\frac{|x|}{2}}\frac{1}{|y|^{2}}\mathrm{d}y\nonumber\\
&\leq\frac{8\pi}{3|x|}+o_{|x|}(1)+\frac{4}{|x|^2}\int_{|y|<R_0}|y|^2v^{4}(y)\mathrm{d}y
+\frac{8\pi}{|x|^3}=o_{|x|}(1).\nonumber
\end{align*}
This proves the asymptotic property \eqref{e4} and hence concludes our proof of Corollary \ref{le4}.
\end{proof}

\medskip

From Lemmas \ref{le1}, \ref{le2} and \ref{le3}, we have proved that the classical solution $(u,v,w)$ of the $3$-D system \eqref{a1} solve the following integral system:
\begin{equation}\label{a2-3d}
\begin{cases}\displaystyle
\ u(x)=\frac{1}{2\pi^2}\int_{\mathbb{R}^{3}}\frac{v^{4}(y)}{|x-y|^{2}}\mathrm{d}y,\\
\displaystyle \ v(x)=\frac{1}{4\pi}\int_{\mathbb{R}^{3}}\frac{e^{pw(y)}}{|x-y|}\mathrm{d}y,\\
\displaystyle \ w(x)=\frac{1}{2\pi^2}\int_{\mathbb{R}^{3}}\ln\left(\frac{|y|}{|x-y|}\right)u^{3}(y)\mathrm{d}y+C_0,
\end{cases}
\end{equation}
where $C_0\in \R$. Next, we will apply the method of moving spheres to show that $\alpha=\frac{5}{p}$ and derive the classification of $(u,v,w)$.

For this purpose, for arbitrarily given $x_{0}\in \mathbb{R}^{3}$ and any $\lambda>0$, we define the Kelvin transformation of $(u,v,w)$ centered at $x_{0}$ by
\begin{equation}\label{r1}
\begin{cases}
\displaystyle u_{x_{0},\lambda}(x)=\frac{\lambda^{2}}{|x-x_{0}|^{2}}u(x^{x_{0},\lambda}),\\[3mm] \displaystyle v_{x_{0},\lambda}(x)=\frac{\lambda}{|x-x_{0}|}v(x^{x_{0},\lambda}),\\[3mm] \displaystyle w_{x_{0},\lambda}(x)=w(x^{x_{0},\lambda})+\frac{5}{p}\ln\frac{\lambda}{|x-x_{0}|},
\end{cases}
\end{equation}
for arbitrary $x\in \mathbb{R}^{3}\backslash\{x_{0}\}$, where $x^{x_{0},\lambda}=\frac{\lambda^{2}(x-x_{0})}{|x-x_{0}|^{2}}+x_{0}$.

Define
$$Q^{u}_{x_0,\la}(x)=u_{x_0,\lambda}(x)-u(x),\ Q^{v}_{x_0,\la}(x)=v_{x_0,\lambda}(x)-v(x) \ \text{and}\ Q^{w}_{x_0,\la}(x)=w_{x_0,\lambda}(x)-w(x),$$
where $\lambda>0$ is an arbitrary positive real number. Denote the  moving the sphere
$$ S_{\lambda}(x_0):=\{x\in \mathbb{R}^{3}\mid |x-x_0|=\lambda\}.$$

Since $(u,v,w)$ solve the integral system \eqref{a2-3d}, through direct calculations, we obtain, for any $\lambda>0$ and all $x\in \mathbb{R}^{3}$,
\begin{align}
u(x)&=\frac{1}{2\pi^2}\int_{\mathbb{R}^{3}}\frac{v^{4}(y)}{|x-y|^{2}}\mathrm{d}y\nonumber\\
&=\frac{1}{2\pi^2}\int_{B_{\lambda}(x_0)}\frac{v^{4}(y)}{|x-y|^{2}}\mathrm{d}y
+\frac{1}{2\pi^2}\int_{B_{\lambda}(x_0)}\frac{v^{4}(y^{x_0,\lambda})}
{|x-y^{x_0,\lambda}|^{2}}\left(\frac{\lambda}{|y-x_0|}\right)^{6}\mathrm{d}y\nonumber\\
&=\frac{1}{2\pi^2}\int_{B_{\lambda}(x_0)}\frac{v^{4}(y)}{|x-y|^{2}}\mathrm{d}y
+\frac{1}{2\pi^2}\int_{B_{\lambda}(x_0)}\frac{v_{x_0,\lambda}^{4}(y)}
{\left|\frac{\lambda(x-x_0)}{|x-x_0|}-\frac{|x-x_0|(y-x_0)}{\lambda}\right|^{2}}\mathrm{d}y\label{a4},
\end{align}
\begin{align}
v(x)&=\frac{1}{4\pi}\int_{\mathbb{R}^{3}}\frac{e^{pw(y)}}{|x-y|}\mathrm{d}y\nonumber\\
&=\frac{1}{4\pi}\int_{B_{\lambda}(x_0)}\frac{e^{pw(y)}}{|x-y|}\mathrm{d}y
+\frac{1}{4\pi}\int_{B_{\lambda}(x_0)}\frac{e^{pw(y^{x_0,\lambda})}}{|x-y^{x_0,\lambda}|}\left(\frac{\lambda}{|y-x_0|}\right)^{6}\mathrm{d}y\nonumber\\
&=\frac{1}{4\pi}\int_{B_{\lambda}(x_0)}\frac{e^{pw(y)}}{|x-y|}\mathrm{d}y
+\frac{1}{4\pi}\int_{B_{\lambda}(x_0)}\frac{e^{pw_{x_0,\lambda}(y)}}{\left|\frac{\lambda(x-x_0)}{|x-x_0|}-\frac{|x-x_0|(y-x_0)}{\lambda}\right|}\mathrm{d}y\label{a5}
\end{align}
and
\begin{align}
w(x)&=\frac{1}{2\pi^2}\int_{\mathbb{R}^{3}}\ln\left(\frac{|y|}{|x-y|}\right)u^{3}(y)\mathrm{d}y+C_0\nonumber\\
&=\frac{1}{2\pi^2}\int_{B_{\lambda}(x_0)}\ln\left(\frac{|y|}{|x-y|}\right)u^{3}(y)\mathrm{d}y
\nonumber\\
&\quad+\frac{1}{2\pi^2}\int_{B_{\lambda}(x_0)}\ln\left(\frac{|y^{x_0,\lambda}|}{|x-y^{x_0,\lambda}|}\right)u^{3}(y^{x_0,\lambda})\left(\frac{\lambda}{|y-x_0|}\right)^{6}\mathrm{d}y+C_0\nonumber\\
&=\frac{1}{2\pi^2}\int_{B_{\lambda}(x_0)}\ln\left(\frac{|y|}{|x-y|}\right)u^{3}(y)\mathrm{d}y
+\frac{1}{2\pi^2}\int_{B_{\lambda}(x_0)}\ln\left(\frac{|y^{x_0,\lambda}|}{|x-y^{x_0,\lambda}|}\right)u_{x_0,\lambda}^{3}(y)\mathrm{d}y+C_0\label{a6}.
\end{align}
Furthermore, we deduce from \eqref{a4}, \eqref{a5} and \eqref{a6} that, for any $\lambda>0$ and all $x\in \mathbb{R}^{3}\backslash\{x_0\}$,
\begin{align*}
u_{x_0,\lambda}(x)&=\frac{1}{2\pi^2}\left(\frac{\lambda}{|x-x_0|}\right)^{2}
\int_{\mathbb{R}^{3}}\frac{v^{4}(y)}{|x^{x_0,\lambda}-y|^{2}}\mathrm{d}y\nonumber\\
&=\frac{1}{2\pi^2}\left(\frac{\lambda}{|x-x_0|}\right)^{2}\left(\int_{B_{\lambda}(x_0)}\frac{v^{4}(y)}{|x^{x_0,\lambda}-y|^{2}}\mathrm{d}y
+\int_{B_{\lambda}(x_0)}\frac{v^{4}(y^{x_0,\lambda})}{|x^{x_0,\lambda}-y^{x_0,\lambda}|^{2}}\left(\frac{\lambda}{|y-x_0|}\right)^{6}\mathrm{d}y\right)\nonumber\\
&=\frac{1}{2\pi^2}\int_{B_{\lambda}(x_0)}\frac{v^{4}(y)}{\left|\frac{\lambda(x-x_0)}{|x-x_0|}-\frac{|x-x_0|(y-x_0)}{\lambda}\right|^{2}}\mathrm{d}y
+\frac{1}{2\pi^2}\int_{B_{\lambda}(x_0)}\frac{v_{x_0,\lambda}^{4}(y)}{|x-y|^{2}}\mathrm{d}y,
\end{align*}
\begin{align*}
v_{x_0,\lambda}(x)&=\frac{1}{4\pi}\frac{\lambda}{|x-x_0|}\int_{\mathbb{R}^{3}}\frac{e^{pw(y)}}{|x^{x_0,\lambda}-y|}\mathrm{d}y\nonumber\\
&=\frac{1}{4\pi}\frac{\lambda}{|x-x_0|}\left(\int_{B_{\lambda}(x_0)}\frac{e^{pw(y)}}{|x^{x_0,\lambda}-y|}\mathrm{d}y
+\int_{B_{\lambda}(x_0)}\frac{e^{pw(y^{x_0,\lambda})}}{|x^{x_0,\lambda}-y^{x_0,\lambda}|}\left(\frac{\lambda}{|y-x_0|}\right)^{6}\mathrm{d}y\right)\nonumber\\
&=\frac{1}{4\pi}\int_{B_{\lambda}(x_0)}\frac{e^{pw(y)}}{\left|
\frac{\lambda(x-x_0)}{|x-x_0|}-\frac{|x-x_0|(y-x_0)}{\lambda}\right|}\mathrm{d}y+\frac{1}{4\pi}
\int_{B_{\lambda}(x_0)}\frac{e^{pw_{x_0,\lambda}(y)}}{|x-y|}\mathrm{d}y,
\end{align*}
\begin{align}
w_{x_0,\lambda}(x)&=\frac{1}{2\pi^2}\int_{B_{\lambda}(x_0)}
\ln\left(\frac{|y|}{|x^{x_0,\lambda}-y|}\right)u^{3}(y)\mathrm{d}y
+\frac{1}{2\pi^2}\int_{B_{\lambda}(x_0)}
\ln\left(\frac{|y^{x_0,\lambda}|}{|x^{x_0,\lambda}-y^{x_0,\lambda}|}\right)u_{x_0,\lambda}^{3}(y)\mathrm{d}y\nonumber\\
&\quad+\frac{5}{p}\ln\frac{\lambda}{|x-x_0|}+C_0\label{a9}.
\end{align}

Consequently, follows from \eqref{a4}--\eqref{a9} that, for any $x\in B_{\lambda}(x_0)\backslash\{x_0\}$, we obtain
\begin{equation}\label{a10}
\begin{aligned}
Q^{u}_{x_0,\la}(x)
&=\frac{1}{2\pi^2}\int_{B_{\lambda}(x_0)}\left(
\frac{1}{|x-y|^{2}}-
\frac{1}{|\frac{\lambda(x-x_0)}{|x-x_0|}-\frac{|x-x_0|(y-x_0)}{\lambda}|^{2}}\right)(v_{x_0,\lambda}^{4}(y)-v^{4}(y))\mathrm{d}y,
\end{aligned}
\end{equation}
\begin{equation}\label{a11}
\begin{aligned}
Q^{v}_{x_0,\la}(x)
&=\frac{1}{4\pi}\int_{B_{\lambda}(x_0)}
\left(\frac{1}{|x-y|}-\frac{1}{|\frac{\lambda(x-x_0)}{|x-x_0|}-\frac{|x-x_0|(y-x_0)}{\lambda}|}\right)
(e^{pw_{x_0,\lambda}(y)}-e^{pw(y)})\mathrm{d}y,
\end{aligned}
\end{equation}
\begin{equation}\label{a12}
\begin{aligned}
Q^{w}_{x_0,\la}(x)
&=\frac{1}{2\pi^2}\int_{B_{\lambda}(x_0)}\ln\left(\frac{|\frac{\lambda(x-x_0)}{|x-x_0|}-\frac{|x-x_0|(y-x_0)}{\lambda}|}{|x-y|}\right)(u_{x_0,\lambda}^{3}(y)-u^{3}(y))\mathrm{d}y
\nonumber\\
&\quad+\left(\frac{5}{p}-\alpha\right)\ln\frac{\lambda}{|x-x_0|},
\end{aligned}
\end{equation}
where $\alpha:=\frac{1}{2\pi^2}\int_{\mathbb{R}^{3}}u^{3}(x)\mathrm{d}x.$

\smallskip

Next, we begin to move the sphere $ S_{\lambda}(x_0)$ from near $\la=0$ or $\la=+\infty$ to its limiting position. Therefore, the process of moving the sphere can be divided into two steps. We will discuss the two different cases $\alpha\geq\frac{5}{p}$ and $\frac{2}{p}\leq\alpha\leq\frac{5}{p}$ separately below.

\medskip

$\mathbf{Step\ 1.}$ Start moving the sphere $S_{\lambda}(x_0)$ from near $\lambda=0$ or $\la=+\infty$.

\smallskip

$\mathbf{Case\ 1:}$ $\alpha\geq \frac{5}{p}$. For $\lambda>0$ sufficiently large, we will show that
\begin{equation}\label{a16}
Q^{u}_{x_0,\la}(x)\leq0,\quad Q^{v}_{x_0,\la}(x)\leq0\quad\text{and}\quad Q^{w}_{x_0,\la}(x)\leq0,\quad \forall x\in B_{\lambda}(x_0)\backslash\{x_0\}.
\end{equation}
Namely, we begin moving the sphere $S_{\lambda}(x_0)$ from near $\la=+\infty$ inward the point $x_0$ such that \eqref{a16} holds.
Denote
\begin{equation*}
B_{\lambda,u}^{+}(x_0):=\{x\in B_{\lambda}(x_0)\backslash\{x_0\}\mid Q^{u}_{x_0,\la}(x)>0 \},
\end{equation*}
\begin{equation*}
B_{\lambda,v}^{+}(x_0):=\{x\in B_{\lambda}(x_0)\backslash\{x_0\}\mid Q^{v}_{x_0,\la}(x)>0 \},
\end{equation*}
\begin{equation*}
B_{\lambda,w}^{+}(x_0):=\{x\in B_{\lambda}(x_0)\backslash\{x_0\}\mid Q^{w}_{x_0,\la}(x)>0 \}.
\end{equation*}
The formula \eqref{a16} equivalent to proving, for $\lambda>0$ big enough,
$$B_{\lambda,u}^{+}(x_0)=B_{\lambda,v}^{+}(x_0)=B_{\lambda,w}^{+}(x_0)=\emptyset.$$
From \eqref{a10} and the mean value theorem, we deduce that, for any $x\in B_{\lambda,u}^{+}(x_0)$,
\begin{align}
0<Q^{u}_{x_0,\la}(x)
&\leq\frac{1}{2\pi^2}\int_{B_{\lambda,v}^{+}(x_0)}\left(\frac{1}{|x-y|^{2}}
-\frac{1}{|\frac{\lambda(x-x_0)}{|x-x_0|}-\frac{|x-x_0|(y-x_0)}{\lambda}|^{2}}\right)(v_{x_0,\lambda}^{4}
(y)-v^{4}(y))\mathrm{d}y\nonumber\\
&\leq\frac{1}{2\pi^2}\int_{B_{\lambda,v}^{+}(x_0)}\frac{4}{|x-y|^{2}}
(\xi_{x_0,\lambda}^v(y))^{3}Q^{v}_{x_0,\la}(y)\mathrm{d}y\nonumber\\
&\leq\frac{4}{2\pi^2}\int_{B_{\lambda,v}^{+}(x_0)}\frac{v_{x_0,\lambda}^{3}(y)}{|x-y|^{2}}
Q^{v}_{x_0,\la}(y)\mathrm{d}y,\label{a18}
\end{align}
where $v(y)<\xi_{x_0,\lambda}^v(y)<v_{x_0,\lambda}(y)$, for any $y\in B_{\lambda,v}^{+}(x_0)$. Similarly, for any $x\in B_{\lambda,v}^{+}(x_0)$, from \eqref{a11}, we have
\begin{align}
0<Q^{v}_{x_0,\la}(x)
&\leq\frac{1}{4\pi}\int_{B_{\lambda,w}^{+}(x_0)}\left(\frac{1}{|x-y|}-\frac{1}
{|\frac{\lambda(x-x_0)}{|x-x_0|}-\frac{|x-x_0|(y-x_0)}{\lambda}|}\right)(e^{pw_{x_0,\lambda}(y)}-e^{pw(y)})\mathrm{d}y\nonumber\\
&\leq\frac{1}{4\pi}\int_{B_{\lambda,w}^{+}(x_0)}\frac{p}{|x-y|}e^{p\xi_{x_0,\lambda}^w(y)}Q^{w}_{x_0,\la}(y)\mathrm{d}y\nonumber\\
&\leq\frac{p}{4\pi}\int_{B_{\lambda,w}^{+}(x_0)}\frac{e^{pw_{x_0,\lambda}(y)}}{|x-y|}Q^{w}_{x_0,\la}(y)\mathrm{d}y,\label{a19}
\end{align}
where $w(y)<\xi_{x_0,\lambda}^w(y)<w_{x_0,\lambda}(y)$, for any $y\in B_{\lambda,w}^{+}(x_0)$. Note that, for any $x,y\in B_\la(x_0)\backslash\{x_0\},\ x\ne y$, we have
\begin{align}
0<2\ln\left(\frac{|\frac{\lambda(x-x_0)}{|x-x_0|}-\frac{|x-x_0|(y-x_0)}{\lambda}|}
{|x-y|}\right)&=\ln\lr
1+\frac{(\la-\frac{|x-x_0|^2}{\la})(\la-\frac{|y-x_0|^2}{\la})}{|x-y|^2}\rr\nonumber\\
&\leq\ln\lr1+\frac
{\la^2}{|x-y|^2}\rr.\label{312ddz}
\end{align}
On the other hand, for any arbitrary $\varepsilon >0$, $\ln(1+t)=o(t^\var),\ \text{as}\ t\to\infty$, which indicates, for any given $\varepsilon >0$, there exists a $\eta(\varepsilon)>0$ such that
\begin{align}\label{313ddz}
\ln(1+t)\leq t^\var,\quad \forall t>\frac{1}{\eta^2(\var)}.
\end{align}
Thus, from \eqref{312ddz} and \eqref{313ddz}, we obtain
\begin{small}
\begin{align}
\ln\left(\frac{|\frac{\lambda(x-x_0)}{|x-x_0|}-\frac{|x-x_0|(y-x_0)}{\lambda}|}{|x-y|}\right)\leq
\frac{1}{2}\frac{\lambda^{2\varepsilon}}{|x-y|^{2\varepsilon}}, \quad \forall \,\, 0<|x-x_0|,|y-x_0|<\lambda,|x-y|<\lambda\eta(\varepsilon)\label{a20},\\
\ln\left(\frac{|\frac{\lambda(x-x_0)}{|x-x_0|}-\frac{|x-x_0|(y-x_0)}{\lambda}|}{|x-y|}\right)\leq
\frac{1}{2}\ln\Big(1+\frac{1}{\eta^2(\varepsilon)}\Big), \,\, \forall \, 0<|x-x_0|,|y-x_0|<\lambda,|x-y|\geq\lambda\eta(\varepsilon)\label{a21}.
\end{align}
\end{small}

Since $\alpha\geq\frac{5}{p}$, from \eqref{a12}, \eqref{a20}, \eqref{a21} and the mean value theorem, we derive that, for any given $\var>0$ small and $x\in B_{\lambda,w}^{+}(x_0)$,
\begin{equation}\label{a22}
\begin{aligned}
0<Q^{w}_{x_0,\la}(x)
&\leq\frac{1}{2\pi^2}\int_{B_{\lambda,u}^{+}(x_0)}\ln\left(
\frac{|\frac{\lambda(x-x_0)}{|x-x_0|}-\frac{|x-x_0|(y-x_0)}{\lambda}|}{|x-y|}\right)(u_{x_0,\lambda}^{3}(y)-u^{3}(y))\mathrm{d}y\\
&\leq\frac{3}{4\pi^2}\int_{B_{\lambda,u}^{+}(x_0)\cap B_{\lambda\eta(\varepsilon)}(x_0)}\frac{\lambda^{2\var}}{|x-y|^{2\var}} (\xi_{x_0,\lambda}^u(y))^{2} Q^{u}_{x_0,\la}(y)\mathrm{d}y\\
&\quad+ \frac{3}{4\pi^2} \ln \left(1+\frac{1}{\eta^2(\varepsilon)}\right)\int_{B_{\lambda,u}^{+}(x_0)\backslash B_{\lambda\eta(\varepsilon)}(x_0)}(\xi_{x_0,\lambda}^u(y))^{2} Q^{u}_{x_0,\la}(y)\mathrm{d}y\\
&\leq\frac{3}{4\pi^2}\int_{B_{\lambda,u}^{+}(x_0)\cap B_{\lambda\eta(\varepsilon)}(x_0)}\frac{\lambda^{2\var}}{|x-y|^{2\var}} u_{x_0,\lambda}^{2} (y) Q^{u}_{x_0,\la}(y)\mathrm{d}y\\
&\quad+C(\eta(\varepsilon))\int_{B_{\lambda,u}^{+}(x_0)\backslash B_{\lambda\eta(\varepsilon)}(x_0)}u_{x_0,\lambda}^{2} (y) Q^{u}_{x_0,\la}(y)\mathrm{d}y,
\end{aligned}
\end{equation}
where $u(y)< \xi_{x_0,\lambda}^u(y)< u_{x_0,\lambda}(y)$, for any $y\in  B_{\lambda,u}^{+}(x_0)$.
\smallskip

Now we need the following Hardy-Littlewood-Sobolev inequality (c.f. e.g. \cite{FL1,FL,Lieb}, see also \cite{DHL,DGZ,DZ,NN}).
\begin{lem} [Hardy-Littlewood-Sobolev inequality]
Let $n\geq1$, $0<s<n$, and $1<p<q<+\infty$ be such that $s+\frac{n}{q}=\frac{n}{p}$. Then we have
\begin{equation}\label{hl}
\left\|\int_{\mathbb{R}^{n}}\frac{f(y)}{|x-y|^{n-s}}\mathrm{d}y\right\|_{L^{q}(\mathbb{R}^{n})}\leq C_{n,s,p,q}\|f\|_{L^{p}(\mathbb{R}^{n})},
\end{equation}
for all $f\in L^{p}(\mathbb{R}^{n})$.
\end{lem}
From the Hardy-Littlewood-Sobolev inequality \eqref{hl}, H\"older inequality, \eqref{a18}, \eqref{a19} and \eqref{a22}, we get, for any given $\var\in(0,\frac{1}{16})$ sufficiently small to be chosen later and any $t\in(\frac{3}{2\var},+\infty)$,
\begin{align}\label{a23}
\|Q_{x_0,\lambda}^u(x)\|_{L^q(B_{\lambda,u}^+(x_0))}
&\leq C\|v^3_{x_0,\lambda}(x)Q_{x_0,\lambda}^v(x)\|_{L^{\frac{4r}{3r+4}}(B_{\lambda,v}^+(x_0))}\nonumber\\
&\leq C\|v_{x_0,\lambda}(x)\|^3_{L^{4}(B_{\lambda,v}^+(x_0))}
\|Q_{x_0,\lambda}^v(x)\|_{L^{r}(B_{\lambda,v}^+(x_0))}\nonumber\\
&\leq C\la^{\frac{3}{2}}\lr\int_{\R^3}\frac{v^4(x)}{|x-x_0|^2}\mathrm{d}x\rr^{\frac{3}{4}}
\|Q_{x_0,\lambda}^v(x)\|_{L^{r}(B_{\lambda,v}^+(x_0))},
\end{align}
\begin{align}\label{a24}
\|Q_{x_0,\lambda}^v(x)\|_{L^r(B_{\lambda,v}^+(x_0))}& \leq C\|e^{pw_{x_0,\lambda}(x)}Q_{x_0,\lambda}^w(x)\|_{L^{\frac{5t}{4t+5}}
(B_{\lambda,w}^+(x_0))}\nonumber\\
& \leq C\|e^{pw_{x_0,\lambda}(x)}\|_{L^{\frac{5}{4}}
(B_{\lambda,w}^+(x_0))}\|Q_{x_0,\lambda}^w(x)\|_{L^{t}(B_{\lambda,w}^+(x_0))}\nonumber\\
&\leq C\la^{-\frac{1}{5}}\lr\int_{\R^3}|x-x_0|^\frac{1}{4}
e^{\frac{5}{4}pw(x)}\mathrm{d}x\rr^{\frac{4}{5}}
\|Q_{x_0,\lambda}^w(x)\|_{L^{t}(B_{\lambda,w}^+(x_0))},
\end{align}
and
\begin{align}\label{a25}
&\quad\|Q_{x_0,\lambda}^w(x)\|_{L^t(B_{\lambda,w}^+(x_0))}\nonumber\\
&\leq C\lambda^{2\var}\|u_{x_0,\lambda}^2(x) Q_{x_0,\lambda}^u(x)\|_{L^{\frac{3t}{3+(3-2\var)t}}(B_{\lambda,v}^+(x_0))}\nonumber\\
&\quad+C(\eta(\varepsilon))|B_{\lambda,w}^+(x_0)|^{\frac{1}{t}}
\int_{B_{\lambda,u}^+(x_0)}u_{x_0,\lambda}^2(x) Q_{x_0,\lambda}^u(x)\mathrm{d}x\nonumber\\
&\leq C\lambda^{2\var}\|u_{x_0,\lambda}(x)\|^2_{L^{\frac{120}{27-40\var}}(B_{\lambda,u}^+(x_0))}
\|Q_{x_0,\lambda}^u(x)\|_{L^{q}(B_{\lambda,u}^+(x_0))}\nonumber\\
&\quad+C(\eta(\varepsilon))\lambda^{\frac{3}{t}}\|u_{x_0,\lambda}(x)\|^2_{L^{\frac{2q}{q-1}}
(B_{\lambda,u}^+(x_0))}
\|Q_{x_0,\lambda}^u(x)\|_{L^{q}(B_{\lambda,u}^+(x_0))}\nonumber\\
&\leq C\lambda^{-\frac{13}{10}-2\var}\lr\int_{\R^3}|x-x_0|^{\frac{78+240\var}{27-40\var}}
u^{\frac{120}{27-40\var}}(x)\mathrm{d}x\rr^{\frac{27-40\var}{60}}
\|Q_{x_0,\lambda}^u(x)\|_{L^{q}(B_{\lambda,u}^+(x_0))}\nonumber\\
&\quad+ C(\eta(\varepsilon))\lambda^{-\frac{13}{10}-\frac{3}{t}}\lr\int_{\R^3}
|x-x_0|^{\frac{26t+120}{9t-20}}u^{\frac{40t}{9t-20}}(x)\mathrm{d}x\rr^{\frac{9t-20}{20t}}
\|Q_{x_0,\lambda}^u(x)\|_{L^{q}(B_{\lambda,u}^+(x_0))},
\end{align}
where $q=\frac{12r}{12+5r}\in(\frac{120}{71},\frac{20}{11})$, and $r=\frac{15t}{15+2t}\in(\frac{120}{21},\frac{15}{2})$. Since $\alpha\geq\frac{5}{p}$, from Lemma \ref{le4} asymptotic properties of $(u,v,w)$ in \eqref{e1}, \eqref{e3} and \eqref{e4}, we deduce that, for $\var\in(0,\frac{1}{16})$ sufficiently small and $t\in(\frac{3}{2\var},+\infty)$ large enough, $\frac{v^4}{|x-x_0|^2}\in L^1(\R^3)$, $|x-x_0|^\frac{1}{4}e^{\frac{5}{4}pw}\in L^1(\R^3)$, $|x-x_0|^{\frac{78+240\var}{27-40\var}}
u^{\frac{120}{27-40\var}}\in L^1(\R^3)$ and $|x-x_0|^{\frac{26t+120}{9t-20}}u^{\frac{40t}{9t-20}}\in L^1(\R^3)$. Therefore, by \eqref{a23}, \eqref{a24} and \eqref{a25}, one has
\begin{align*}
\|Q^{u}_{x_0,\la}(x)\|_{L^{q}(B_{\lambda,u}^{+}(x_0))}\leq C_{t,\var}\max\{\la^{-2\var},\la^{-\frac{3}{t}}\}\|Q^{u}_{x_0,\la}(x)\|_{L^{q}(B_{\lambda,u}^{+}(x_0))},
\end{align*}
\begin{align*}
\|Q^{v}_{x_0,\la}(x)\|_{L^{r}(B_{\lambda,v}^{+}(x_0))}\leq C_{t,\var}\max\{\la^{-2\var},\la^{-\frac{3}{t}}\}
\|Q^{v}_{x_0,\la}(x)\|_{L^{r}(B_{\lambda,v}^{+}(x_0))},
\end{align*}
\begin{align*}
\|Q^{w}_{x_0,\la}(x)\|_{L^{t}(B_{\lambda,w}^{+}(x_0))}\leq C_{t,\var}\max\{\la^{-2\var},\la^{-\frac{3}{t}}\}\|Q^{w}_{x_0,\la}(x)\|_{L^{t}(B_{\lambda,w}^{+}(x_0))}.
\end{align*}
Therefore, choose $\var\in(0,\frac{1}{16})$ small enough and $t\in(\frac{3}{2\var},+\infty)$ large enough, there exists a $\Lambda_{0}>0$ big enough such that
\begin{equation}\label{a30}
\begin{aligned}
\|Q^{u}_{x_0,\la}(x)\|_{L^{q}(B_{\lambda,u}^{+}(x_0))}\leq \frac{1}{2}\|Q^{u}_{x_0,\la}(x)\|_{L^{q}(B_{\lambda,u}^{+}(x_0))},
\end{aligned}
\end{equation}
\begin{equation}\label{a31}
\begin{aligned}
\|Q^{v}_{x_0,\la}(x)\|_{L^{r}(B_{\lambda,v}^{+}(x_0))}\leq \frac{1}{2}\|Q^{v}_{x_0,\la}(x)\|_{L^{r}(B_{\lambda,v}^{+}(x_0))},
\end{aligned}
\end{equation}
\begin{equation}\label{a32}
\begin{aligned}
\|Q^{w}_{x_0,\la}(x)\|_{L^{t}(B_{\lambda,w}^{+}(x_0))}\leq \frac{1}{2}\|Q^{w}_{x_0,\la}(x)\|_{L^{t}(B_{\lambda,w}^{+}(x_0))}
\end{aligned}
\end{equation}
for all $\Lambda_{0}\leq\lambda<+\infty$. By \eqref{a30}, \eqref{a31} and \eqref{a32}, we obtain
$$\|Q^{u}_{x_0,\la}(x)\|_{L^{q}(B_{\lambda,u}^{+}(x_0))}=\|Q^{v}_{x_0,\la}(x)\|_{L^{r}(B_{\lambda,v}^{+}(x_0))}
=\|Q^{w}_{x_0,\la}(x)\|_{L^{t}(B_{\lambda,w}^{+}(x_0))}=0,$$
which means $B_{\lambda,u}^{+}(x_0)=B_{\lambda,v}^{+}(x_0)=B_{\lambda,w}^{+}(x_0)=\emptyset$, for any $\Lambda_0\leq\lambda<+\infty$.

Thus, for all $\Lambda_0\leq\lambda<+\infty$,
$$Q^{u}_{x_0,\la}(x)\leq0,\ Q^{v}_{x_0,\la}(x)\leq0,\ Q^{w}_{x_0,\la}(x)\leq0,\quad \forall x\in B_{\lambda}(x_0)\backslash\{x_0\}.$$
which completes Step 1 for the case $\alpha\geq\frac{5}{p}$.

\smallskip

$\mathbf{Case\ 2:}$  $\frac{2}{p}\leq\alpha\leq\frac{5}{p}$. We will prove that, for $\lambda>0$ sufficiently small,
\begin{equation}\label{b1}
Q^{u}_{x_0,\la}(x)\geq0,\ Q^{v}_{x_0,\la}(x)\geq0,\ Q^{w}_{x_0,\la}(x)\geq0,\quad \forall x\in B_{\lambda}(x_0)\backslash\{x_0\}.
\end{equation}
Namely, we start moving the sphere
$S_{\lambda}(x_0):=\{x\in \mathbb{R}^{3}\mid |x-x_0|=\lambda\}$
from near the point $x_0$ outward such that (\ref{b1}) holds. Denote
\begin{equation*}
B_{\lambda,u}^{-}(x_0):=\{x\in B_{\lambda}(x_0)\backslash\{x_0\}\mid Q^{u}_{x_0,\la}(x)<0 \},
\end{equation*}
\begin{equation*}
B_{\lambda,v}^{-}(x_0):=\{x\in B_{\lambda}(x_0)\backslash\{x_0\}\mid Q^{v}_{x_0,\la}(x)<0 \},
\end{equation*}
\begin{equation*}
B_{\lambda,w}^{-}(x_0):=\{x\in B_{\lambda}(x_0)\backslash\{x_0\}\mid Q^{w}_{x_0,\la}(x)<0 \}.
\end{equation*}
For $\lambda>0$ sufficiently small, \eqref{b1} is equivalent to proving that the following statement
\begin{equation*}
B_{\lambda,u}^{-}(x_0)=B_{\lambda,v}^{-}(x_0)=B_{\lambda,w}^{-}(x_0)=\emptyset
\end{equation*}
is true. From the mean value theorem and \eqref{a10}, we derive that, for any $x\in B_{\lambda,u}^{-}(x_0)$,
\begin{align}
0>Q^{u}_{x_0,\la}(x)&=u_{x_0,\lambda}(x)-u(x)\nonumber\\
&\geq\frac{1}{2\pi^2}\int_{B_{\lambda,v}^{-}(x_0)}\left(\frac{1}{|x-y|^{2}}
-\frac{1}{|\frac{\lambda(x-x_0)}{|x-x_0|}-\frac{|x-x_0|(y-x_0)}{\lambda}|^{2}}\right)
(v_{x_0,\lambda}^{4}(y)-v^{4}(y))\mathrm{d}y\nonumber\\
&\geq\frac{1}{2\pi^2}\int_{B_{\lambda,v}^{-}(x_0)}\frac{4}{|x-y|^{2}}
(\bar{\xi}^v_{x_0,\lambda}(y))^{3}Q^{v}_{x_0,\la}(y)\mathrm{d}y\nonumber\\
&\geq\frac{4}{2\pi^2}\int_{B_{\lambda,v}^{-}(x_0)}\frac{v^{3}(y)}{|x-y|^{2}}
Q^{v}_{x_0,\la}(y)\mathrm{d}y,\label{b6}
\end{align}
where $v_{x_0,\lambda}(y)<\bar{\xi}_{x_0,\lambda}^v(y)<v(y)$, for any $y\in B_{\lambda,v}^{-}(x_0)$. From \eqref{a11}, for any $x\in B_{\lambda,v}^{-}(x_0)$, we have
\begin{align}
0>Q^{v}_{x_0,\la}(x)&=v_{x_0,\lambda}(x)-v(x)\nonumber\\
&\geq\frac{1}{2\pi}\int_{B_{\lambda,w}^{-}(x_0)}\left(\frac{1}{|x-y|}-\frac{1}
{|\frac{\lambda(x-x_0)}{|x-x_0|}-\frac{|x-x_0|(y-x_0)}{\lambda}|}\right)
(e^{pw_{x_0,\lambda}(y)}-e^{pw(y)})\mathrm{d}y\nonumber\\
&\geq\frac{1}{2\pi}\int_{B_{\lambda,w}^{-}(x_0)}\frac{p}{|x-y|}e^{p\bar{\xi}^w_{x_0,
\lambda}(y)}Q^{w}_{x_0,\la}(y)\mathrm{d}y\nonumber\\
&\geq\frac{p}{2\pi}\int_{B_{\lambda,w}^{-}(x_0)}\frac{e^{pw(y)}}{|x-y|}Q^{w}_{x_0,\la}
(y)\mathrm{d}y,\label{b7}
\end{align}
where $w_{x_0,\lambda}(y)<\bar{\xi}_{x_0,\lambda}^w(y)<w(y)$, for any $y\in B_{\lambda,w}^{-}(x_0)$. Since $\frac{2}{p}\leq\alpha\leq\frac{5}{p}$, from \eqref{a12}, \eqref{a20}, \eqref{a21} and the mean value theorem, we deduce that, for any given $\var>0$ small and $x\in B_{\lambda,w}^{-}(x_0)$,
\begin{align}
0>Q^{w}_{x_0,\la}(x)&=w_{x_0,\lambda}(x)-w(x)\nonumber\\
&\geq\frac{1}{2\pi^2}\int_{B_{\lambda,u}^{-}(x_0)}\ln\left(
\frac{|\frac{\lambda(x-x_0)}{|x-x_0|}-\frac{|x-x_0|(y-x_0)}{\lambda}|}{|x-y|}\right)
(u_{x_0,\lambda}^{3}(x)-u^{3}(x))\mathrm{d}y\nonumber\\
&\geq\frac{3}{4\pi^2}\int_{B_{\lambda,u}^{-}(x_0)\cap B_{\lambda\eta(\varepsilon)}(x_0)}\frac{\lambda^{2\var}}{|x-y|^{2\var}} (\bar{\xi}^u_{x_0,\lambda}(y))^{2} Q^{u}_{x_0,\la}(y)\mathrm{d}y\nonumber\\
&\quad+\frac{3}{4\pi^2} \ln \left(1+\frac{1}{\eta^2(\varepsilon)}\right)\int_{B_{\lambda,u}^{-}(x_0)\backslash B_{\lambda\eta(\varepsilon)}(x_0)}(\bar{\xi}^u_{x_0,\lambda}(y))^{2} Q^{u}_{x_0,\la}(y)\mathrm{d}y\nonumber\\
&\geq\frac{3}{4\pi^2}\int_{B_{\lambda,u}^{-}(x_0)\cap B_{\lambda\eta(\varepsilon)}(x_0)}\frac{\lambda^{2\var}}{|x-y|^{2\var}} u^{2} (y) Q^{u}_{x_0,\la}(y)\mathrm{d}y\nonumber\\
&\quad+C(\eta(\varepsilon))\int_{B_{\lambda,u}^{-}(x_0)\backslash B_{\lambda\eta(\varepsilon)}(x_0)}u^{2} (y) Q^{u}_{x_0,\la}(y)\mathrm{d}y,\label{b8}
\end{align}
where $u_{x_0,\lambda}(y)< \bar{\xi}^u_{x_0,\lambda}(y)< u(y)$ for any $y\in  B_{\lambda,u}^{-}(x_0)$. Take $\var\in(0,\frac{1}{16})$ sufficiently small and $t\in(\frac{3}{2\var},+\infty)$ large enough, from the Hardy-Littlewood-Sobolev inequality \eqref{hl}, H\"older inequality, \eqref{b6}, \eqref{b7} and \eqref{b8}, we obtain
\begin{align}\label{b341}
\|Q_{x_0,\lambda}^u(x)\|_{L^q(B_{\lambda,u}^-(x_0))}
&\leq C\|v^3(x)Q_{x_0,\lambda}^v(x)\|_{L^{\frac{4r}{3r+4}}(B_{\lambda,v}^-(x_0))}\nonumber\\
&\leq C\|v(x)\|^3_{L^{4}(B_{\lambda,v}^-(x_0))}
\|Q_{x_0,\lambda}^v(x)\|_{L^{r}(B_{\lambda,v}^-(x_0))}\nonumber\\
&\leq C\la^{\frac{9}{4}}
\|Q_{x_0,\lambda}^v(x)\|_{L^{r}(B_{\lambda,v}^-(x_0))},
\end{align}
\begin{align}\label{b342}
\|Q_{x_0,\lambda}^v(x)\|_{L^r(B_{\lambda,v}^-(x_0))}& \leq \|e^{pw(x)}Q_{x_0,\lambda}^w(x)\|_{L^{\frac{5t}{4t+5}}
(B_{\lambda,w}^-(x_0))}\nonumber\\
& \leq \|e^{pw(x)}\|_{L^{\frac{5}{4}}
(B_{\lambda,w}^-(x_0))}\|Q_{x_0,\lambda}^w(x)\|_{L^{t}(B_{\lambda,w}^-(x_0))}\nonumber\\
&\leq C\la^{\frac{12}{5}}
\|Q_{x_0,\lambda}^w(x)\|_{L^{t}(B_{\lambda,w}^-(x_0))},
\end{align}
and
\begin{align}\label{b343}
\|Q_{x_0,\lambda}^w(x)\|_{L^t(B_{\lambda,w}^-(x_0))}&\leq C\lambda^{2\var}\|u^2(x) Q_{x_0,\lambda}^u(x)\|_{L^{\frac{3t}{3+(3-2\var)t}}(B_{\lambda,v}^-(x_0))}\nonumber\\
&\quad+C(\eta(\varepsilon))|B_{\lambda,w}^-(x_0)|^{\frac{1}{t}}
\int_{B_{\lambda,u}^-(x_0)}u^2(x) Q_{x_0,\lambda}^u(x)\mathrm{d}x\nonumber\\
&\leq C\lambda^{2\var}\|u(x)\|^2_{L^{\frac{120}{27-40\var}}(B_{\lambda,u}^-(x_0))}
\|Q_{x_0,\lambda}^u(x)\|_{L^{q}(B_{\lambda,u}^-(x_0))}\nonumber\\
&\quad+C(\eta(\varepsilon))\lambda^{\frac{3}{t}}\|u(x)\|^2_{L^{\frac{2q}{q-1}}
(B_{\lambda,u}^-(x_0))}
\|Q_{x_0,\lambda}^u(x)\|_{L^{q}(B_{\lambda,u}^-(x_0))}\nonumber\\
&\leq C\lambda^{\frac{27}{20}}
\|Q_{x_0,\lambda}^u(x)\|_{L^{q}(B_{\lambda,u}^-(x_0))},
\end{align}
where $q=\frac{12r}{12+5r}\in(\frac{120}{71},\frac{20}{11})$, and $r=\frac{15t}{15+2t}\in(\frac{120}{21},\frac{15}{2})$. Combining \eqref{b341}, \eqref{b342} with \eqref{b343}, we have
\begin{align*}
\|Q^{u}_{x_0,\la}(x)\|_{L^{q}(B_{\lambda,u}^{-}(x_0))}\leq C\la^{6}\|Q^{u}_{x_0,\la}(x)\|_{L^{q}(B_{\lambda,u}^{-}(x_0))},
\end{align*}
\begin{align*}
\|Q^{v}_{x_0,\la}(x)\|_{L^{r}(B_{\lambda,v}^{-}(x_0))}\leq C\la^{6}\|Q^{v}_{x_0,\la}(x)\|_{L^{r}(B_{\lambda,v}^{-}(x_0))},
\end{align*}
\begin{align*}
\|Q^{w}_{x_0,\la}(x)\|_{L^{t}(B_{\lambda,w}^{-}(x_0))}\leq C\la^{6}\|Q^{w}_{x_0,\la}(x)\|_{L^{t}(B_{\lambda,w}^{-}(x_0))}.
\end{align*}
Then there exists a $\var_{0}>0$ small enough, such that
\begin{align*}
\|Q^{u}_{x_0,\la}(x)\|_{L^{q}(B_{\lambda,u}^{-}(x_0))}\leq \frac{1}{2}\|Q^{u}_{x_0,\la}(x)\|_{L^{q}(B_{\lambda,u}^{-}(x_0))},
\end{align*}
\begin{align*}
\|Q^{v}_{x_0,\la}(x)\|_{L^{r}(B_{\lambda,v}^{-}(x_0))}\leq \frac{1}{2}\|Q^{v}_{x_0,\la}(x)\|_{L^{r}(B_{\lambda,v}^{-}(x_0))},
\end{align*}
\begin{align*}
\|Q^{w}_{x_0,\la}(x)\|_{L^{t}(B_{\lambda,w}^{-}(x_0))}\leq \frac{1}{2}\|Q^{w}_{x_0,\la}(x)\|_{L^{t}(B_{\lambda,w}^{-}(x_0))}
\end{align*}
for any $0<\lambda\leq \var_{0}$. Therefore, we have proved that, for all $0<\lambda\leq \var_{0}$,
$$\|Q^{u}_{x_0,\la}(x)\|_{L^{q}(B_{\lambda,u}^{-}(x_0))}=
\|Q^{v}_{x_0,\la}(x)\|_{L^{r}(B_{\lambda,v}^{-}(x_0))}
=\|Q^{w}_{x_0,\la}(x)\|_{L^{t}(B_{\lambda,w}^{-}(x_0))}=0,$$
which indicates $B_{\lambda,u}^{-}(x_0)=B_{\lambda,v}^{-}(x_0)=B_{\lambda,w}^{-}(x_0)=\emptyset$. Furthermore,
$$Q^{u}_{x_0,\la}(x)\geq0,\ Q^{v}_{x_0,\la}(x)\geq0,\ Q^{w}_{x_0,\la}(x)\geq0,\qquad \forall x\in B_{\lambda}(x_0)\backslash\{x_0\}.$$
This completes Step 1 for the case $\frac{2}{p}\leq\alpha\leq\frac{5}{p}$.

\medskip

$\mathbf{Step\ 2}.$ Move the sphere $S_{\lambda}$ outward or inward to the limiting position.

In what follows, we will derive contradictions in both the cases $\alpha>\frac{5}{p}$ and $\frac{2}{p}\leq\alpha<\frac{5}{p}$, and hence we must have $\alpha=\frac{5}{p}$.

\smallskip

$\mathbf{Case\ 1:}$ $\alpha>\frac{5}{p}$. Step 1 provides a starting point in the method of moving spheres. Next, we will continue to reduce the radius $\la$ as long as \eqref{a16} still holds. The critical scale $\lambda_{x_0}$ is defined by
\begin{equation}\label{c11}
\lambda_{x_0}:=\inf\{\lambda>0\mid Q_{x_0,\mu}^{u}\leq0,\ Q_{x_0,\mu}^{v}\leq0,\ Q_{x_0,\mu}^{w}\leq0,\ \text{in}\ B_{\mu}(x_0)\backslash\{x_0\},\ \forall \ \lambda\leq\mu<+\infty\}.
\end{equation}

According to Step 1, we know $0\leq\lambda_{x_0}<+\infty$ is well-defined. We first prove that in the case of $\alpha>\frac{5}{p}$, it must have $\lambda_{x_0}=0$. Otherwise, assume $\lambda_{x_0}>0$, we will show that the sphere can be moved inward a little further, which contradicts the definition of $\lambda_{x_0}$.

From the definition of $\lambda_{x_0}$, we have
$$Q_{x_0,\lambda_{x_0}}^{u}\leq0,\ Q_{x_0,\lambda_{x_0}}^{v}\leq0,\ Q_{x_0,\lambda_{x_0}}^{w}\leq0,\quad\ \text{in}\ B_{\lambda_{x_0}}(x_0)\backslash\{x_0\}.$$
By \eqref{a12}, for any $x\in B_{\lambda_{x_0}}(x_0)\backslash\{x_0\}$, it holds

\begin{equation*}
\begin{aligned}
Q_{x_0,\lambda_{x_0}}^{w}(x)&=\frac{1}{2\pi^2}\int_{B_{\lambda_{x_0}}(x_0)}\ln
\left(\frac{|\frac{\la_{x_0}(x-x_0)}{|x-x_0|}-\frac{|x-x_0|(y-x_0)}{\la_{x_0}}|}{|x-y|}
\right)
(u_{x_0,\la_{x_0}}^{3}(y)-u^{3}(y))\mathrm{d}y\\
&\quad+\left(\frac{5}{p}-\alpha\right)\ln\frac{\la_{x_0}}{|x-x_0|}\\
&\leq\lr\frac{5}{p}-\alpha\rr\ln\frac{\lambda_{x_0}}{|x-x_0|}<0.
\end{aligned}
\end{equation*}
Combing with \eqref{a10} and \eqref{a11}, it implies that
\begin{equation*}
\begin{aligned}
Q_{x_0,\lambda_{x_0}}^{v}(x)&=\frac{1}{4\pi}\int_{B_{\la_{x_0}}(x_0)}\left(\frac{1}{|x-y|}
-\frac{1}{|\frac{\la_{x_0}(x-x_0)}{|x-x_0|}-\frac{|x-x_0|(y-x_0)}{\la_{x_0}}|}\right)
(e^{pw_{x_0,\la_{x_0}}(y)}
-e^{pw(y)})\mathrm{d}y
<0
\end{aligned}
\end{equation*}
and
\begin{equation*}
\begin{aligned}
Q_{x_0,\lambda_{x_0}}^{u}(x)
&=\frac{1}{2\pi^2}\int_{B_{\la_{x_0}}(x_0)}\left(\frac{1}{|x-y|^{2}}
-\frac{1}{|\frac{\la_{x_0}(x-x_0)}{|x-x_0|}-\frac{|x-x_0|(y-x_0)}{\la_{x_0}}|^{2}}\right)
(v_{x_0,\la_{x_0}}^{4}(y)-v^{4}(y))\mathrm{d}y
<0.
\end{aligned}
\end{equation*}

Choose a small enough constant $\delta_{1}>0$, which will be determined later. Denote the narrow region
\begin{equation}\label{c5}
A_{\delta_{1}}:=\{x\in \mathbb{R}^{3}\mid  0<|x-x_0|<\delta_{1}\ \text{or}\ \lambda_{x_0}-\delta_{1}<|x-x_0|<\lambda_{x_0}\}\subset B_{\lambda_{x_0}}(x_0)\backslash\{x_0\}.
\end{equation}
Since $Q_{x_0,\lambda_{x_0}}^{u},\ Q_{x_0,\lambda_{x_0}}^{v}$ and $Q_{x_0,\lambda_{x_0}}^{w}$ are continuous w.r.t. $x$ in $\mathbb{R}^{3}\backslash\{x_0\}$ and $A_{\delta_{1}}^{c}:=(B_{\lambda_{x_0}}(x_0)\backslash\{x_0\})\backslash A_{\delta_{1}}$ is a compact subset, there exists a positive constant $c_{0}>0$ such that
$$Q_{x_0,\lambda_{x_0}}^{u}<-c_{0},\ Q_{x_0,\lambda_{x_0}}^{v}<-c_{0},\ Q_{x_0,\lambda_{x_0}}^{w}<-c_{0},\quad\ \forall\ x\in A_{\delta_{1}}^{c}.$$
Due to the continuity of $Q_{x_0,\lambda_{x_0}}^{u},\ Q_{x_0,\lambda_{x_0}}^{v}$ and $Q_{x_0,\lambda_{x_0}}^{w}$ w.r.t. $\la$, choosing $\delta_{2}<\delta_1$ small enough such that, for any $\lambda\in[\lambda_{x_0}-\delta_{2},\lambda_{x_0}]$, we have
$$Q^{u}_{x_0,\la}<-\frac{c_{0}}{2},\ Q^{v}_{x_0,\la}<-\frac{c_{0}}{2},\ Q^{w}_{x_0,\la}<-\frac{c_{0}}{2},\quad\ \forall\ x\in A_{\delta_{1}}^{c}.$$
Hence we must have
\begin{equation}\label{c6}
\begin{aligned}
B_{\lambda,u}^{+}\cup B_{\lambda,v}^{+}\cup B_{\lambda,w}^{+}(x_0)&\subset(B_{\lambda}(x_0)\backslash\{x_0\})\backslash A_{\delta_{1}}^{c}\\
&=
\{x\in \mathbb{R}^{3}\mid  0<|x-x_0|<\delta_{1}\ \text{or}\ \lambda_{x_0}-\delta_{1}<|x-x_0|<\lambda\},
\end{aligned}
\end{equation}
for any $\lambda\in[\lambda_{x_0}-\delta_{2},\lambda_{x_0}]$. By \eqref{a23}, \eqref{a24}, \eqref{a25}, we obtain, for $\var\in(0,\frac{1}{16})$ sufficiently small and $t\in(\frac{3}{2\var},+\infty)$ large enough,
\begin{align}\label{c7}
\|Q^{u}_{x_0,\la}(x)\|_{L^{q}(B_{\lambda,u}^{+}(x_0))}\leq& C_{\var,\la}\|v_{x_0,\lambda}(x)\|^3_{L^{4}(B_{\lambda,v}^+(x_0))}
\|e^{pw_{x_0,\lambda}(x)}\|_{L^{\frac{5}{4}}
(B_{\lambda,w}^+(x_0))}\|Q_{x_0,\lambda}^u(x)\|_{L^{q}(B_{\lambda,u}^+(x_0))}\nonumber\\
&\quad\times \lr\|u_{x_0,\lambda}(x)\|^2_{L^{\frac{120}{27-40\var}}(B_{\lambda,u}^+(x_0))}
+|B_{\lambda,w}^+(x_0)|^{\frac{1}{t}}\|u_{x_0,\lambda}(x)\|^2_{L^{\frac{2q}{q-1}}
(B_{\lambda,u}^+(x_0))}\rr,
\end{align}
\begin{align}\label{341d}
\|Q^{v}_{x_0,\la}(x)\|_{L^{r}(B_{\lambda,v}^{+}(x_0))}\leq& C_{\var,\la}\|v_{x_0,\lambda}(x)\|^3_{L^{4}(B_{\lambda,v}^+(x_0))}
\|e^{pw_{x_0,\lambda}(x)}\|_{L^{\frac{5}{4}}
(B_{\lambda,w}^+(x_0))}\|Q_{x_0,\lambda}^v(x)\|_{L^{r}(B_{\lambda,v}^+(x_0))}\nonumber\\
&\quad\times \lr\|u_{x_0,\lambda}(x)\|^2_{L^{\frac{120}{27-40\var}}(B_{\lambda,u}^+(x_0))}
+|B_{\lambda,w}^+(x_0)|^{\frac{1}{t}}\|u_{x_0,\lambda}(x)\|^2_{L^{\frac{2q}{q-1}}
(B_{\lambda,u}^+(x_0))}\rr,
\end{align}
\begin{align}\label{342d}
\|Q^{w}_{x_0,\la}(x)\|_{L^{t}(B_{\lambda,w}^{+}(x_0))}\leq& C_{\var,\la}\|v_{x_0,\lambda}(x)\|^3_{L^{4}(B_{\lambda,v}^+(x_0))}
\|e^{pw_{x_0,\lambda}(x)}\|_{L^{\frac{5}{4}}
(B_{\lambda,w}^+(x_0))}\|Q_{x_0,\lambda}^w(x)\|_{L^{t}(B_{\lambda,w}^+(x_0))}\nonumber\\
&\quad\times \lr\|u_{x_0,\lambda}(x)\|^2_{L^{\frac{120}{27-40\var}}(B_{\lambda,u}^+(x_0))}
+|B_{\lambda,w}^+(x_0)|^{\frac{1}{t}}\|u_{x_0,\lambda}(x)\|^2_{L^{\frac{2q}{q-1}}
(B_{\lambda,u}^+(x_0))}\rr,
\end{align}
where $q=\frac{12r}{12+5r}$, and $r=\frac{15t}{15+2t}$. From $\alpha\geq\frac{5}{p}$, by asymptotic properties of $(u,v,w)$ in \eqref{e1}, \eqref{e3} and \eqref{e4}, we deduce for $\var\in(0,\frac{1}{16})$ sufficiently small and $t\in(\frac{3}{2\var},+\infty)$ large enough, we get
$v_{x_0,\lambda}\in L^4(\R^3), e^{pw_{x_0,\lambda}}\in L^{\frac{5}{4}}(\R^3), u_{x_0,\lambda}\in L^{\frac{120}{27-40\var}}(\R^3)\cap L^{\frac{2q}{q-1}}(\R^3)$. Thus, from \eqref{c6}, take $\delta_1$ and $\delta_2$ small enough such that, for any $\lambda\in[\lambda_{x_0}-\delta_{2},\lambda_{x_0}]$,
\begin{align*}
&\quad\quad C_{\var,\la}\|v_{x_0,\lambda}(x)\|^3_{L^{4}(B_{\lambda,v}^+(x_0))}
\|e^{pw_{x_0,\lambda}(x)}\|_{L^{\frac{5}{4}}
(B_{\lambda,w}^+(x_0))}\\
&\times\lr\|u_{x_0,\lambda}(x)\|^2_{L^{\frac{120}{27-40\var}}(B_{\lambda,u}^+(x_0))}
+|B_{\lambda,w}^+(x_0)|^{\frac{1}{t}}\|u_{x_0,\lambda}(x)\|^2_{L^{\frac{2q}{q-1}}
(B_{\lambda,u}^+(x_0))}\rr
<\frac{1}{2}.
\end{align*}
From \eqref{c7},  \eqref{341d} and \eqref{342d}, we get that, for any $\lambda\in[\lambda_{x_0}-\delta_{2},\lambda_{x_0}]$,
$$\|Q^{u}_{x_0,\la}(x)\|_{L^{q}(B_{\lambda,u}^{+}(x_0))}
=\|Q^{v}_{x_0,\la}(x)\|_{L^{r}(B_{\lambda,v}^{+}(x_0))}
=\|Q^{w}_{x_0,\la}(x)\|_{L^{t}(B_{\lambda,w}^{+}(x_0))}=0.$$
Thus, for any $\lambda\in[\lambda_{x_0}-\delta_{2},\lambda_{x_0}]$,
$$Q^{u}_{x_0,\la}(x)\leq0,\ Q^{v}_{x_0,\la}(x)\leq0,\ Q^{w}_{x_0,\la}(x)\leq0,\quad\ \forall x\in B_{\lambda}(x_0)\backslash\{x_0\},$$
which contradicts the definition of $\lambda_{x_0}$. Therefore, we must have $\lambda_{x_0}=0$.

In order to derive a contradiction, we need the following calculus Lemma (see Lemmas 11.1 and 11.2 in \cite{LZ1}).

\begin{lem} \label{le10}
Let $n\geq1$, $\nu\in\mathbb{R}$ and $u\in C^{1}(\mathbb{R}^{n})$. For every $x_{0}\in\mathbb{R}^{n}$ and $\lambda>0$, define
$$u_{x_{0},\lambda}(x)=\left(\frac{\lambda}{|x-x_{0}|}\right)^{\nu}u
\left(\frac{\lambda^{2}(x-x_{0})}{|x-x_{0}|^{2}}+x_{0}\right),
\quad\forall x\in \mathbb{R}^{n}\backslash \{x_{0}\}.$$
Then, we have\\
$(i)$ If for every $x_{0}\in \mathbb{R}^{n}$, there exists a $0<\lambda_{x_{0}}<+\infty$ such that
$$u_{x_{0},\lambda_{x_{0}}}(x)=u(x),\qquad \forall\ x\in \mathbb{R}^{n}\backslash \{x_{0}\},$$
then for some $C\in R$, $\mu>0$ and $\bar{x}\in\mathbb{R}^{n}$,
$$u(x)=C\left(\frac{\mu}{1+\mu^{2}|x-\bar{x}|^{2}}\right)^{\frac{\nu}{2}}.$$
$(ii)$ If for every $x_{0}\in \mathbb{R}^{n}$ and any $0<\lambda<+\infty$,
$$u_{x_{0},\lambda}(x)\geq u(x),\qquad \forall\ x\in B_{\lambda}(x_{0})\backslash \{x_{0}\},$$
then $u\equiv C$ for some constant $C\in \mathbb{R}$.
\end{lem}
\begin{rem}\label{rem11}
In Lemma 11.1 and Lemma 11.2 of \cite{LZ1}, Li and Zhang have proved Lemma \ref{le10} for $\nu>0$. Nevertheless, their methods can also be applied to show Lemma \ref{le10} in the cases $\nu\leq0$, see \cite{Li,LZ,Xu}.
\end{rem}

From Lemma \ref{le10} (ii), replacing $u$ by $-u$ therein, we deduce that $u\equiv C_1$ for some constant $C_1$. Since $u^{3}\in L^{1}(\mathbb{R}^{3})$, we must have $u\equiv0$. However, in the system \eqref{a1}, we have $0=e^{pw(x)}>0$ in $\mathbb{R}^{3}$, which is a contradiction. Then $\alpha>\frac{5}{p}$ is impossible.

\smallskip

$\mathbf{Case\ 2:}$ $\frac{2}{p}\leq\alpha<\frac{5}{p}$.
In this case, the critical scale $\la_{x_0}$ is defined by
\begin{equation}\label{c10}
\la_{x_0}:=\sup\{\lambda>0\mid Q_{x_0,\mu}^{u}\geq0,\ Q_{x_0,\mu}^{v}\geq0,\ Q_{x_0,\mu}^{w}\geq0,\ \text{in}\ B_{\mu}(x_0)\backslash\{x_0\},\ \forall \ 0<\mu\leq\lambda\}.
\end{equation}
From step 1 it can be seen that $\la_{x_0}$ is well defined and  $0<\la_{x_0}\leq+\infty$ for any $x_{0}\in\mathbb{R}^{3}$. We will show that $\lambda_{x_0}=+\infty$, which will lead to a contradiction again as in $\mathbf{Case\ 1:}$ $\alpha>\frac{5}{p}$. Otherwise, assuming $\la_{x_0}<\infty$, we will prove that the sphere can be moved outward a bit further, which contradicts the definition of $\la_{x_0}$.

According to the definition of $\la_{x_0}$, we have
$$Q_{x_0,\la_{x_0}}^{u}\geq0,\ Q_{x_0,\la_{x_0}}^{v}\geq0,\ Q_{x_0,\la_{x_0}}^{w}\geq0 \qquad\ \text{in}\ B_{\la_{x_0}}(x_0)\backslash\{x_0\}.$$
We deduce from \eqref{a12} that, for any $x\in B_{\la_{x_0}}(x_0)\backslash\{x_0\}$, it holds
\begin{equation*}
\begin{aligned}
Q_{x_0,\la_{x_0}}^{w}(x)&=\frac{1}{2\pi^2}\int_{B_{\la_{x_0}}(x_0)}
\ln\left(\frac{|\frac{\la_{x_0}(x-x_0)}{|x-x_0|}-\frac{|x-x_0|(y-x_0)}{\la_{x_0}}|}{|x-y|}\right)
(u_{x_0,\la_{x_0}}^{3}(y)-u^{3}(y))\mathrm{d}y\\
&\quad+\left(\frac{5}{p}-\alpha\right)\ln\frac{\la_{x_0}}{|x-x_0|}\\
&\geq\left(\frac{5}{p}-\alpha\right)\ln\frac{\la_{x_0}}{|x-x_0|}>0,
\end{aligned}
\end{equation*}
which together with \eqref{a10} and \eqref{a11} imply that
\begin{equation*}
\begin{aligned}
Q_{x_0,\la_{x_0}}^{v}(x)&=\frac{1}{4\pi}\int_{B_{\la_{x_0}}(x_0)}
\left(\frac{1}{|x-y|}-\frac{1}{|\frac{\la_{x_0}(x-x_0)}{|x-x_0|}-\frac{|x-x_0|(y-x_0)}{\la_{x_0}}|}\right)
(e^{pw_{x_0,\la_{x_0}}(y)}-e^{pw(y)})\mathrm{d}y
>0
\end{aligned}
\end{equation*}
and
\begin{equation*}
\begin{aligned}
Q_{x_0,\la_{x_0}}^{u}(x)
&=\frac{1}{2\pi^2}\int_{B_{\la_{x_0}}(x_0)}
\left(\frac{1}{|x-y|^{2}}-\frac{1}{|\frac{\la_{x_0}(x-x_0)}{|x-x_0|}-\frac{|x-x_0|(y-x_0)}{\la_{x_0}}|^{2}}\right)
(v_{x_0,\la_{x_0}}^{4}(y)-v^{4}(y))\mathrm{d}y
>0.
\end{aligned}
\end{equation*}

Take a sufficiently small constant $\delta_1>0$, which will be determined later. Denote the narrow region $A_{\delta_1}\subset(B_{\la_{x_0}}(x_0)\backslash\{x_0\})$ in \eqref{c5}. Since that $Q_{x_0,\la_{x_0}}^{u},\ Q_{x_0,\la_{x_0}}^{v}$ and $Q_{x_0,\la_{x_0}}^{w}$ are continuous about $x$ in $\mathbb{R}^{3}\backslash\{x_0\}$ and $A_{\delta_1}^{c}:=(B_{\la_{x_0}}(x_0)\backslash\{x_0\})\backslash A_{\delta_1}$ is a compact subset, there exists a constant $c_{1}>0$ such that
$$Q_{x_0,\la_{x_0}}^{u}>c_{1},\ Q_{x_0,\la_{x_0}}^{v}>c_{1},\ Q_{x_0,\la_{x_0}}^{w}>c_{1},\quad\ \forall\ x\in A_{\delta_1}^{c}.$$
By the continuity of $Q_{x_0,\la_{x_0}}^{u},\ Q_{x_0,\la_{x_0}}^{v}$ and $Q_{x_0,\la_{x_0}}^{w}$ w.r.t. $\la$, we can take $\delta_2>0$ small enough such that, for any $\lambda\in[\la_{x_0},\la_{x_0}+\delta_2]$,
$$Q^{u}_{x_0,\la}>\frac{c_{1}}{2},\ Q^{v}_{x_0,\la}>\frac{c_{1}}{2},\ Q^{w}_{x_0,\la}>\frac{c_{1}}{2},\quad\ \forall\ x\in A_{\delta_1}^{c}.$$
Thus we must have
\begin{equation*}
\begin{aligned}
B_{\lambda,u}^{-}\cup B_{\lambda,v}^{-}\cup B_{\lambda,w}^{-}&\subset(B_{\la}(x_0)\backslash\{x_0\})\backslash A_{\delta_1}^{c}\\
&=
\{x\in \mathbb{R}^{3}\mid  0<|x-x_0|<\delta_1\ \text{or}\ \la_{x_0}-\delta_{1}<|x-x_0|<\lambda\},
\end{aligned}
\end{equation*}
for any $\lambda\in[\la_{x_0},\la_{x_0}+\delta_2]$. From \eqref{b341}, \eqref{b342} and \eqref{b343}, we obtain
\begin{align}\label{b345}
\|Q_{x_0,\lambda}^u(x)\|_{L^q(B_{\lambda,u}^-(x_0))}
&\leq C\|v\|^3_{L^{4}(B_{\lambda,v}^-(x_0))}
\|e^{pw(x)}\|_{L^{\frac{5}{4}}
(B_{\lambda,w}^-(x_0))}\|Q_{x_0,\lambda}^u(x)\|_{L^{q}(B_{\lambda,u}^-(x_0))}\nonumber\\
&\quad\times\lr\|u(x)\|^2_{L^{\frac{120}{27-40\var}}(B_{\lambda,u}^-(x_0))}
+|B_{\lambda,w}^-(x_0)|^{\frac{1}{t}}\|u(x)\|^2_{L^{\frac{2q}{q-1}}
(B_{\lambda,u}^-(x_0))}\rr,
\end{align}
\begin{align}\label{b346}
\|Q_{x_0,\lambda}^v(x)\|_{L^r(B_{\lambda,v}^-(x_0))}
&\leq C\|v\|^3_{L^{4}(B_{\lambda,v}^-(x_0))}
\|e^{pw(x)}\|_{L^{\frac{5}{4}}
(B_{\lambda,w}^-(x_0))}\|Q_{x_0,\lambda}^v(x)\|_{L^{r}(B_{\lambda,v}^-(x_0))}\nonumber\\
&\quad\times\lr\|u(x)\|^2_{L^{\frac{120}{27-40\var}}(B_{\lambda,u}^-(x_0))}
+|B_{\lambda,w}^-(x_0)|^{\frac{1}{t}}\|u(x)\|^2_{L^{\frac{2q}{q-1}}
(B_{\lambda,u}^-(x_0))}\rr
\end{align}
and
\begin{align}\label{b347}
\|Q_{x_0,\lambda}^w(x)\|_{L^t(B_{\lambda,w}^-(x_0))}
&\leq C\|v\|^3_{L^{4}(B_{\lambda,v}^-(x_0))}
\|e^{pw(x)}\|_{L^{\frac{5}{4}}
(B_{\lambda,w}^-(x_0))}\|Q_{x_0,\lambda}^w(x)\|_{L^{t}(B_{\lambda,w}^-(x_0))}\nonumber\\
&\quad\times\lr\|u(x)\|^2_{L^{\frac{120}{27-40\var}}(B_{\lambda,u}^-(x_0))}
+|B_{\lambda,w}^-(x_0)|^{\frac{1}{t}}\|u(x)\|^2_{L^{\frac{2q}{q-1}}
(B_{\lambda,u}^-(x_0))}\rr.
\end{align}
From the local boundedness of $u,v$ and $w$, choose $\delta_1$ and $\delta_2$ small enough, we obtain that, for any $\lambda\in[\la_{x_0},\la_{x_0}+\delta_2]$,
\begin{align*}
&\quad \quad C\|v\|^3_{L^{4}(B_{\lambda,v}^-(x_0))}
\|e^{\frac{5}{4}pw(x)}\|_{L^{\frac{5}{4}}
(B_{\lambda,w}^-(x_0))}\\
&\times\lr\|u(x)\|^2_{L^{\frac{120}{27-40\var}}(B_{\lambda,u}^-(x_0))}
+|B_{\lambda,w}^-(x_0)|^{\frac{1}{t}}\|u(x)\|^2_{L^{\frac{2q}{q-1}}
(B_{\lambda,u}^-(x_0))}\rr<\frac{1}{2}.
\end{align*}
Combing with \eqref{b345}, \eqref{b346} and \eqref{b347}, we have, for any $\lambda\in[\la_{x_0},\la_{x_0}+\delta_2]$,
$$\|Q^{u}_{x_0,\la}(x)\|_{L^{q}(B_{\lambda,u}^{-}(x_0))}
=\|Q^{v}_{x_0,\la}(x)\|_{L^{r}(B_{\lambda,v}^{-}(x_0))}
=\|Q^{w}_{x_0,\la}(x)\|_{L^{t}(B_{\lambda,w}^{-}(x_0))}=0$$
and hence
$$Q^{u}_{x_0,\la}(x)\geq0,\ Q^{v}_{x_0,\la}(x)\geq0,\ Q^{w}_{x_0,\la}(x)\geq0,\quad\ \forall x\in B_{\lambda}(x_0)\backslash\{x_0\}.$$
This contradicts the definition of $\la_{x_0}$. Therefore, we must have $\la_{x_0}=+\infty$.

As a consequence, from Lemma \ref{le10} (ii) and $u^{3}\in L^{1}(\mathbb{R}^{3})$, we can also derive $u\equiv0$, which contradicts the first and second equations of the $3$-D system \eqref{a1}. Then $\frac{2}{p}\leq\alpha<\frac{5}{p}$ is also impossible.

\medskip

From the contradictions derived in both cases 1 and 2, we conclude
\begin{equation}\label{348}
\alpha:=\frac{1}{2\pi^2}\int_{\mathbb{R}^{3}}u^{3}(x)\mathrm{d}x=\frac{5}{p}.
\end{equation}
From Step 1, we deduce that, for $\lambda>0$ large, there holds
\begin{equation}\label{d2}
Q^{u}_{x_0,\la}(x)\leq0,\ Q^{v}_{x_0,\la}(x)\leq0,\ Q^{w}_{x_0,\la}(x)\leq0 \qquad\ \text{in}\ B_{\lambda}(x_0)\backslash\{x_0\},
\end{equation}
as well as for $\lambda>0$ small,
\begin{equation}\label{d2+}
Q^{u}_{x_0,\la}(x)\geq0,\ Q^{v}_{x_0,\la}(x)\geq0,\ Q^{w}_{x_0,\la}(x)\geq0 \qquad\ \text{in}\ B_{\lambda}(x_0)\backslash\{x_0\}.
\end{equation}
If the critical scale (defined in \eqref{c10}) $\la_{x_0}<+\infty$, we must obtain $Q_{x_0,\lambda_{x_0}}^{u}(x)=Q_{x_0,\lambda_{x_0}}^{v}(x)=Q_{x_0,\lambda_{x_0}}^{w}(x)=0\ \text{in}\ B_{\lambda_{x_0}}(x_0)\backslash\{x_0\}$. Otherwise, the sphere can also be moved outward a bit further such that \eqref{d2+} still holds (see Case 1 or Case 2 in Step 2), which contradicts the definition \eqref{c10} of $\la_{x_0}$. If the critical scale (defined in \eqref{c10}) $\la_{x_0}=+\infty$, from \eqref{d2}, we get $Q^{u}_{x_0,\la}(x)=Q^{v}_{x_0,\la}(x)=Q^{w}_{x_0,\la}(x)=0\ \text{in}\ B_{\lambda}(x_0)\backslash\{x_0\}$ for $\la$ large enough. As a consequence, for arbitrary $x_{0}\in\mathbb{R}^{3}$, there exists a $\la>0$ depending on $x_0$ such that
$$Q^{u}_{x_0,\la}(x)=Q^{v}_{x_0,\la}(x)=Q^{w}_{x_0,\la}(x)=0,\qquad\forall x\in B_{\lambda}(x_0)\backslash\{x_0\}.$$
Thus, we deduce from Lemma \ref{le10} (i) that, for some $C_{1},C_{2}\in \mathbb{R}$, $\mu>0$ and $x_{0}\in \mathbb{R}^{3}$, $u,v$ must be of the form
\begin{equation}\label{351}u(x)=\frac{C_{1}\mu}{1+\mu^{2}|x-x_{0}|^{2}},\qquad \forall\ x\in \mathbb{R}^{3}\end{equation}
and
\begin{equation}\label{352}v(x)=\lr\frac{C_{2}\mu}{1+\mu^{2}|x-x_{0}|^{2}}\rr^{\frac{1}{2}},\qquad \forall\ x\in \mathbb{R}^{3},\end{equation}
which combines with the system \eqref{a1} and the asymptotic  behavior \eqref{d4} imply
\begin{equation}\label{353}w(x)=\frac{5}{2p}\ln\lr\frac{C_{3}\mu}{1+\mu^{2}|x-x_{0}|
^{2}}\rr,\qquad \forall\ x\in \mathbb{R}^{3}.\end{equation}
where $C_3>0$ is a constant. By direct calculations, \eqref{348} and \eqref{351}, we have
$C_1=\left(\frac{40}{p}\right)^{\frac{1}{3}}$ and then
$$u(x)=\frac{\left(\frac{40}{p}\right)^{\frac{1}{3}}\mu}{1+\mu^{2}|x-x_{0}|^{2}},\qquad \forall\ x\in \mathbb{R}^{3}.$$
From the asymptotic behavior \eqref{e4}, combing with \eqref{352}, we obtain
$C_2=\left(\frac{320}{p}\right)^{\frac{1}{6}}$ and thus
\begin{equation*}
v(x)=\lr\frac{\left(\frac{320}{p}\right)^{\frac{1}{6}}\mu}{1+\mu^{2}|x-x_{0}|^{2}}\rr
^{\frac{1}{2}},\qquad \forall\ x\in \mathbb{R}^{3}.\end{equation*}
By the asymptotic behavior \eqref{e3}, combing with \eqref{353}, we get
$C_3=\left(\frac{405}{p}\right)^{\frac{1}{15}}$ and hence
$$w(x)=\frac{5}{2p}\ln\lr\frac{\left(\frac{405}{p}\right)^{\frac{1}{15}}\mu}
{1+\mu^{2}|x-x_{0}|^{2}}\rr,\qquad \forall\ x\in \mathbb{R}^{3}.$$
This completes our proof of Theorem \ref{thm0}.

\section{Proof of Theorem \ref{thm1}}

In this section, we classify the classical solutions $(u,v,w)$ to the $4$-D system \eqref{eq41} and hence carry out our proof of Theorem \ref{thm1}.

\smallskip

Suppose that $(u,v,w)$ is a pair of classical solution to the system \eqref{eq41} with $u,v\geq0$.

\smallskip

First, we show the following integral representation formula for $v$.
\begin{lem} \label{lemm212}
Suppose that $\int_{\mathbb{R}^{4}}v^{4}(x)\mathrm{d}x<+\infty$. if $(u,v,w)$ is a solution of the problem \eqref{eq41}, then we have, for any $x\in \mathbb{R}^{4}$,
\begin{equation}\label{534q}
v(x)=\frac{1}{4\pi^2}\int_{\mathbb{R}^{4}}\frac{u^2(y)}{|x-y|^2}\mathrm{d}y.
\end{equation}
Furthermore, $v >0$ in $\mathbb{R}^{4}$, $v(x) \geq \frac{C}{|x|^2}$
for some constant $C > 0$ and $|x|$ large enough, and
\begin{equation}\label{535eq}
\int_{\mathbb{R}^{4}}\frac{u^2(x)}{|x|^2}\mathrm{d}x<+\infty.
\end{equation}
\end{lem}
\begin{proof}[{\sl Proof of Lemma~\ref{lemm212}}]
Similar to the proof of Lemma \ref{le2}, by using the maximum principle and Liouville theorem of $-\Delta$, we can conclude
\begin{equation}\label{535q}
v(x)=\frac{1}{4\pi^2}\int_{\mathbb{R}^{4}}\frac{u^2(y)}{|x-y|^2}\mathrm{d}y+C>C\geq0,
\end{equation}
where $C\geq0$ is a constant. From the finite total curvature condition $\int_{\mathbb{R}^{4}}v^{4}(x)\mathrm{d}x<+\infty$ and \eqref{535q}, we deduce $C=0$, which means that \eqref{534q} holds. Let $x=0$ in \eqref{535q}, we get the integrability in \eqref{535eq}, i.e.,
$$\int_{\mathbb{R}^{4}}\frac{u^2(y)}{|y|^2}\mathrm{d}y\leq4\pi^2 v(0)<+\infty.$$
Moreover, by the integral representation formula \eqref{534q} of $v$, for any $|x|$ sufficiently large, we have
\begin{equation*}
\begin{aligned}
v(x)&= \frac{1}{4\pi^2}\int_{\mathbb{R}^{4}}\frac{u^2(y)}{|x-y|^2}\mathrm{d}y \geq\frac{1}{4\pi^2}\int_{1\leq|y|<\frac{|x|}{2}}\frac{u^2(y)}{|x-y|^2}\mathrm{d}y\\
&\geq \frac{1}{9\pi^2|x|^2}\int_{1\leq|y|<\frac{|x|}{2}}\frac{u^2(y)}{|y|^2}\mathrm{d}y\geq \frac{1}{18\pi^2|x|^2}\int_{|y|\geq1}\frac{u^2(y)}{|y|^2}\mathrm{d}y\\
&:=\frac{C}{|x|^2},
\end{aligned}
\end{equation*}
which completes the proof of Lemma \ref{lemm212}.
\end{proof}

Next, we prove the integral representation formula for $u$.
\begin{lem} \label{le51}
Suppose that $\int_{\mathbb{R}^{4}}v^{4}(x)\mathrm{d}x<+\infty$, then we have, for any $x\in \mathbb{R}^{4}$,
\begin{equation}\label{eq51}
u(x)=\frac{1}{4\pi^2}\int_{\mathbb{R}^{4}}\frac{e^{pw(y)}}{|x-y|^{3}}\mathrm{d}y.
\end{equation}
Consequently, $u >0$ in $\mathbb{R}^{4}$, $u(x) \geq \frac{C}{|x|^{3}}$
for some constants $C> 0$ and $|x|$ large enough, and
\begin{equation}\label{eq52}
\int_{\mathbb{R}^{4}}\frac{e^{pw(x)}}{|x|^{3}}\mathrm{d}x<+\infty.
\end{equation}
\end{lem}
\begin{proof}[{\sl Proof of Lemma~\ref{le51}}]
Similar to the proof of Lemma \ref{le1}, by the maximum principle and the Liouville theorem for fractional Laplacian $(-\Delta)^{\frac{1}{2}}$ in Lemmas \ref{lem52} and \ref{lem53}, we can obtain
\begin{equation}\label{eq55}
u(x)=\frac{1}{4\pi^2}\int_{\mathbb{R}^{4}}\frac{e^{pw(y)}}{|x-y|^{3}}\mathrm{d}y+C> C\geq0.
\end{equation}
From the integrability $\int_{\mathbb{R}^{4}}\frac{u^2(x)}{|x|^2}\mathrm{d}x<+\infty$ in \eqref{535eq}, we derive that $C=0$ and hence the integral representation \eqref{eq51}. Taking $x=0$ in \eqref{eq55}, we arrive at
\begin{equation*}
\int_{\mathbb{R}^{4}}\frac{e^{pw(y)}}{|y|^{3}}\mathrm{d}y\leq 4\pi^2 u(0)<+\infty,
\end{equation*}
which gives the desired integrability of $w$ in \eqref{eq52}. Furthermore, by the integral representation formula \eqref{eq51} for $u$, we have, for any $|x|$ sufficiently
large,
\begin{equation*}
\begin{aligned}
u(x)&= \frac{1}{4\pi^2}\int_{\mathbb{R}^{4}}\frac{e^{pw(y)}}{|x-y|^{3}}\mathrm{d}y\geq
\frac{1}{4\pi^2}\int_{1\leq|y|<\frac{|x|}{2}}\frac{e^{pw(y)}}{|x-y|^{3}}\mathrm{d}y\\
&\geq \frac{2}{27|x|^{3}}\int_{1\leq|y|<\frac{|x|}{2}}\frac{e^{pw(y)}}{|y|^{3}}\mathrm{d}y\geq \frac{1}{27|x|^{3}}\int_{|y|\geq1}\frac{e^{pw(y)}}{|y|^{3}}\mathrm{d}y\\
&:=\frac{C}{|x|^{3}},
\end{aligned}
\end{equation*}
which completes the proof of Lemma \ref{le51}.
\end{proof}

Finally, we are to show the integral expression for $w(x)$. To this end, define
\begin{equation}\label{517}
q(x):=\frac{1}{8\pi^2}\int_{\mathbb{R}^{4}}\ln\left(\frac{|x-y|}{|y|}\right)v^{4}(y)\mathrm{d}y,
\end{equation}
and
\begin{equation}\label{e518}
\alpha:=\frac{1}{8\pi^2}\int_{\mathbb{R}^{4}}v^{4}(x)\mathrm{d}x.
\end{equation}
Since $\int_{\mathbb{R}^{4}}v^{4}(x)\mathrm{d}x<+\infty$ and $v\in C^2(\R^4)$, we can see that $q(x)$ and $\alpha$ are well-defined. Furthermore, $q(x)$ solves
\begin{equation}\label{518}
\Delta q(x)=\frac{1}{4\pi^2}\int_{\mathbb{R}^{4}}\frac{v^{4}(y)}{|x-y|^{2}}\mathrm{d}y
\end{equation}
and
\begin{equation*}
\Delta^2 q(x)=-v^4(x).
\end{equation*}
\begin{lem}\label{lem59}
Suppose that $\int_{\mathbb{R}^{4}}v^{4}(x)\mathrm{d}x<+\infty$, then there is a constant $C$ such that
\begin{equation}\label{521}
q(x)\leq\alpha\ln|x|+C.
\end{equation}
\begin{proof}[{\sl Proof of Lemma~\ref{lem59}}]
The proof of this Lemma is similar to that of Lemma \ref{lem1}, we omit it here.
\end{proof}
\end{lem}

Now, we derive the integral representation formula of $\Delta w$.
\begin{lem}\label{lem510}
Suppose that $\int_{\mathbb{R}^{4}}v^{4}(x)\mathrm{d}x<+\infty$, then $\Delta w(x)$ satisfies
\begin{equation}\label{522}
\Delta w(x)=-\frac{1}{4\pi^2}\int_{\mathbb{R}^{4}}\frac{v^{4}(y)}{|x-y|^{2}}\mathrm{d}y-C
\end{equation}
for some constant $C\geq0$.
\end{lem}
\begin{proof}[{\sl Proof of Lemma~\ref{lem510}}]
Let $m(x)=w(x)+q(x)$, by the $4$-D system \eqref{eq41}, we obtain
$$\Delta^{2} m(x)=0,$$
i.e., $\Delta m$ is a harmonic function in $\R^4$. By the mean value theorem of harmonic function, for any $x_0\in\R^4$ and $r>0$, we get
\begin{align}\label{dd417}
\Delta m(x_0)&=\frac{2}{\pi^2r^4}\int_{|y-x_0|\leq r}\Delta m(y)\mathrm{d}y\\
&=\frac{2}{\pi^2r^4}\int_{|y-x_0|= r}\frac{\partial m}{\partial r}\mathrm{d}\sigma\nonumber.
\end{align}
Integrating the equation \eqref{dd417} from $0$ to $r$, we have
\begin{align*}
\frac{r^2}{8}\Delta m(x_0)&=\bbint_{|y-x_0|= r}m(y)\mathrm{d}\sigma-m(x_0).
\end{align*}
Using the Jensen's inequality, we derive
\begin{align}\label{dd418}
e^{\frac{pr^2}{8}\Delta m(x_0)}&=e^{-pm(x_0)}e^{\bbint_{|y-x_0|= r}pm(y)\mathrm{d}\sigma}\\
&\leq e^{-pm(x_0)}\bbint_{|y-x_0|= r}e^{pm(y)}\mathrm{d}\sigma\nonumber.
\end{align}
From the inequality \eqref{521}, we obtain $m(x)=w(x)+q(x)\leq w(x)+\alpha\ln|x|+C$. Combining this with \eqref{dd418} and the integrability $|x|^{-3}e^{pw(x)}\in L^1(\R^4)$ in \eqref{eq52}, we get
$$r^{-p\alpha}e^{\frac{pr^2}{8}\Delta m(x_0)}\in L^1\left([1,+\infty)\right).$$
Thus $\Delta m(x_0)\leq0 \,\, \text{for all} \ x_0\in\R^4$. By Liouville theorem, $\Delta m(x)=-C$ in $\R^4$ for some constant $C\geq0$, which combines with \eqref{518} imply \eqref{522}. This finishes the proof of Lemma \ref{lem510}.
\end{proof}

Furthermore, we can deduce the constant $C$ in \eqref{522} equals to zero.
\begin{lem}\label{lem45}
Assume $w(x)=o(|x|^{2})$ at $\infty$. Then we have
\begin{equation}\label{dd419}
\Delta w(x)=-\frac{1}{4\pi^2}\int_{\mathbb{R}^{4}}\frac{v^{4}(y)}{|x-y|^{2}}\mathrm{d}y.
\end{equation}
\end{lem}
\begin{proof}[{\sl Proof of Lemma~\ref{lem45}}]
First, by Lemma \ref{lem510}, we have
\begin{equation*}
\Delta w(x)=-\frac{1}{4\pi^2}\int_{\mathbb{R}^{4}}\frac{v^{4}(y)}{|x-y|^{2}}\mathrm{d}y-C.
\end{equation*}
If $C=0$, we directly obtain \eqref{dd419}. If $C>0$, take $\var\in(0,\frac{C}{16})$ and $R_{0}$ sufficiently large such that
\begin{align*}
w(y)+\var|y|^2\geq\frac{\var}{2}|y|^2\geq0 \quad\, \text{for}\,\,  |y|\geq R_0 \qquad \text{and}\qquad
\Delta w(x)\leq-C<0 \,\quad \text{in} \,\, \R^4.
\end{align*}
Define
$$n(y)=w(y)+\var|y|^2+D(|y|^{-2}-R_0^{-2}).$$
So one has
\begin{align}\label{dd420}
\Delta n(y)=\Delta w(y)+8\var<-\frac{C}{2}<0 \qquad \text{for} \,\,  |y|\geq R_0.
\end{align}
Furthermore, we have
\begin{align*}
&\lim_{|y|\to\infty}n(y)=+\infty,\quad\text{for any}\ D>0,\\
&\lim_{D\to+\infty}n(y)=-\infty,\quad\text{for any}\ y\in\R^4\backslash\overline{B_{R_0}(0)}.
\end{align*}
Consequently, we can take $D$ sufficiently large such that $\inf\limits_{|y|\geq R_0}n(y)$ is obtained by some $y_0\in\R^4$ with $|y_0|>R_0$. From the maximum principle to \eqref{dd420}, we derive a contradiction which implies $C=0$.
\end{proof}

\smallskip

In order to derive the integral equation for $w$, we need to estimate the asymptotic behavior of $v$, which is one of the main key ingredients in our proof. In fact, by exploiting a bi-harmonic version (in Lemma 2.3 of \cite{Lin}) of the Brezis-Merle's estimate in \cite{BF} and elliptic estimates, we can deduce from either the assumption $(H_{1})$ or the assumption $(H_{2})$ that $v(x)=O\left(|x|^{2K}\right)$ at $\infty$ for some $K\gg1$ arbitrarily large.
\begin{lem} \label{lem46}
Assume $\int_{\mathbb{R}^{4}}v^{4}(x)\mathrm{d}x<+\infty$, then we have\\
(i)\ if $\int_{\mathbb{R}^{4}}e^{\Lambda pw(y)}\mathrm{d}y<+\infty$ for some $\Lambda>0$, then we have $v=O(|x|^{6})$ as $|x|\rightarrow+\infty$; \\
(ii)\ if $u=O(|x|^{K})$ at $\infty$ for some $K\gg1$ arbitrarily large, then we have $v=O(|x|^{2K})$ as $|x|\rightarrow+\infty$.
\end{lem}
Before proving the key Lemma \ref{lem46}, we need to recall a bi-harmonic version of Brezis-Merle type estimate in \cite{BF}, for which refer to \cite{Lin}.
\begin{lem}[\textbf{\cite{Lin}}] \label{lem47}
Suppose that $g(x)$ satisfies
\begin{equation*}
\begin{cases}
\Delta^2g(y)=f(y)\quad &\text{in}\ B_R\subset\R^4,\\
g(y)=\Delta g(y)=0\quad &\text{on}\ \partial B_R,
\end{cases}
\end{equation*}
where $f\in L^1(B_R)$. Then, for any $\delta\in(0,32\pi^2)$, there exists $C_\delta>0$ such that
$$\int_{B_R}e^{\frac{\delta |g|}{\|f\|_{L^1}}}\mathrm{d}x\leq C_\delta R^4.$$
\end{lem}
\begin{proof}[{\sl Proof of Lemma~\ref{lem46}}]
(i)\ If $\int_{\mathbb{R}^{4}}e^{\Lambda pw(y)}\mathrm{d}y<+\infty$ for some $\Lambda>0$, let $g(x)$ solves
\begin{equation*}
\begin{cases}
\Delta^2g(y)=v^4(y)\quad &\text{in}\ B(x,4),\\
g(y)=\Delta g(y)=0\quad &\text{on}\ \partial B(x,4).
\end{cases}
\end{equation*}
By the finite total curvature condition $\int_{\mathbb{R}^{4}}v^{4}(x)\mathrm{d}x<+\infty$, Lemma \ref{lem47} implies that there exists $R\geq2$ large enough such that, $\int_{B(x,4)}e^{7p|g|}\mathrm{d}x\leq C$ for all $x\in \mathbb{R}^{4}\setminus B_R(0)$, where $C$ is a constant independent of $x$.

For $|x|\geq R$, let $h(y)=w(y)-g(y)$ in $B(x,4)$, then $h(y)$ solves
\begin{equation}\label{a422}
\begin{cases}
\Delta^2h(y)=0\quad &\text{in}\ B(x,4),\\
h(y)=w(y),\ \Delta h(y)=\Delta w(y)\quad &\text{on}\ \partial B(x,4).
\end{cases}
\end{equation}
Take $\bar h(y)=-\Delta h(y)$, then $\Delta \bar h(y)=0$ in $B(x,4)$. Thus for any $y\in B(x,2)$, by the mean value theorem and the Harnack inequality of harmonic functions, we obtain from \eqref{a422} that
\begin{equation}\label{423d}
\bar h(y)\leq C_1 \bar h(x)=C_1\bbint_{\partial B(x,4)}\bar h(y)\mathrm{d}\sigma=-C_1\bbint_{\partial B(x,4)}\Delta h(y)\mathrm{d}\sigma=-C_1\bbint_{\partial B(x,4)}\Delta w(y)\mathrm{d}\sigma,
\end{equation}
where $C_1$ is a constant independent of $x$ and $y$. Next, we estimate integral $-\bbint_{\partial B(x,4)}\Delta w(y)\mathrm{d}\sigma$. By $4$-D system \eqref{eq41}, $w(y)$ satisfies
\begin{equation}\label{424d}(-\Delta)^2w(y)=v^4(y)\quad\text{in}\ \R^4.\end{equation}
Integrating equation \eqref{424d} on $B(x,r)$ with $r>0$, we have
\begin{equation}\label{425d}
\int_{\partial B(x,r)}\frac{\partial}{\partial\nu}(\Delta w)(y)\mathrm{d}\sigma=\int_{B(x,r)}v^4(y)\mathrm{d}y,\end{equation}
where $\nu$ denotes the unit outer normal vector on $\partial B(x,r)$. Integrating equation \eqref{425d} w.r.t. $r$ from $0$ to $4$, from the integral representation formula of $\Delta w(x)$ in \eqref{522} in Lemma \ref{lem510}, we get
\begin{align}
-\bbint_{\partial B(x,4)}\Delta w(y)\mathrm{d}\sigma&=-\Delta w(x)-\frac{1}{4\pi^2}\int_{B(x,4)}\lr\frac{1}{|x-y|^2}-\frac{1}{16}\rr v^4(y)\mathrm{d}y\nonumber\\
&=\frac{1}{4\pi^2}\int_{\R^4\setminus B(x,4)}\frac{v^4(y)}{|x-y|^2}\mathrm{d}y+\frac{1}{64\pi^2}
\int_{B(x,4)}v^4(y)\mathrm{d}y+C \nonumber\\
&\leq\frac{1}{64\pi^2}\int_{\R^4}v^4(y)\mathrm{d}y+C\leq C.\label{426d}
\end{align}
Combining \eqref{423d} with \eqref{426d}, we have $\bar h(y)\leq C$ for $y\in B(x,2)$. Note that $h(y)$ satisfies
\begin{equation}\label{a427}
\begin{cases}
\Delta h(y)=-\bar h(y)\quad &\text{in}\ B(x,4),\\
h(y)=w(y)\quad &\text{on}\ \partial B(x,4).
\end{cases}
\end{equation}
By a standard regularity estimate of elliptic equations on \eqref{a427}, we obtain
\begin{equation}\label{428}
\sup_{B(x,1)}h(y)\leq C\lr\|h^+\|_{L^k(B(x,2))}+\|\bar h\|_{L^\infty(B(x,2))}\rr
\end{equation}
for any $k>1$. Then we estimate $\|h^+\|_{L^k(B(x,2))}$ in \eqref{428}.

Since $h=w-g$, we get $h^+\leq w^++|g|$. From the assumption $\int_{\mathbb{R}^{4}}e^{\Lambda pw(x)}\mathrm{d}x<+\infty$ for some $\Lambda>0$, we get
\begin{align}\label{429}
\|h^+\|_{L^k(B(x,2))}&\leq C \|w^+\|_{L^k(B(x,2))}+C\|g\|_{L^k(B(x,2))}
\\
&\leq C\lr\lr\int_{B(x,2)}e^{\Lambda pw(y)}\mathrm{d}y\rr^{\frac{1}{k}}+
\lr\int_{B(x,2)}e^{7p|g|(y)}\mathrm{d}y\rr^{\frac{1}{k}}\rr
\nonumber\\
&\leq C \nonumber
\end{align}
for some constant $C>0$ independent of $x$. Then, from \eqref{428} and \eqref{429}, we obtain $\sup\limits_{B(x,1)}h(y)\leq C$. From $w=h+g$, we infer that
\begin{align}\label{430}
\int_{B(x,1)}e^{7pw(y)}\mathrm{d}y\leq C\int_{B(x,1)}e^{7p|g|(y)}\mathrm{d}y\leq C.
\end{align}
By the integral representation formula of $u(x)$ in \eqref{eq51} and the integrability $|x|^{-3}e^{pw}\in L^{1}(\mathbb{R}^{4})$ in \eqref{eq52}, we obtain from \eqref{430} that, for any $|x|\geq R$,
\begin{align}\label{431}
u(x)&=\frac{1}{4\pi^2}\int_{\mathbb{R}^{4}}\frac{e^{pw(y)}}{|x-y|^{3}}\mathrm{d}y=\frac{1}{4\pi^2}\int_{|y-x|\geq1}\frac{e^{pw(y)}}{|x-y|^{3}}\mathrm{d}y
+\frac{1}{4\pi^2}\int_{B(x,1)}\frac{e^{pw(y)}}{|x-y|^{3}}\mathrm{d}y\\
&\leq\frac{1}{4\pi^2}\int_{|y-x|\geq1}\frac{e^{pw(y)}}{|y|^{3}}\mathrm{d}y\cdot\max_{|y-x|\geq1}\frac{|y|^{3}}{|x-y|^{3}}\nonumber\\
&\quad +\frac{1}{4\pi^2}\lr\int_{B(x,1)}\frac{1}{|x-y|^{\frac{7}{2}}}\mathrm{d}y\rr^{\frac{6}{7}}
\lr\int_{B(x,1)}e^{7pw(y)}\mathrm{d}y\rr^{\frac{1}{7}}\nonumber\\
&\leq \frac{|x|^{3}}{4\pi^2}\int_{\mathbb{R}^{4}}\frac{e^{pw(y)}}{|y|^{3}}\mathrm{d}y+C=O(|x|^{3}). \nonumber
\end{align}
By the 2nd equation $-\Delta v=u^2$ of $4$-D system and a standard elliptic estimate, from \eqref{431} and $\int_{\mathbb{R}^{4}}v^{4}(x)\mathrm{d}x<+\infty$, we get, for any $|x|\geq R+1$,
\begin{equation*}
\sup_{B(x,\frac{1}{2})}v(y)\leq C\lr\|v\|_{L^4(B(x,1))}+\| u^2\|_{L^\infty(B(x,1))}\rr=O(|x|^{6}).
\end{equation*}
(ii)\ If $u=O(|x|^{K})$ at $\infty$ for some $K\gg1$ arbitrarily large, from the 2nd equation $-\Delta v=u^2$ of $4$-D system, a standard elliptic estimate and the finite total curvature condition $\int_{\mathbb{R}^{4}}v^{4}(x)\mathrm{d}x<+\infty$, we get, for $|x|$ large,
\begin{equation*}
\sup_{B(x,\frac{1}{2})}v(y)\leq C\lr\|v\|_{L^4(B(x,1))}+\| u^2\|_{L^\infty(B(x,1))}\rr=O(|x|^{2K}).
\end{equation*}
This concludes our proof of the key Lemma \ref{lem46}.
\end{proof}
\begin{rem}\label{rem2}
If $\Lambda=1$ in the assumption $(i)$ of Lemma \ref{lem46} (i.e., $\int_{\mathbb{R}^{4}}e^{pw(x)}\mathrm{d}x<+\infty$), then it is clear that the estimate \eqref{431} can be improved and we get, for any $|x|\geq R$, $$u(x)\leq\frac{1}{4\pi^2}\int_{|y-x|\geq1}\frac{e^{pw(y)}}{|x-y|^{3}}\mathrm{d}y+C\leq\frac{1}{4\pi^2}\int_{|y-x|\geq1}e^{pw(y)}\mathrm{d}y+C\leq C.$$
Then it follows from the standard elliptic estimate that $v=O(1)$ (i.e., $v(x)\leq C$).
\end{rem}

\smallskip

Now, we are ready to get the integral representation formula and asymptotic property for $w$.
\begin{lem} \label{lem515}
Assume $\int_{\mathbb{R}^{4}}v^{4}(x)\mathrm{d}x<+\infty$ and $w(x)=o(|x|^{2})$ as $|x|\rightarrow+\infty$. Suppose that \\
either $(H_1)$ $u=O(|x|^{K})$ or $v=O(|x|^{K})$ as $|x|\rightarrow+\infty$ for some $K\gg1$ arbitrarily large, \\
or $(H_2)$ $\int_{\mathbb{R}^{4}}e^{\Lambda pw(y)}\mathrm{d}y<+\infty$ for some $\Lambda>0$, \\
then we have
\begin{equation}\label{541}
w(x)=\frac{1}{8\pi^2}\int_{\mathbb{R}^{4}}\ln\left(\frac{|y|}{|x-y|}\right)v^{4}(y)\mathrm{d}y+C_0,
\end{equation}
where $C_0\in\R$ is a constant. Moreover,
\begin{equation}\label{419}
\lim_{|x|\rightarrow+\infty}\frac{w(x)}{\ln|x|}=-\alpha.
\end{equation}
\end{lem}
\begin{proof}[{\sl Proof of Lemma~\ref{lem515}}]
We will first prove the following asymptotic property:
\begin{equation}\label{420}
\lim_{|x|\rightarrow+\infty}\frac{q(x)}{\ln|x|}=\alpha.
\end{equation}
We only need to show that
\begin{equation*}
\lim_{|x|\rightarrow+\infty}\int_{\mathbb{R}^{4}}\frac{\ln|x-y|-\ln|y|
-\ln|x|}{\ln|x|}v^{4}(y)\mathrm{d}y=0.
\end{equation*}
From $\exp^{L}+L\ln L$ inequality in Lemma \ref{lem212}, we get
\begin{equation}\label{422}
\begin{aligned}
\int_{B_{1}(x)}\ln\left(\frac{1}{|x-y|}\right)v^{4}(y)\mathrm{d}y&\leq\int_{B_{1}(x)}
\frac{1}{|x-y|}\mathrm{d}y+\int_{B_{1}(x)}v^{4}(y)\ln(v^{4}(y)+1)\mathrm{d}y\\
&\leq\frac{2\pi^2}{3}+\left(\max_{|x-y|\leq1}\ln(v^{4}(y)+1)\right)\int_{B_{1}(x)}
v^{4}(y)\mathrm{d}y.
\end{aligned}
\end{equation}
Since Lemma \ref{lem46} implies $v=O(|x|^{2K})$ at $\infty$ for some $K\gg1$ arbitrarily large, we obtain from \eqref{422} and the finite total curvature condition $\int_{\mathbb{R}^{4}}v^{4}(x)\mathrm{d}x<+\infty$ that, for any $|x|\geq e^{2}$ large enough,
\begin{equation}\label{423}
\begin{aligned}
&\quad\Big|\int_{\mathbb{R}^{4}}\frac{\ln|x-y|-\ln|y|-\ln|x|}{\ln|x|}v^{4}(y)\mathrm{d}y\Big|\\
\leq&3\int_{B_{1}(x)}v^{4}(y)\mathrm{d}y+\frac{2\pi^2}{3\ln|x|}+
\frac{O(8K\ln|x|)}{\ln|x|}\int_{B_{1}(x)}v^{4}(y)\mathrm{d}y\\
&+\displaystyle\frac{\max_{|y|\leq\ln|x|}|\ln\frac{|x-y|}{|x|}|}{\ln|x|}
\int_{|y|\leq\ln|x|}v^{4}(y)\mathrm{d}y+\frac{1}{\ln|x|}\int_{|y|<\ln|x|}|\ln|y||v^{4}(y)\mathrm{d}y\\
&+\sup_{|x-y|\geq1,|y|\geq\ln|x|}\frac{|\ln|x-y|-\ln|y|-\ln|x||}{\ln|x|}
\int_{|y|\geq\ln|x|}v^{4}(y)\mathrm{d}y\\
\leq&o_{|x|}(1)+\frac{2\pi^2}{3\ln|x|}+\frac{\ln2}{\ln|x|}\int_{R^{4}}v^{4}(x)\mathrm{d}x+
\frac{1}{\ln|x|}\int_{|y|<1}\ln\left(\frac{1}{|y|}\right)v^{4}(y)\mathrm{d}y\\
&\quad+\frac{\ln(\ln|x|)}{\ln|x|}\int_{R^{3}}v^{4}(y)\mathrm{d}y
+\left(2+\frac{\ln2}{\ln|x|}\right)\int_{|y|\geq\ln|x|}v^{4}(y)\mathrm{d}y=o_{|x|}(1).
\end{aligned}
\end{equation}
Let $|x|\rightarrow+\infty$ in \eqref{423}, we have
$$\lim_{|x|\rightarrow+\infty}\int_{\mathbb{R}^{4}}\frac{\ln|x-y|-\ln|y|
-\ln|x|}{\ln|x|}v^{4}(y)\mathrm{d}y=0,$$
which proves the asymptotic property \eqref{420}. Next, we show \eqref{541}. By \eqref{518} and \eqref{dd419}, we have
$$\Delta\left(q(x)+w(x)\right)=0.$$
From \eqref{420} and the condition $w(x)=o(|x|^{2})$ at $\infty$, we can immediately deduce from Lemma \ref{lemma214} (ii) that, for some constant $a_i\in\R$ ($i=0,1,2,3,4$),
\begin{equation*}
w(x)+q(x)=\sum_{i=1}^{4}a_ix_i+a_0,\quad \forall\ x\in\mathbb{R}^{4}.
\end{equation*}
Since \eqref{eq52} implies $|x|^{-3}e^{pw(x)}=|x|^{-3}e^{-pq(x)}e^{pa_0}e^{p\sum_{i=1}^{4}a_ix_i}\in L^1(\R^4)$, we infer from the asymptotic property \eqref{420} that $a_1=a_2=a_3=a_4=0$. Hence the integral representation formula \eqref{541} for $w$ holds. The asymptotic property \eqref{419} of $w$ follows immediately from \eqref{541} and \eqref{420}. This completes our proof of Lemma \ref{lem515}.
\end{proof}

As a consequence of Lemma \ref{lem515}, we have the following Corollary.
\begin{cor} \label{lem58}
Assume $\int_{\mathbb{R}^{4}}v^{4}(x)\mathrm{d}x<+\infty$ and $w(x)=o(|x|^{2})$ as $|x|\rightarrow+\infty$. Suppose that \\
either $(H_1)$ $u=O(|x|^{K})$ or $v=O(|x|^{K})$ as $|x|\rightarrow+\infty$ for some $K\gg1$ arbitrarily large, \\
or $(H_2)$ $\int_{\mathbb{R}^{4}}e^{\Lambda pw(y)}\mathrm{d}y<+\infty$ for some $\Lambda>0$. \\
Then we have, for arbitrarily small $\delta>0$,
\begin{equation}\label{514}
\lim_{|x|\rightarrow+\infty}\frac{e^{pw(x)}}{|x|^{-\alpha p-\delta}}=+\infty,\quad
\lim_{|x|\rightarrow+\infty}\frac{e^{pw(x)}}{|x|^{-\alpha p+\delta}}=0.
\end{equation}
Consequently, $\alpha=\frac{1}{8\pi^2}\int_{\mathbb{R}^{4}}v^{4}(x)\mathrm{d}x\geq\frac{1}{p}$. Furthermore, if $\alpha>\frac{4}{p}$, then
\begin{equation}\label{540eq}
\gamma:=\frac{1}{4\pi^2}\int_{\mathbb{R}^{4}}e^{pw(x)}\mathrm{d}x<+\infty
\end{equation}
and
\begin{equation}\label{516}
\lim_{|x|\rightarrow+\infty}|x|^{3}u(x)=\gamma,
\end{equation}
in addition,
\begin{equation}\label{430eq1}
\beta:=\frac{1}{4\pi^2}
\int_{\mathbb{R}^{4}}u^{2}(x)\mathrm{d}x
<+\infty
\end{equation}
and
\begin{equation}\label{542eq}
\lim_{|x|\rightarrow+\infty}|x|^2v(x)=\beta.
\end{equation}
\end{cor}
\begin{proof}[{\sl Proof of Corollary~\ref{lem58}}]
From Lemma \ref{lem515}, we obtain $$w(x)=-\alpha\ln |x|+o(\ln |x|), \qquad \text{as} \,\, |x|\to\infty,$$
which indicates that
\begin{equation}\label{4dasym}
  e^{pw(x)}=e^{-\alpha p\ln |x|+o(\ln |x|)}=|x|^{-\alpha p}e^{o(\ln |x|)}, \qquad \text{as} \,\, |x|\to\infty,
\end{equation}
where $\alpha$ is defined in \eqref{e518}. Thus, for arbitrarily small $\delta>0$, the asymptotic property \eqref{514} follows immediately from \eqref{4dasym}. Moreover, by the integrability $\int_{\mathbb{R}^{4}}\frac{e^{pw(x)}}{|x|^{3}}\mathrm{d}x<+\infty$ in \eqref{eq52}, we derive $\alpha\geq\frac{1}{p}$.

If $\alpha>\frac{4}{p}$, from the asymptotic property \eqref{514}, it follows immediately that $\gamma:=\int_{\mathbb{R}^{4}}e^{pw(x)}\mathrm{d}x<+\infty$. Now we aim to show the asymptotic property \eqref{516} for $u$. Let $\delta=\frac{\alpha p-4}{2}$, according to \eqref{514}, there exists a $R_0\geq1$ large enough such that
\begin{equation}\label{541eq}
e^{pw(x)}\leq|x|^{-\frac{\alpha p+4}{2}},\qquad\forall \,\, |x|\geq R_0.
\end{equation}
In order to get \eqref{516}, by the integral representation formula \eqref{eq51} of $u$, we only need to show
\begin{equation*}
\lim_{|x|\rightarrow+\infty}\int_{\R^4}\frac{|x|^3-|x-y|^3}{|x-y|^3}e^{pw(y)}\mathrm{d}y=0.
\end{equation*}
Indeed, from \eqref{540eq} and \eqref{541eq}, we obtain, for any $|x|\geq R_0$,
\begin{align*}
&\quad\left|\int_{\R^4}\frac{|x|^3-|x-y|^3}{|x-y|^3}e^{pw(y)}\mathrm{d}y\right|\\
&\leq\int_{|y-x|<\frac{|x|}{2}}\frac{1}{|x-y|^3|y|^{\frac{\alpha p-2}{2}}}\mathrm{d}y
+9\int_{\{|y-x|\geq\frac{|x|}{2}\}\cap\{|y|\geq\frac{|x|}{2}\}}
e^{pw(y)}\mathrm{d}y\nonumber\\
&\quad+\frac{8}{|x|^3}\int_{|y|<R_0}|y|^3e^{pw(y)}\mathrm{d}y
+\frac{8}{|x|^3}\int_{R_0\leq |y|<\frac{|x|}{2}}\frac{1}{|y|^{\frac{\alpha p-2}{2}}}\mathrm{d}y\nonumber\\
&\leq\frac{2^{\frac{\alpha p}{2}}\pi^2}{(\alpha p+2)|x|^{\frac{\alpha p-4}{2}}}+o_{|x|}(1)+\frac{8}{|x|^3}\int_{|y|<R_0}|y|^3e^{pw(y)}\mathrm{d}y
+\frac{16\pi^2}{|x|^3}
\eta(x)=o_{|x|}(1),\nonumber
\end{align*}
where $\eta(x)=\frac{2}{10-\alpha p}\left(\frac{|x|}{2}\right)^{\frac{10-\alpha p}{2}},\
\text{if} \ 4<\alpha p<10,\ \eta(x)=\ln\left(\frac{|x|}{2}\right)\
\text{if} \ \alpha p=10,\ \eta(x)=\frac{2}{\alpha p-10}R_0^{\frac{\alpha p-10}{2}},\\
\text{if} \ \alpha p>10$. Thus we obtain the asymptotic property \eqref{516} for $u$. From \eqref{516}, we get
\begin{equation}\label{431eq}
u(x)\leq\frac{C}{|x|^3},\qquad\forall \,\, |x|\geq R_0.
\end{equation}
Thus, one has $\beta:=\frac{1}{4\pi^2}\int_{\mathbb{R}^{4}}u^{2}(x)\mathrm{d}x<+\infty$.

At the end, from the integral representation formula \eqref{534q} of $v$, we can see that, the asymptotic property \eqref{542eq} of $v$ is equivalent to
\begin{equation}\label{546eq}
\lim_{|x|\rightarrow+\infty}\int_{\R^4}\frac{|x|^2-|x-y|^2}{|x-y|^2}u^{2}(y)\mathrm{d}y=0.
\end{equation}
We only need to show \eqref{546eq}. Indeed, from \eqref{430eq1} and \eqref{431eq}, we get, for any $|x|\geq R_0$,
\begin{align*}
&\quad\left|\int_{\R^4}\frac{|x|^2-|x-y|^2}{|x-y|^2}u^{2}(y)\mathrm{d}y\right|\\
&\leq\int_{|y-x|<\frac{|x|}{2}}\frac{1}{|x-y|^2|y|^{4}}\mathrm{d}y
+5\int_{\{|y-x|\geq\frac{|x|}{2}\}\cap\{|y|\geq\frac{|x|}{2}\}}
u^{2}(y)\mathrm{d}y\nonumber\\
&\quad+\frac{4}{|x|^2}\int_{|y|<R_0}|y|^2u^{2}(y)\mathrm{d}y
+\frac{4}{|x|^2}\int_{R_0\leq |y|<\frac{|x|}{2}}\frac{1}{|y|^{4}}\mathrm{d}y\nonumber\\
&\leq\frac{4\pi^2}{|x|^2}+o_{|x|}(1)+\frac{4}{|x|^2}\int_{|y|<R_0}|y|^2u^{2}(y)\mathrm{d}y
+\frac{8\pi^2}{|x|^2}
\ln\left(\frac{|x|}{2}\right)=o_{|x|}(1).\nonumber
\end{align*}
This proves the asymptotic property \eqref{542eq} and hence concludes our proof of Corollary \ref{lem58}.
\end{proof}

\smallskip

By Lemma \ref{lemm212}, \ref{le51} and \ref{lem515}, we have proved that the classical solution $(u,v,w)$ of the systems \eqref{eq41} solve the following integral system:
\begin{equation}\label{a2}
\begin{cases}\displaystyle
\ u(x)=\frac{1}{4\pi^2}\int_{\mathbb{R}^{4}}\frac{e^{pw(y)}}{|x-y|^{3}}\mathrm{d}y,\\
\displaystyle \ v(x)=\frac{1}{4\pi^2}\int_{\mathbb{R}^{4}}\frac{u^2(y)}{|x-y|^2}\mathrm{d}y,\\
\displaystyle \ w(x)=\frac{1}{8\pi^2}\int_{\mathbb{R}^{4}}\ln\left(\frac{|y|}{|x-y|}\right)v^{4}(y)\mathrm{d}y
+C_0,
\end{cases}
\end{equation}
where $C_0\in \R$. Next, we will apply the method of moving spheres to show that $\alpha=\frac{5}{p}$ and derive the classification of $(u,v,w)$.

\medskip

To this end, for arbitrarily given $x_0\in\R^4$ and any $\lambda>0$, we define the Kelvin transforms of $(u,v,w)$ centered at $x_0$ by
\begin{align}\label{eq43}
\begin{cases}
u_{x_0,\lambda}(x):=\left(\frac{\lambda}{|x-x_0|}\right)^3u(x^{x_0,\lambda}),\\
v_{x_0,\lambda}(x):=\left(\frac{\lambda}{|x-x_0|}\right)^2v(x^{x_0,\lambda}),\\
w_{x_0,\lambda}(x):=w(x^{x_0,\lambda})+\frac{5}{p}\ln\frac{\lambda}{|x-x_0|}
\end{cases}
\end{align}
for arbitrary $x\in\R^4\backslash\{x_0\}$, where $x^{x_0,\lambda}:=\frac{\lambda^2(x-x_0)}{|x-x_0|^2}+x_0$.

\smallskip

Now, we will carry out the the method of moving spheres to the $4$-D IE system \eqref{a2} with respect to arbitrarily given point $x_{0}\in\mathbb{R}^{4}$. Let $\lambda>0$ be an arbitrary positive real number. We start moving the sphere $S_{\lambda}(x_{0}):=\{x\in\mathbb{R}^{4}\mid \, |x-x_{0}|=\lambda\}$ from near $\lambda=0$ or $\lambda=+\infty$, until its limiting position. Therefore, the moving sphere process can be divided into two steps.

\smallskip

In what follows, the two different cases $\alpha\geq\frac{5}{p}$ and $\frac{1}{p}\leq \alpha\leq\frac{5}{p}$ will be discussed separately.

\medskip

\textbf{Step 1:} Moving the sphere $S_\la(x_0)$ from $\la=0$ or $\la=+\infty$. Define
$$Q_{x_0,\la}^u(x):=u_{x_0,\la}(x)-u(x),\quad Q_{x_0,\la}^v(x):=v_{x_0,\la}(x)-v(x),\quad Q_{x_0,\la}^w(x):=w_{x_0,\la}(x)-w(x)$$
and let
\begin{align*}
B_{\la,u}^+(x_0):=\{x\in B_\la(x_0)\backslash\{x_0\}\ |Q_{x_0,\la}^u(x)>0\},\\
B_{\la,v}^+(x_0):=\{x\in B_\la(x_0)\backslash\{x_0\}\ |Q_{x_0,\la}^v(x)>0\},\\
B_{\la,w}^+(x_0):=\{x\in B_\la(x_0)\backslash\{x_0\}\ |Q_{x_0,\la}^w(x)>0\}.
\end{align*}

\smallskip

\textbf{Case 1:} $\alpha\geq\frac{5}{p}$. We will prove that, for $\la>0$ large enough,
\begin{align}\label{66}
Q_{x_0,\la}^u(x)\leq0,\quad Q_{x_0,\la}^v(x)\leq0 \quad \text{and}\ Q_{x_0,\la}^w(x)\leq0.\end{align}
This is equal to prove for $\la>0$ large enough,
\begin{align*}
B_{\la,u}^+(x_0)=B_{\la,v}^+(x_0)=B_{\la,w}^+(x_0)=\emptyset, \quad\quad\forall x\in
B_{\lambda}(x_0)\backslash\{x_0\}.
\end{align*}

We start moving the sphere $S_\la(x_0)$ from near $\la=+\infty$ towards the point $x_0$ such that \eqref{66} holds. Since $(u,v,w)$ solve the integral system \eqref{a2}, for any $\lambda>0$ and all $x\in\R^4$, we obtain
\begin{align}\label{eq68}
u(x)&=\frac{1}{4\pi^2}\int_{\mathbb{R}^{4}}\frac{e^{pw(y)}}{|x-y|^{3}}\mathrm{d}y\nonumber\\
&=\frac{1}{4\pi^2}\int_{B_\lambda(x_0)}\frac{e^{pw(y)}}{|x-y|^{3}}\mathrm{d}y+
\frac{1}{4\pi^2}\int_{B_\lambda(x_0)}\frac{e^{pw(y^{x_0,\lambda})}}{|x-y^{x_0,\lambda}|^{3}}\left(\frac{\lambda}{|y-x_0|}\right)^8\mathrm{d}y\nonumber\\
&=\frac{1}{4\pi^2}\int_{B_\lambda(x_0)}\frac{e^{pw(y)}}{|x-y|^{3}}\mathrm{d}y
+\frac{1}{4\pi^2}\int_{B_\lambda(x_0)}\frac{e^{pw_{x_0,\lambda}(y)}}{\left|\frac{\lambda(x-x_0)}{|x-x_0|}-\frac{|x-x_0|(y-x_0)}{\lambda}\right|^{3}}\mathrm{d}y,
\end{align}

\begin{align}\label{eq69}
v(x)&=\frac{1}{4\pi^2}\int_{\mathbb{R}^{4}}\frac{u^2(y)}{|x-y|^{2}}\mathrm{d}y\nonumber\\
&=\frac{1}{4\pi^2}\int_{B_\lambda(x_0)}\frac{u^2(y)}{|x-y|^{2}}\mathrm{d}y+\frac{1}{4\pi^2}\int_{B_\lambda(x_0)}\frac{u^2(y^{x_0,\lambda})}{|x-y^{x_0,\lambda}|^{2}}\left(\frac{\lambda}{|y-x_0|}\right)^8\mathrm{d}y\nonumber\\
&=\frac{1}{4\pi^2}\int_{B_\lambda(x_0)}\frac{u^2(y)}{|x-y|^{2}}\mathrm{d}y+
\frac{1}{4\pi^2}\int_{B_\lambda(x_0)}\frac{u^2_{x_0,\lambda}(y)}{\left|\frac{\lambda(x-x_0)}{|x-x_0|}-\frac{|x-x_0|(y-x_0)}{\lambda}\right|^{2}}\mathrm{d}y
\end{align}
and

\begin{align}\label{eq610}
w(x)&=\frac{1}{8\pi^2}\int_{\mathbb{R}^{4}}\ln\left(\frac{|y|}{|x-y|}\right)v^4(y)\mathrm{d}y+C_0\nonumber\\
&=\frac{1}{8\pi^2}\int_{B_\lambda(x_0)}\ln\left(\frac{|y|}{|x-y|}\right)v^4(y)\mathrm{d}y\nonumber\\
&\quad+\frac{1}{8\pi^2}\int_{B_\lambda(x_0)}\ln\left(\frac{|y^{x_0,\lambda}|}{|x-y^{x_0,\lambda}|}\right)v^4(y^{x_0,\lambda})\left(\frac{\lambda}{|y-x_0|}\right)^8\mathrm{d}y+C_0\nonumber\\
&=\frac{1}{8\pi^2}\int_{B_\lambda(x_0)}\ln\left(\frac{|y|}{|x-y|}\right)v^4(y)\mathrm{d}y
+\frac{1}{8\pi^2}\int_{B_\lambda(x_0)}\ln\left(\frac{|y^{x_0,\lambda}|}
{|x-y^{x_0,\lambda}|}\right)v^4_{x_0,\lambda}(y)\mathrm{d}y+C_0.
\end{align}

From \eqref{eq68}, \eqref{eq69} and \eqref{eq610}, we deduce for any $\lambda>0$ and all $x\in\R^4\backslash\{x_0\}$,
\begin{align}\label{eq611}
u_{x_0,\lambda}(x)&=\frac{1}{4\pi^2}\left(\frac{\lambda}{|x-x_0|}\right)^3
\int_{\mathbb{R}^{4}}\frac{e^{pw(y)}}{|x^{x_0,\lambda}-y|^{3}}\mathrm{d}y\nonumber\\
&=\frac{1}{4\pi^2}\left(\frac{\lambda}{|x-x_0|}\right)^3\left(\int_{B_\lambda(x_0)}\frac{e^{pw(y)}}{|x^{x_0,\lambda}-y|^{3}}\mathrm{d}y+\int_{B_\lambda(x_0)}\frac{e^{pw(y^{x_0,\lambda})}}{|x^{x_0,\lambda}-y^{x_0,\lambda}|^{3}}\left(\frac{\lambda}{|y-x_0|}\right)^8\mathrm{d}y\right)\nonumber\\
&=\frac{1}{4\pi^2}\int_{B_\lambda(x_0)}\frac{e^{pw(y)}}
{\left|\frac{\lambda(x-x_0)}{|x-x_0|}-\frac{|x-x_0|(y-x_0)}{\lambda}\right|^{3}}\mathrm{d}y+\frac{1}{4\pi^2}\int_{B_\lambda(x_0)}\frac{e^{pw_{x_0,\lambda}(y)}}{|x-y|^{3}}\mathrm{d}y,
\end{align}
\begin{align}\label{eq612}
v_{x_0,\lambda}(x)&=\frac{1}{4\pi^2}\left(\frac{\lambda}{|x-x_0|}\right)^2\int_{\mathbb{R}^{4}}\frac{u^2(y)}{|x^{x_0,\lambda}-y|^{2}}\mathrm{d}y\nonumber\\
&=\frac{1}{4\pi^2}\left(\frac{\lambda}{|x-x_0|}\right)^2\left(\int_{B_\lambda(x_0)}\frac{u^2(y)}{|x^{x_0,\lambda}-y|^{2}}\mathrm{d}y+\int_{B_\lambda(x_0)}\frac{u^2(y^{x_0,\lambda})}{|x^{x_0,\lambda}-y^{x_0,\lambda}|^{2}}\left(\frac{\lambda}{|y-x_0|}\right)^8\mathrm{d}y\right)\nonumber\\
&=\frac{1}{4\pi^2}\int_{B_\lambda(x_0)}\frac{u^2(y)}{\left
|\frac{\lambda(x-x_0)}{|x-x_0|}-\frac{|x-x_0|(y-x_0)}{\lambda}\right|^{2}}\mathrm{d}y+\frac{1}{4\pi^2}
\int_{B_\lambda(x_0)}\frac{u^2_{x_0,\lambda}(y)}{|x-y|^{2}}\mathrm{d}y,
\end{align}
\begin{align}\label{eq613}
w_{x_0,\lambda}(x)&=\frac{1}{8\pi^2}\int_{B_\lambda(x_0)}\ln\left(\frac{|y|}{|x^{x_0,\lambda}-y|}
\right)v^4(y)\mathrm{d}y
+\frac{1}{8\pi^2}\int_{B_\lambda(x_0)}\ln\left(\frac{|y^{x_0,\lambda}|}{|x^{x_0,\lambda}-y^{x_0,\lambda}|}\right)
v^4_{x_0,\lambda}(y)\mathrm{d}y\nonumber\\
&\quad+\frac{5}{p}\ln\frac{\lambda}{|x-x_0|}+C_0.
\end{align}

By \eqref{eq611}, \eqref{eq612} and \eqref{eq613}, for any $x\in B_\lambda(x_0)\backslash\{x_0\}$, we have
\begin{align}\label{eq614}
Q_{x_0,\lambda}^u(x)
=\frac{1}{4\pi^2}\int_{B_\lambda(x_0)}\left(\frac{1}{|x-y|^{3}}-\frac{1}{\left|\frac{\lambda(x-x_0)}{|x-x_0|}-\frac{|x-x_0|(y-x_0)}{\lambda}\right|^{3}}\right)\left(e^{pw_{x_0,\lambda}(y)}-e^{pw(y)}\right)\mathrm{d}y,
\end{align}
\begin{align}\label{eq615}
Q_{x_0,\lambda}^v(x)
=\frac{1}{4\pi^2}\int_{B_\lambda(x_0)}\left(\frac{1}{|x-y|^{2}}
-\frac{1}{\left|\frac{\lambda(x-x_0)}{|x-x_0|}-\frac{|x-x_0|(y-x_0)}{\lambda}\right|^{2}}\right)
\left(u^2_{x_0,\lambda}(y)-u^2(y)\right)\mathrm{d}y,
\end{align}
\begin{align}\label{eq616}
Q_{x_0,\lambda}^w(x)
=\frac{1}{8\pi^2}\int_{B_\lambda(x_0)}\ln\left(\frac{|\frac{\lambda(x-x_0)}{|x-x_0|}-\frac{|x-x_0|(y-x_0)}{\lambda}|}{|x-y|}\right)(v^4_{x_0,\lambda}(y)-v^4(y))\mathrm{d}y
+\Big(\frac{5}{p}-\alpha\Big)\ln\frac{\lambda}{|x-x_0|},
\end{align}
where $\alpha$ is defined in \eqref{e518}. From the mean value theorem and \eqref{eq614}, for any $x\in B_{\lambda,u}^+(x_0)$, we obtain
\begin{align}\label{eq414}
0<Q_{x_0,\lambda}^u(x)&\leq\frac{1}{4\pi^2}\int_{B_{\lambda,w}^+(x_0)}\left(\frac{1}{|x-y|^{3}}-\frac{1}{|\frac{\lambda(x-x_0)}{|x-x_0|}-\frac{|x-x_0|(y-x_0)}{\lambda}|^{3}}\right)(e^{pw_{x_0,\lambda}(y)}-e^{pw(y)})\mathrm{d}y\nonumber\\
&\leq\frac{1}{4\pi^2}\int_{B_{\lambda,w}^+(x_0)}\frac{p}{|x-y|^{3}}e^{p\xi_{x_0,\la}^w(y)}Q_{x_0,\lambda}^w(y)\mathrm{d}y\nonumber\\
&\leq\frac{p}{4\pi^2}\int_{B_{\lambda,w}^+(x_0)}\frac{e^{pw_{x_0,\lambda}(y)}}{|x-y|^{3}}Q_{x_0,\lambda}^w(y)\mathrm{d}y,
\end{align}
where $w(y)<\xi_{x_0,\la}^w(y)<w_{x_0,\lambda}(y)$, for any $y\in B_{\lambda,w}^+(x_0)$. Similarly, for any $x\in B_{\lambda,v}^+(x_0)$, we have
\begin{align}\label{eq415}
0<Q_{x_0,\lambda}^v(x)
&\leq\frac{1}{4\pi^2}\int_{B_{\lambda,u}^+(x_0)}\left(\frac{1}{|x-y|^{2}}-\frac{1}{|\frac{\lambda(x-x_0)}{|x-x_0|}-\frac{|x-x_0|(y-x_0)}{\lambda}|^{2}}\right)(u^2_{x_0,\lambda}(y)-u^2(y))\mathrm{d}y\nonumber\\
&\leq\frac{1}{4\pi^2}\int_{B_{\lambda,u}^+(x_0)}\frac{2}{|x-y|^{2}}\xi_{x_0,\la}^uQ_{x_0,\lambda}^u(y)\mathrm{d}y\nonumber\\
&\leq\frac{1}{2\pi^2}\int_{B_{\lambda,u}^+(x_0)}\frac{u_{x_0,\lambda}(y)}{|x-y|^{2}}Q_{x_0,\lambda}^u(y)\mathrm{d}y
\end{align}
where $u(y)<\xi_{x_0,\la}^u(y)< u_{x_0,\lambda}(y)$, for any $y\in B_{\lambda,u}^+(x_0)$. Since $\alpha\geq\frac{5}{p}$, from \eqref{a20}, \eqref{a21}, \eqref{eq616} and the mean value theorem, for any given $\var>0$ small enough and $x\in B_{\lambda,w}^+(x_0)$, we have

\begin{align}\label{eq418}
0<Q_{x_0,\lambda}^w(x)
&\leq\frac{1}{8\pi^2}\int_{B_{\lambda,v}^+(x_0)}\ln\left(\frac{|\frac{\lambda(x-x_0)}{|x-x_0|}-\frac{|x-x_0|(y-x_0)}{\lambda}|}{|x-y|}\right)(v^4_{x_0,\lambda}(y)-v^4(y))\mathrm{d}y\nonumber\\
&\leq\frac{1}{4\pi^2}\int_{B_{\lambda,v}^+(x_0)\cap B_{\lambda\eta(\var)}(x_0)}\frac{\lambda^{2\var}}{|x-y|^{2\var}}(\xi_{x_0,\la}^v(y))^3Q_{x_0,\lambda}^v(y)\mathrm{d}y\nonumber\\
&\quad+\frac{1}{4\pi^2}\ln\left(1+\frac{1}{\eta^2(\var)}\right)\int_{B_{\lambda,v}^+(x_0)
\backslash B_{\lambda\eta(\var)}(x_0)}(\xi_{x_0,\la}^v(y))^3Q_{x_0,\lambda}^v(y)\mathrm{d}y\nonumber\\
&\leq\frac{1}{4\pi^2}\int_{B_{\lambda,v}^+(x_0)\cap B_{\lambda\eta(\varepsilon)}(x_0)}\frac{\lambda^{2\var}}{|x-y|^{2\var}}v^3_{x_0,\lambda}(y)Q_{x_0,\lambda}^v(y)\mathrm{d}y
\nonumber\\&\quad+
C(\eta(\varepsilon))\int_{B_{\lambda,v}^+(x_0)\backslash B_{\lambda\eta(\varepsilon)}(x_0)}v^3_{x_0,\lambda}(y)Q_{x_0,\lambda}^v(y)\mathrm{d}y
\end{align}
where $v(y)<\xi_{x_0,\la}^v(y)< v_{x_0,\lambda}(y)$, for any $y\in B_{\lambda,v}^+(x_0)$. By Hardy-Littlewood-Sobolev inequality \eqref{hl}, H\"{o}lder inequality, \eqref{eq414}, \eqref{eq415} and \eqref{eq418}, we obtain, for any given $\var\in(0,\frac{1}{2})$ small enough to be chosen later and any $t\in(\frac{2}{\var},+\infty)$
\begin{align}\label{eq419}
\|Q_{x_0,\lambda}^u(x)\|_{L^q(B_{\lambda,u}^+(x_0))}
&\leq C\|e^{pw_{x_0,\lambda}(x)}Q_{x_0,\lambda}^w(x)\|_{L^{\frac{6t}{6+5t}}(B_{\lambda,w}^+(x_0))}\nonumber\\
&\leq C\|e^{pw_{x_0,\lambda}(x)}\|_{L^{\frac{6}{5}}(B_{\lambda,w}^+(x_0))}\|Q_{x_0,\lambda}^w(x)\|_{L^{t}(B_{\lambda,w}^+(x_0))}\nonumber\\
&\leq C\la^{\frac{5}{3}}\lr\int_{\R^4}\frac{e^{\frac{6}{5}pw(x)}}{|x-x_0|^2}\mathrm{d}x\rr^{\frac{5}{6}}\|Q_{x_0,\lambda}^w(x)\|_{L^{t}(B_{\lambda,w}^+(x_0))},
\end{align}
\begin{align}\label{eq420}
\|Q_{x_0,\lambda}^v(x)\|_{L^r(B_{\lambda,v}^+(x_0))}& \leq C\|u_{x_0,\lambda}(x) Q_{x_0,\lambda}^u(x)\|_{L^{\frac{3q}{q+3}}(B_{\lambda,u}^+(x_0))}\nonumber\\
& \leq C\|u_{x_0,\lambda}(x)\|_{L^{3}(B_{\lambda,u}^+(x_0))}\|Q_{x_0,\lambda}^u(x)\|_{L^{q}(B_{\lambda,u}^+(x_0))}\nonumber\\
&\leq C\la^{-\frac{1}{3}}\lr\int_{\R^4}|x-x_0|u^3(x)\mathrm{d}x\rr^{\frac{1}{3}}\|Q_{x_0,\lambda}^u(x)\|_{L^{q}(B_{\lambda,u}^+(x_0))},
\end{align}
and
\begin{align}\label{eq421}
&\quad \|Q_{x_0,\lambda}^w(x)\|_{L^t(B_{\lambda,w}^+(x_0))}\nonumber\\
&\leq C\lambda^{2\var}\|v^3_{x_0,\lambda}(x)Q_{x_0,\lambda}^v(x)\|_{L^{\frac{2t}{2+(2-\var)t}}(B_{\lambda,v}^+(x_0))}\nonumber\\
&\quad+C(\eta(\varepsilon))|B_{\lambda,w}^+(x_0)|^{\frac{1}{t}}
\int_{B_{\lambda,v}^+(x_0)}v^3_{x_0,\lambda}(x)Q_{x_0,\lambda}^v(x)\mathrm{d}x\nonumber\\
&\leq C\lambda^{2\var}\|v_{x_0,\lambda}(x)\|^3_{L^{\frac{36}{7-6\var}}(B_{\lambda,v}^+(x_0))}
\|Q_{x_0,\lambda}^v(x)\|_{L^{r}(B_{\lambda,v}^+(x_0))}\nonumber\\
&\quad+C(\eta(\varepsilon))\lambda^{\frac{4}{t}}
\|v_{x_0,\lambda}(x)\|^3_{L^{\frac{3r}{r-1}}(B_{\lambda,v}^+(x_0))}
\|Q_{x_0,\lambda}^v(x)\|_{L^{r}(B_{\lambda,v}^+(x_0))}\nonumber\\
&\leq C\lambda^{-\frac{4}{3}-2\var}\lr\int_{\R^4}|x-x_0|^{\frac{16+48\var}{7-6\var}}v^{\frac{36}{7-6\var}}(x)\mathrm{d}x\rr^{\frac{7-6\var}{12}}
\|Q_{x_0,\lambda}^v(x)\|_{L^{r}(B_{\lambda,v}^+(x_0))}\nonumber\\
&\quad+ C(\eta(\varepsilon))\lambda^{-\frac{4}{3}-\frac{4}{t}}\lr\int_{\R^4}
|x-x_0|^{\frac{16t-96}{7t-12}}v^{\frac{36t}{7t-12}}(x)\mathrm{d}x\rr^{\frac{7t-12}{12t}}
\|Q_{x_0,\lambda}^v(x)\|_{L^{r}(B_{\lambda,v}^+(x_0))},
\end{align}
where $q=\frac{12t}{12+7t}\in(\frac{6}{5},\frac{12}{7}),r=\frac{12t}{12+5t}\in
(\frac{3}{2},\frac{12}{5})$. Since $\alpha\geq\frac{5}{p}$, from asymptotic properties of $(u,v,w)$ in \eqref{514}, \eqref{516} and \eqref{542eq}, we deduce that, for $\var\in(0,\frac{1}{2})$ sufficiently small and $t\in(\frac{2}{\var},+\infty)$ large enough, $\frac{e^{\frac{6}{5}pw}}{|x-x_0|^2}\in L^1(\R^4),|x-x_0|u^3\in L^1(\R^4),|x-x_0|^{\frac{16+48\var}{7-6\var}}v^{\frac{36}{7-6\var}}\in L^1(\R^4),|x-x_0|^{\frac{16t-96}{7t-12}}v^{\frac{36t}{7t-12}}\in L^1(\R^4)$. Thus, by \eqref{eq419}, \eqref{eq420} and \eqref{eq421}, we get
\begin{align*}
\|Q^{u}_{x_0,\la}(x)\|_{L^{q}(B_{\lambda,u}^{+}(x_0))}\leq C_{t,\var}\max\{\la^{-2\var},\la^{-\frac{4}{t}}\}\|Q^{u}_{x_0,\la}(x)\|_{L^{q}(B_{\lambda,u}^{+}(x_0))},
\end{align*}
\begin{align*}
\|Q^{v}_{x_0,\la}(x)\|_{L^{r}(B_{\lambda,v}^{+}(x_0))}\leq C_{t,\var}\max\{\la^{-2\var},\la^{-\frac{4}{t}}\}\|Q^{v}_{x_0,\la}(x)
\|_{L^{r}(B_{\lambda,v}^{+}(x_0))},
\end{align*}
\begin{align*}
\|Q^{w}_{x_0,\la}(x)\|_{L^{t}(B_{\lambda,w}^{+}(x_0))}\leq C_{t,\var}\max\{\la^{-2\var},\la^{-\frac{4}{t}}\}\|Q^{w}_{x_0,\la}(x)\|_{L^{t}(B_{\lambda,w}^{+}(x_0))}.
\end{align*}
Thus choose $\var\in(0,\frac{1}{2})$ small enough and $t\in(\frac{2}{\var},+\infty)$ large enough, there exists a $\Lambda_0>0$ big enough such that
\begin{align}
\|Q_{x_0,\lambda}^u(x)\|_{L^{q}(B_{\lambda,u}^+(x_0))}\leq\frac{1}{2}\|Q_{x_0,\lambda}^u(x)\|_{L^{q}(B_{\lambda,u}^+(x_0))}\label{eq426},\\
\|Q_{x_0,\lambda}^v(x)\|_{L^{r}(B_{\lambda,v}^+(x_0))}\leq\frac{1}{2}\|Q_{x_0,\lambda}^v(x)\|_{L^{r}(B_{\lambda,v}^+(x_0))}\label{eq427},\\
\|Q_{x_0,\lambda}^w(x)\|_{L^{t}(B_{\lambda,w}^+(x_0))}\leq\frac{1}{2}\|Q_{x_0,\lambda}^w(x)\|_{L^{t}(B_{\lambda,w}^+(x_0))}\label{eq428},
\end{align}
for all $\Lambda_0\leq\lambda<+\infty$. By \eqref{eq426}, \eqref{eq427} and \eqref{eq428}, we obtain
$$\|Q_{x_0,\lambda}^u(x)\|_{L^{q}(B_{\lambda,u}^+(x_0))}=\|Q_{x_0,\lambda}^v(x)\|_{L^{r}(B_{\lambda,v}^+(x_0))}=\|Q_{x_0,\lambda}^w(x)\|_{L^{t}(B_{\lambda,w}^+(x_0))}=0,$$
which implies $B_{\lambda,u}^+(x_0)=B_{\lambda,v}^+(x_0)=B_{\lambda,w}^+(x_0)=\emptyset$ for any $\Lambda_0\leq\lambda<+\infty$.

Thus, for all $\Lambda_0\leq\lambda<+\infty$,
$$Q^{u}_{x_0,\la}(x)\leq0,\ Q^{v}_{x_0,\la}(x)\leq0,\ Q^{w}_{x_0,\la}(x)\leq0,\quad \forall x\in B_{\lambda}(x_0)\backslash\{x_0\}.$$
This finishes step 1 for the case $\alpha\geq\frac{5}{p}$.

\smallskip

\textbf{Case 2:} $\frac{1}{p}\leq\alpha\leq\frac{5}{p}$. Denote
\begin{align*}
B_{\la,u}^-(x_0):=\{x\in B_\la(x_0)\backslash\{x_0\}\ |Q_{x_0,\la}^u(x)<0\},\\
B_{\la,v}^-(x_0):=\{x\in B_\la(x_0)\backslash\{x_0\}\ |Q_{x_0,\la}^v(x)<0\},\\
B_{\la,w}^-(x_0):=\{x\in B_\la(x_0)\backslash\{x_0\}\ |Q_{x_0,\la}^w(x)<0\}.
\end{align*} We will deduce that for $\lambda>0$ small enough,
\begin{align}\label{eq429}
Q_{x_0,\la}^u(x)\geq0,\quad Q_{x_0,\la}^v(x)\geq0\quad\text{and}\quad Q_{x_0,\la}^w(x)\geq0,\quad\quad\forall x\in
B_{\lambda}(x_0)\backslash\{x_0\},
\end{align}
which is equal to prove for $\la>0$ small enough,
\begin{align*}
B_{\lambda,u}^-(x_0)=B_{\lambda,v}^-(x_0)=B_{\lambda,w}^-(x_0)=\emptyset.
\end{align*}

We start moving the sphere $S_{\lambda}(x_0)$
from near the point $x_0$ outward such that \eqref{eq429} holds. From \eqref{eq614} and the mean value theorem, for any $x\in B_{\lambda,u}^-(x_0)$, we have
\begin{align}\label{eq431}
0>Q_{x_0,\lambda}^u(x)
&\geq\frac{1}{4\pi^2}\int_{B_{\lambda,w}^-(x_0)}\Bigg[\frac{1}{|x-y|^{3}}-\frac{1}{\left|\frac{\lambda(x-x_0)}{|x-x_0|}-\frac{|x-x_0|(y-x_0)}{\lambda}\right|^{3}}\Bigg]
\left(e^{pw_{x_0,\lambda}(y)}-e^{pw(y)}\right)\mathrm{d}y\nonumber\\
&\geq\frac{1}{4\pi^2}\int_{B_{\lambda,w}^-(x_0)}\frac{p}{|x-y|^{3}}e^{p\bar{\xi}_{x_0,\lambda}^w(y)}Q_{x_0,\lambda}^w(y)\mathrm{d}y\nonumber\\
&\geq\frac{p}{4\pi^2}\int_{B_{\lambda,w}^-(x_0)}\frac{e^{pw(y)}}{|x-y|^{3}}Q_{x_0,\lambda}^w(y)\mathrm{d}y,
\end{align}
where $w_{x_0,\lambda}(y)<\bar{\xi}_{x_0,\lambda}^w(y)<w(y)$, for any $y\in B_{\lambda,w}^-(x_0)$. From \eqref{eq615} and the mean value theorem, for any $x\in B_{\lambda,v}^-(x_0)$, one can deduce
\begin{align}\label{eq432}
0>Q_{x_0,\lambda}^v(x)
&\geq\frac{1}{4\pi^2}\int_{B_{\lambda,u}^-(x_0)}\left(\frac{1}{|x-y|^{2}}-\frac{1}{|\frac{\lambda(x-x_0)}{|x-x_0|}-\frac{|x-x_0|(y-x_0)}{\lambda}|^{2}}\right)
\left(u^2_{x_0,\lambda}(y)-u^2(y)\right)\mathrm{d}y\nonumber\\
&\geq\frac{1}{4\pi^2}\int_{B_{\lambda,u}^-(x_0)}\frac{2}{|x-y|^{2}}\bar{\xi}_{x_0,\lambda}^uQ_\lambda^u(y)\mathrm{d}y\nonumber\\
&\geq\frac{1}{2\pi^2}\int_{B_{\lambda,u}^-(x_0)}\frac{u(y)}{|x-y|^{2}}Q_{x_0,\lambda}^u(y)\mathrm{d}y,
\end{align}
where $u_{x_0,\lambda}(y)<\bar{\xi}_{x_0,\lambda}^u(y)<u(y)$, for any $y\in B_{\lambda,u}^-(x_0)$. Due to $\frac{1}{p}\leq\alpha\leq\frac{5}{p}$, from \eqref{a20}, \eqref{a21}, \eqref{eq616} and by the mean value theorem, for any given $\var>0$ small enough and $x\in B_{\lambda,w}^-(x_0)$, we have
\begin{align}\label{eq433}
0>Q_{x_0,\lambda}^w(x)
&\geq\frac{1}{8\pi^2}\int_{B_{\lambda,v}^-(x_0)}\ln\left(\frac{|\frac{\lambda(x-x_0)}{|x-x_0|}-\frac{|x-x_0|(y-x_0)}{\lambda}|}{|x-y|}\right)(v^4_{x_0,\lambda}(y)-v^4(y))\mathrm{d}y\nonumber\\
&\geq\frac{1}{4\pi^2}\int_{B_{\lambda,v}^-(x_0)\cap B_{\lambda\eta(\var)}(x_0)}\frac{\lambda^{2\var}}{|x-y|^{2\var}}(\bar{\xi}_{x_0,\lambda}^v(y))^3Q_{x_0,\lambda}^v(y)\mathrm{d}y\nonumber\\
&\quad+\frac{1}{4\pi^2}\ln\Big(1+\frac{1}{\eta^2(\var)}\Big)\int_{B_{\lambda,v}^-(x_0)\backslash B_{\lambda\eta(\var)}(x_0)}(\bar{\xi}_{x_0,\lambda}^v(y))^3Q_{x_0,\lambda}^v(y)\mathrm{d}y\nonumber\\
&\geq\frac{1}{4\pi^2}\int_{B_{\lambda,v}^-(x_0)\cap B_{\lambda\eta(\varepsilon)}(x_0)}\frac{\lambda^{2\var}}{|x-y|^{2\var}}v^3(y)Q_{x_0,\lambda}^v(y)\mathrm{d}y
\nonumber\\&\quad+
C(\eta(\varepsilon))\int_{B_{\lambda,v}^+(x_0)\backslash B_{\lambda\eta(\varepsilon)}(x_0)}v^3(y)Q_{x_0,\lambda}^v(y)\mathrm{d}y,
\end{align}
where $v_{x_0,\lambda}(y)<\bar{\xi}_{x_0,\lambda}^v(y)<v(y)$, for any $y\in B_{\lambda,v}^-(x_0)$. Take $\var\in(0,\frac{1}{2})$ small enough and $t\in(\frac{2}{\var},+\infty)$ large enough From Hardy-Littlewood-Sobolev inequality \eqref{hl}, H\"{o}lder inequality, \eqref{eq431}, \eqref{eq432} and \eqref{eq433}, we have
\begin{align}\label{eq639}
\|Q_{x_0,\lambda}^u(x)\|_{L^q(B_{\lambda,u}^-(x_0))}
&\leq C\|e^{pw(x)}Q_{x_0,\lambda}^w(x)\|_{L^{\frac{6t}{6+5t}}(B_{\lambda,w}^-(x_0))}\nonumber\\
&\leq C\|e^{pw(x)}\|_{L^{\frac{6}{5}}(B_{\lambda,w}^-(x_0))}\|Q_{x_0,\lambda}^w(x)
\|_{L^{t}(B_{\lambda,w}^-(x_0))}\nonumber\\
&\leq C\la^{\frac{10}{3}}\|Q_{x_0,\lambda}^w(x)\|_{L^{t}(B_{\lambda,w}^-(x_0))},
\end{align}
\begin{align}\label{eq640}
\|Q_{x_0,\lambda}^v(x)\|_{L^r(B_{\lambda,v}^-(x_0))}& \leq C\|u(x) Q_{x_0,\lambda}^u(x)\|_{L^{\frac{3q}{q+3}}(B_{\lambda,u}^-(x_0))}\nonumber\\
& \leq C\|u(x)\|_{L^{3}(B_{\lambda,u}^-(x_0))}\|Q_{x_0,\lambda}^u(x)\|
_{L^{q}(B_{\lambda,u}^-(x_0))}\nonumber\\
&\leq C\la^{\frac{4}{3}}\|Q_{x_0,\lambda}^u(x)\|_{L^{q}(B_{\lambda,u}^-(x_0))},
\end{align}
and
\begin{align}\label{eq641}
\|Q_{x_0,\lambda}^w(x)\|_{L^t(B_{\lambda,w}^-(x_0))}&\leq C\lambda^{2\var}\|v^3(x)Q_{x_0,\lambda}^v(x)\|
_{L^{\frac{2t}{2+(2-\var)t}}(B_{\lambda,v}^-(x_0))}\nonumber\\
&\quad+C(\eta(\varepsilon))|B_{\lambda,w}^-(x_0)|^{\frac{1}{t}}
\int_{B_{\lambda,v}^-(x_0)}v^3Q_{x_0,\lambda}^v(x)\mathrm{d}x\nonumber\\
&\leq C\lambda^{2\var}\|v(x)\|^3_{L^{\frac{36}{7-6\var}}(B_{\lambda,v}^-(x_0))}
\|Q_{x_0,\lambda}^v(x)\|_{L^{r}(B_{\lambda,v}^-(x_0))}\nonumber\\
&\quad+C(\eta(\varepsilon))\lambda^{\frac{4}{t}}\|v(x)\|^3_{L^{\frac{3r}{r-1}}
(B_{\lambda,v}^-(x_0))}
\|Q_{x_0,\lambda}^v(x)\|_{L^{r}(B_{\lambda,v}^-(x_0))}\nonumber\\
&\leq C\lambda^{\frac{7}{3}}
\|Q_{x_0,\lambda}^v(x)\|_{L^{r}(B_{\lambda,v}^-(x_0))},
\end{align}
where $q=\frac{12t}{12+7t}\in(\frac{6}{5},\frac{12}{7}),r=\frac{12t}{12+5t}\in
(\frac{3}{2},\frac{12}{5})$. Thus, by \eqref{eq639}, \eqref{eq640} and \eqref{eq641}, we have
\begin{align*}
\|Q^{u}_{x_0,\la}(x)\|_{L^{q}(B_{\lambda,u}^{-}(x_0))}\leq C\la^{7}\|Q^{u}_{x_0,\la}(x)\|_{L^{q}(B_{\lambda,u}^{-}(x_0))},
\end{align*}
\begin{align*}
\|Q^{v}_{x_0,\la}(x)\|_{L^{r}(B_{\lambda,v}^{-}(x_0))}\leq C\la^{7}\|Q^{v}_{x_0,\la}(x)\|_{L^{r}(B_{\lambda,v}^{-}(x_0))},
\end{align*}
\begin{align*}
\|Q^{w}_{x_0,\la}(x)\|_{L^{t}(B_{\lambda,w}^{-}(x_0))}\leq C\la^{7}\|Q^{w}_{x_0,\la}(x)\|_{L^{t}(B_{\lambda,w}^{-}(x_0))}.
\end{align*}
Then there exists a $\var_0$ small enough such that
\begin{align*}
\|Q_{x_0,\lambda}^u(x)\|_{L^{q}(B_{\lambda,u}^-(x_0))}\leq\frac{1}{2}\|Q_{x_0,\lambda}^u(x)\|
_{L^{q}(B_{\lambda,u}^-(x_0))},\\
\|Q_{x_0,\lambda}^v(x)\|_{L^{r}(B_{\lambda,v}^-(x_0))}\leq\frac{1}{2}\|Q_{x_0,\lambda}^v(x)\|
_{L^{r}(B_{\lambda,v}^-(x_0))},\\
\|Q_{x_0,\lambda}^w(x)\|_{L^{t}(B_{\lambda,w}^-(x_0))}\leq\frac{1}{2}\|Q_{x_0,\lambda}^w(x)\|
_{L^{t}(B_{\lambda,w}^-(x_0))}
\end{align*}
for all $0<\lambda\leq\var_0$. Furthermore,
$$\|Q_{x_0,\lambda}^u(x)\|_{L^{q}(B_{\lambda,u}^-(x_0))}
=\|Q_{x_0,\lambda}^v(x)\|_{L^{r}(B_{\lambda,v}^-(x_0))}
=\|Q_{x_0,\lambda}^w(x)\|_{L^{t}(B_{\lambda,w}^-(x_0))}=0,$$
which implies $B_{\lambda,u}^-(x_0)=B_{\lambda,v}^-(x_0)=B_{\lambda,w}^-(x_0)=\emptyset$ for any $0<\lambda\leq\var_0$.

Thus, for all $0<\lambda\leq\var_0$,
$$Q^{u}_{x_0,\la}(x)\geq0,\ Q^{v}_{x_0,\la}(x)\geq0,\ Q^{w}_{x_0,\la}(x)\geq0,\quad \forall x\in B_{\lambda}(x_0)\backslash\{x_0\}.$$
This completes step 1 for the case $\frac{1}{p}\leq\alpha\leq\frac{5}{p}$.

\medskip

\textbf{Step 2:} Move the sphere $S_{\lambda}$ outward or inward to the limiting position.

In what follows, we will derive contradictions in both the cases $\alpha>\frac{5}{p}$ and $\frac{1}{p}\leq\alpha<\frac{5}{p}$, and hence we must have $\alpha=\frac{5}{p}$.

\smallskip

\textbf{Case 1:} $\alpha>\frac{5}{p}$. Step 1 provides a starting point for the method of moving sphere. In this step, we will continue to reduce the radius $\la$ as long as \eqref{66} remains true. The critical scale $\lambda_{x_0}$ is defined by
\begin{align}
\lambda_{x_0}:=\inf\{\lambda>0\ |\ Q_{x_0,\mu}^u\leq0,\ Q_{x_0,\mu}^v\leq0,\ Q_{x_0,\mu}^w\leq0,\ \text{in}\ B_\mu(x_0)\backslash\{x_0\},\ \forall \ \lambda\leq\mu<+\infty\}.\label{eq434}
\end{align}
From step 1, we know that $\lambda_{x_0}$ is well defined. First, we deduce that in the case of $\alpha>\frac{5}{p}$, it must hold $\lambda_{x_0}=0$. In other words, if $\lambda_{x_0}>0$, we will show that the sphere $S_{\la_{x_0}}$ can continue to move inward a bit, which contradicts the definition of $\lambda_{x_0}$.

According to the definition \eqref{eq434} of $\lambda_{x_0}$, we have
$$Q_{x_0,\lambda_{x_0}}^u\leq0,\quad Q_{x_0,\lambda_{x_0}}^v\leq0,\quad Q_{x_0,\lambda_{x_0}}^w\leq0,\quad\text{in} \
B_{\lambda_{x_0}}(x_0)\backslash\{x_0\}.$$
From \eqref{eq616}, for any $x\in
B_{\lambda_{x_0}}(x_0)\backslash\{x_0\}$, it holds
\begin{align*}
Q_{x_0,\lambda_{x_0}}^w(x)&=\frac{1}{8\pi^2}\int_{B_{\lambda_{x_0}}(x_0)}
\ln\left(\frac{|\frac{\lambda_{x_0}(x-x_0)}{|x-x_0|}-\frac{|x-x_0|}{\lambda_{x_0}}
(y-x_0)|}{|x-y|}\right)
(v^4_{x_0,\lambda_{x_0}}(y)-v^4(y))\mathrm{d}y\nonumber\\
&\quad+\Big(\frac{5}{p}-\alpha\Big)\ln\frac{\lambda_{x_0}}{|x-x_0|}\nonumber\\
&\leq\Big(\frac{5}{p}-\alpha\Big)\ln\frac{\lambda_{x_0}}{|x-x_0|}
<0,
\end{align*}
which combines with \eqref{eq614} and \eqref{eq615} indicate that
\begin{align*}
Q_{x_0,\lambda_{x_0}}^u(x)=\frac{1}{4\pi^2}\int_{B_{\lambda_{x_0}}(x_0)}
\Bigg[\frac{1}{|x-y|^{3}}-\frac{1}{|\frac{\lambda_{x_0}(x-x_0)}{|x-x_0|}-\frac{|x-x_0|}
{\lambda_{x_0}}(y-x_0)|^{3}}\Bigg]
\left(e^{pw_{x_0,\lambda_{x_0}}(y)}-e^{pw(y)}\right)\mathrm{d}y
<0
\end{align*}
and
\begin{align*}
Q_{x_0,\lambda_{x_0}}^v(x)=\frac{1}{4\pi^2}\int_{B_{\lambda_{x_0}}(x_0)}
\Bigg[\frac{1}{|x-y|^{2}}
-\frac{1}{|\frac{\lambda_{x_0}(x-x_0)}{|x-x_0|}-\frac{|x-x_0|}{\lambda_{x_0}}(y-x_0)|^{2}}\Bigg]\left(u^2_{\lambda_{x_0}}(y)-u^2(y)\right)\mathrm{d}y
<0.
\end{align*}

Define the narrow region
\begin{align}\label{eq438}
A_{\delta_1}:=\{x\in\R^4\ |0<|x-x_0|<\delta_1 \ \text{or} \ \lambda_{x_0}-\delta_1<|x-x_0|<\lambda_{x_0}\}\subset B_{\lambda_{x_0}}(x_0)\backslash\{x_0\},
\end{align}
where $\delta_1>0$ small enough will be determined later. Since $Q_{x_0,\lambda_{x_0}}^u,\ Q_{x_0,\lambda_{x_0}}^v \ \text{and}\  Q_{x_0,\lambda_{x_0}}^w$ are continuous about $x$ in $\R^4\backslash\{x_0\}$ and $A_{\delta_1}^c:=(B_{\lambda_{x_0}}(x_0)\backslash\{x_0\})\backslash A_{\delta_1}$ is a compact subset, there exists a positive constant $c_0>0$ such that
$$Q_{x_0,\lambda_{x_0}}^u(x)<-c_0,\quad Q_{x_0,\lambda_{x_0}}^v(x)<-c_0\quad\text{and}\quad Q_{x_0,\lambda_{x_0}}^w(x)<-c_0,\quad \
\forall \ x\in A_{\delta_1}^c.$$
Due to the continuity of $Q_{x_0,\lambda_{x_0}}^u,\ Q_{x_0,\lambda_{x_0}}^v \ \text{and}\  Q_{x_0,\lambda_{x_0}}^w$ w.r.t. $\la$, we can choose $\delta_2>0$ small enough such that, for any $\la\in[\la_{x_0}-\delta_2,\la_{x_0}]$,
$$Q_{x_0,\lambda}^u(x)<-\frac{c_0}{2},\quad Q_{x_0,\lambda}^v(x)<-\frac{c_0}{2}\quad\text{and}\quad Q_{x_0,\lambda}^w(x)<-\frac{c_0}{2},\quad \
\forall \ x\in A_{\delta_1}^c,$$
which indicates that
\begin{align}\label{eq439}
B_{\lambda,u}^+(x_0)\cup B_{\lambda,v}^+(x_0)\cup B_{\lambda,w}^+(x_0)&\subset(B_{\lambda}(x_0)\backslash\{x_0\})\backslash A_{\delta_1}^c\nonumber\\
&=
\{x\in\R^4|0<|x-x_0|<\delta_1 \ \text{or} \ \lambda_{x_0}-\delta_1<|x-x_0|<\lambda\}
\end{align}
for any $\lambda\in[\lambda_{x_0}-\delta_2,\lambda_{x_0}]$. From \eqref{eq419}, \eqref{eq420} and \eqref{eq421}, we get
\begin{align}\label{eq440}
\|Q_{x_0,\lambda}^u(x)\|_{L^q(B_{\lambda,u}^+(x_0))} &\leq C_{\lambda,\var}\|e^{pw_{x_0,\lambda}(x)}\|_{L^{\frac{6}{5}}(B_{\lambda,w}^+(x_0))}
\|u_{x_0,\lambda}(x)\|_{L^{3}(B_{\lambda,u}^+(x_0))}
\|Q_{x_0,\lambda}^u(x)\|_{L^q(B_{\lambda,u}^+(x_0))}\nonumber\\
&\quad\times\lr\|v_{x_0,\lambda}(x)\|^3_{L^{\frac{36}{7-6\var}}(B_{\lambda,v}^+(x_0))}
+|B_{\lambda,w}^+(x_0)|^{\frac{1}{t}}\|v_{x_0,\lambda}(x)\|_{L^{\frac{3r}{r-1}}(B_{\lambda,v}^+(x_0))}\rr,
\end{align}
\begin{align}\label{eq441}
\|Q_{x_0,\lambda}^v(x)\|_{L^r(B_{\lambda,v}^+(x_0))} &\leq C_{\lambda,\var}\|e^{pw_{x_0,\lambda}(x)}\|_{L^{\frac{6}{5}}(B_{\lambda,w}^+(x_0))}
\|u_{x_0,\lambda}(x)\|_{L^{3}(B_{\lambda,u}^+(x_0))}
\|Q_{x_0,\lambda}^v(x)\|_{L^r(B_{\lambda,v}^+(x_0))}\nonumber\\
&\quad\times\lr\|v_{x_0,\lambda}(x)\|^3_{L^{\frac{36}{7-6\var}}(B_{\lambda,v}^+(x_0))}
+|B_{\lambda,w}^+(x_0)|^{\frac{1}{t}}
\|v_{x_0,\lambda}(x)\|_{L^{\frac{3r}{r-1}}(B_{\lambda,v}^+(x_0))}\rr
\end{align}
and
\begin{align}\label{eq442}
\|Q_{x_0,\lambda}^w(x)\|_{L^t(B_{\lambda,w}^+(x_0))} &\leq C_{\lambda,\var}\|e^{pw_{x_0,\lambda}(x)}\|_{L^{\frac{6}{5}}(B_{\lambda,w}^+(x_0))}
\|u_{x_0,\lambda}(x)\|_{L^{3}(B_{\lambda,u}^+(x_0))}
\|Q_{x_0,\lambda}^w(x)\|_{L^t(B_{\lambda,w}^+(x_0))}\nonumber\\
&\quad\times\lr\|v_{x_0,\lambda}(x)\|^3_{L^{\frac{36}{7-6\var}}(B_{\lambda,v}^+(x_0))}
+|B_{\lambda,w}^+(x_0)|^{\frac{1}{t}}
\|v_{x_0,\lambda}(x)\|_{L^{\frac{3r}{r-1}}(B_{\lambda,v}^+(x_0))}\rr,
\end{align}
where $q=\frac{12t}{12+7t}, r=\frac{12t}{12+5t}$. Since $\alpha\geq\frac{5}{p}$, from asymptotic properties of $(u,v,w)$ in \eqref{514}, \eqref{516} and \eqref{542eq}, we deduce that, for $\var\in(0,\frac{1}{2})$ sufficiently small and $t\in(\frac{2}{\var},+\infty)$ large enough, $e^{pw_{x_0,\lambda}}\in L^{\frac{6}{5}}(\R^4)$, $u_{x_0,\la}\in L^{3}(\R^4)$ and $v_{x_0,\la}\in L^{\frac{3r}{r-1}}(\R^4)\cap L^{\frac{36}{7-6\var}}(\R^4)$. According to \eqref{eq439}, we can choose $\delta_1$ and $\delta_2$ sufficiently small such that, for any $\lambda\in[\lambda_{x_0}-\delta_2,\lambda_{x_0}]$,
\begin{align*}
&\quad\quad C_{\lambda,\var}\|e^{pw_{x_0,\lambda}(x)}\|_{L^{\frac{6}{5}}(B_{\lambda,w}^+(x_0))}
\|u_{x_0,\lambda}(x)\|_{L^{3}(B_{\lambda,u}^+(x_0))}\nonumber\\
&\times\lr\|v_{x_0,\lambda}(x)\|^3_{L^{\frac{36}{7-6\var}}(B_{\lambda,v}^+(x_0))}
+|B_{\lambda,w}^+(x_0)|^{\frac{1}{t}}
\|v_{x_0,\lambda}(x)\|_{L^{\frac{3r}{r-1}}(B_{\lambda,v}^+(x_0))}\rr<\frac{1}{2},
\end{align*}
Combining this with \eqref{eq440}, \eqref{eq441} and \eqref{eq442}, we have, for any $\lambda\in[\lambda_{x_0}-\delta_2,\lambda_{x_0}]$,
$$\|Q_{x_0,\lambda}^u(x)\|_{L^{q}(B_{\lambda,u}^+(x_0))}=\|Q_{x_0,\lambda}^v(x)\|_{L^{r}(B_{\lambda,v}^+(x_0))}=\|Q_{x_0,\lambda}^w(x)\|_{L^{t}(B_{\lambda,w}^+(x_0))}=0.$$
It follows immediately that, for any $\lambda\in[\lambda_{x_0}-\delta_2,\lambda_{x_0}]$,
$$Q_{x_0,\lambda}^u(x)\leq0,\quad Q_{x_0,\lambda}^v(x)\leq0\quad\text{and}\quad Q_{x_0,\lambda}^w(x)\leq0,\quad\quad\forall x\in
B_{\lambda}(x_0)\backslash\{x_0\},$$
which contradicts the definition of $\lambda_{x_0}$. Thus we must have $\lambda_{x_0}=0$.

As a consequence, from Lemma \ref{le10} (ii), replacing $v$ by $-v$, we derive $v\equiv C$. Recall the finite total curvature condition $v^4\in L^1(\R^4)$, it follows that $v\equiv0$. However, from the $4$-D system \eqref{eq41}, one has $0=e^{pw(x)}>0$ in $\R^4$, which is a contradiction and hence $\alpha>\frac{5}{p}$ is false.

\smallskip

\textbf{Case 2:} $\frac{1}{p}\leq\alpha<\frac{5}{p}$. In this case, we define the critical scale $\la_{x_0}$ by
\begin{align}\label{eq443}
\la_{x_0}:=\sup\{\lambda>0\ |\ Q_{x_0,\mu}^u\geq0,\ Q_{x_0,\mu}^v\geq0,\ Q_{x_0,\mu}^w\geq0,\ \text{in}\ \ B_\mu(x_0)\backslash\{x_0\},\ \forall \ 0<\mu\leq\lambda\}.
\end{align}
From step 1 it can be seen that $\la_{x_0}$ is well defined and  $0<\la_{x_0}\leq+\infty$ for any $x_{0}\in\mathbb{R}^{4}$. We will show that $\lambda_{x_0}=+\infty$, which will lead to a contradiction again as in $\mathbf{Case\ 1:}$ $\alpha>\frac{5}{p}$. Suppose on the contrary that $\la_{x_0}<\infty$, we will prove that the sphere can be moved outward a bit further, which contradicts the definition of $\la_{x_0}$.

According to the definition of $\la_{x_0}$, we have
$$Q_{x_0,\la_{x_0}}^u\geq0,\quad Q_{x_0,\la_{x_0}}^v\geq0,\quad Q_{x_0,\la_{x_0}}^w\geq0,\quad\text{in} \
B_{\la_{x_0}}(x_0)\backslash\{x_0\}.$$
From \eqref{eq616}, for any $x\in
B_{\la_{x_0}}(x_0)\backslash\{x_0\}$, it holds
\begin{align*}
Q_{x_0,\la_{x_0}}^w(x)&=\frac{1}{8\pi^2}\int_{B_{\la_{x_0}}(x_0)}
\ln\left(\frac{|\frac{\la_{x_0}(x-x_0)}{|x-x_0|}-\frac{|x-x_0|(y-x_0)}{\la_{x_0}}|}
{|x-y|}\right)(v^4_{x_0,\lambda}(y)-v^4(y))\mathrm{d}y\\
&\quad+\Big(\frac{5}{p}-\alpha\Big)\ln\frac{\lambda_{x_0}}{|x-x_0|}\nonumber\\
&\geq\Big(\frac{5}{p}-\alpha\Big)\ln\frac{\la_{x_0}}{|x-x_0|}
>0,
\end{align*}
which combines with \eqref{eq614} and \eqref{eq615} indicate that
\begin{align*}
Q_{x_0,\la_{x_0}}^u(x)=\frac{1}{4\pi^2}\int_{B_{\la_{x_0}}(x_0)}
\Bigg(\frac{1}{|x-y|^{3}}
-\frac{1}{\left|\frac{\la_{x_0}(x-x_0)}{|x-x_0|}-\frac{|x-x_0|(y-x_0)}{\la_{x_0}}\right|^{3}}\Bigg)
\left(e^{pw_{x_0,\lambda}(y)}-e^{pw(y)}\right)\mathrm{d}y
>0
\end{align*}
and
\begin{align*}
Q_{x_0,\la_{x_0}}^v(x)=\frac{1}{4\pi^2}
\int_{B_{\la_{x_0}}(x_0)}\Bigg(\frac{1}{|x-y|^{2}}-\frac{1}{\left|
\frac{\la_{x_0}(x-x_0)}{|x-x_0|}-\frac{|x-x_0|(y-x_0)}{\la_{x_0}}\right|^{2}}
\Bigg)\left(u^2_{x_0,\lambda}(y)-u^2(y)\right)\mathrm{d}y
>0.
\end{align*}
Recall that the narrow region $A_{\delta_1}$ is defined in \eqref{eq438}. Since $A_{\delta_1}^c:=(B_{\lambda_{x_0}}(x_0)\backslash\{x_0\})\backslash A_{\delta_1}$ is a compact subset and $Q_{x_0,\lambda_{x_0}}^u,\ Q_{x_0,\lambda_{x_0}}^v \ \text{and}\  Q_{x_0,\lambda_{x_0}}^w$ are continuous w.r.t. $x$ in $\R^4\backslash\{x_0\}$, there exists a positive constant $C_1>0$ such that
$$Q_{x_0,\lambda_{x_0}}^u(x)>C_1,\quad Q_{x_0,\lambda_{x_0}}^v(x)>C_1\quad\text{and}\quad Q_{x_0,\lambda_{x_0}}^w(x)>C_1,\quad
\forall \ x\in A_{\delta_1}^c.$$
Furthermore, we can choose $\delta_2>0$ small enough such that, for any $\la\in[\la_{x_0},\la_{x_0}+\delta_2]$,
$$Q_{x_0,\lambda}^u(x)>\frac{C_1}{2},\quad Q_{x_0,\lambda}^v(x)>\frac{C_1}{2}\quad\text{and}\quad Q_{x_0,\lambda}^w(x)>\frac{C_1}{2},\quad \
\forall \ x\in A_{\delta_1}^c,$$
which implies that
\begin{align*}
B_{\lambda,u}^-(x_0)\cup B_{\lambda,v}^-(x_0)\cup B_{\lambda,w}^-(x_0)&\subset(B_{\lambda_{x_0}}(x_0)\backslash\{x_0\})\backslash A_{\delta_1}^c\nonumber\\
&=
\{x\in\R^4|0<|x-x_0|<\delta_1 \ \text{or} \ \lambda_{x_0}-\delta_1<|x-x_0|<\lambda\}
\end{align*}
for any $\la\in[\la_{x_0},\la_{x_0}+\delta_2]$. By \eqref{eq639}, \eqref{eq640} and \eqref{eq641}, we obtain
\begin{align}\label{eq659}
\|Q_{x_0,\lambda}^u(x)\|_{L^q(B_{\lambda,u}^-(x_0))} &\leq C_{\lambda,\var}\|e^{pw(x)}\|_{L^{\frac{6}{5}}(B_{\lambda,w}^-(x_0))}
\|u(x)\|_{L^{3}(B_{\lambda,u}^-(x_0))}\|Q_{x_0,\lambda}^u(x)\|_{L^q(B_{\lambda,u}^-(x_0))} \nonumber\\
&\quad\times\lr\|v(x)\|^3_{L^{\frac{36}{7-6\var}}(B_{\lambda,v}^-(x_0))}
+|B_{\lambda,w}^-(x_0)|^{\frac{1}{t}}\|v(x)\|^3_{L^{\frac{3r}{r-1}}(B_{\lambda,v}^-(x_0))}\rr,
\end{align}
\begin{align}\label{eq660}
\|Q_{x_0,\lambda}^v(x)\|_{L^r(B_{\lambda,v}^-(x_0))} &\leq C_{\lambda,\var}\|e^{pw(x)}\|_{L^{\frac{6}{5}}(B_{\lambda,w}^-(x_0))}
\|u(x)\|_{L^{3}(B_{\lambda,u}^-(x_0))}\|Q_{x_0,\lambda}^v(x)\|_{L^r(B_{\lambda,v}^-(x_0))} \nonumber\\
&\quad\times\lr\|v(x)\|^3_{L^{\frac{36}{7-6\var}}(B_{\lambda,v}^-(x_0))}
+|B_{\lambda,w}^-(x_0)|^{\frac{1}{t}}\|v(x)\|^3_{L^{\frac{3r}{r-1}}(B_{\lambda,v}^-(x_0))}\rr
\end{align}
and
\begin{align}\label{eq661}
\|Q_{x_0,\lambda}^w(x)\|_{L^t(B_{\lambda,w}^-(x_0))} &\leq C_{\lambda,\var}\|e^{pw(x)}\|_{L^{\frac{6}{5}}(B_{\lambda,w}^-(x_0))}
\|u(x)\|_{L^{3}(B_{\lambda,u}^-(x_0))}\|Q_{x_0,\lambda}^w(x)\|_{L^t(B_{\lambda,w}^-(x_0))} \nonumber\\
&\quad\times\lr\|v(x)\|^3_{L^{\frac{36}{7-6\var}}(B_{\lambda,v}^-(x_0))}
+|B_{\lambda,w}^-(x_0)|^{\frac{1}{t}}\|v(x)\|^3_{L^{\frac{3r}{r-1}}(B_{\lambda,v}^-(x_0))}\rr.
\end{align}
From the local boundness of $u$, $v$ and $w$, one can choose $\delta_1$ and $\delta_2$ small enough such that, for any $\lambda\in[\lambda_{x_0},\lambda_{x_0}+\delta_2]$,
\begin{align*}
&\quad\quad C_{\lambda,\var}\|e^{pw(x)}\|_{L^{\frac{6}{5}}(B_{\lambda,w}^-(x_0))}
\|u(x)\|_{L^{3}(B_{\lambda,u}^-(x_0))}\\
&\times\lr\|v(x)\|^3_{L^{\frac{36}{7-6\var}}(B_{\lambda,v}^-(x_0))}
+|B_{\lambda,w}^-(x_0)|^{\frac{1}{t}}\|v(x)\|^3_{L^{\frac{3r}{r-1}}
(B_{\lambda,v}^-(x_0))}\rr
<\frac{1}{2},
\end{align*}
which combines with \eqref{eq659}, \eqref{eq660} and \eqref{eq661} indicate that, for any $\lambda\in[\lambda_{x_0},\lambda_{x_0}+\delta_2]$,
$$\|Q_{x_0,\lambda}^u(x)\|_{L^{q}(B_{\lambda,u}^-(x_0))}=\|Q_{x_0,\lambda}^v(x)\|_{L^{r}(B_{\lambda,v}^-(x_0))}=\|Q_{x_0,\lambda}^w(x)\|_{L^{t}(B_{\lambda,w}^-(x_0))}=0.$$
Consequently, for any $\lambda\in[\lambda_{x_0},\lambda_{x_0}+\delta_2]$,
$$Q_{x_0,\lambda}^u(x)\geq0,\quad Q_{x_0,\lambda}^v(x)\geq0\quad\text{and}\quad Q_{x_0,\lambda}^w(x)\geq0,\quad\quad\forall x\in
B_{\lambda}(x_0)\backslash\{x_0\}.$$
This contradicts the definition \eqref{eq443} of $\lambda_{x_0}$. Hence we must have $\lambda_{x_0}=+\infty$.

From Lemma \ref{le10} (ii) and the finite total curvature condition $v^4\in L^1(\R^4)$, we obtain $v\equiv0$. However, by the $4$-D system \eqref{eq41}, we have $0=e^{pw(x)}>0$ in $\R^4$, which yields again a contradiction. Thus $\frac{1}{p}\leq\alpha<\frac{5}{p}$ is absurd.

\medskip

From the contradictions derived in both cases 1 and 2, we conclude
\begin{align}\label{eq662}
\alpha:=\frac{1}{8\pi^2}\int_{\mathbb{R}^{4}}v^{4}(y)\mathrm{d}y=\frac{5}{p}.
\end{align}
By Step 1, we infer that, for $\lambda>0$ large, there holds
\begin{equation}\label{ddz2}
Q^{u}_{x_0,\la}(x)\leq0,\ Q^{v}_{x_0,\la}(x)\leq0,\ Q^{w}_{x_0,\la}(x)\leq0 \quad\ \text{in}\ B_{\lambda}(x_0)\backslash\{x_0\},
\end{equation}
while for $\lambda>0$ small,
\begin{equation}\label{d24d}
Q^{u}_{x_0,\la}(x)\geq0,\ Q^{v}_{x_0,\la}(x)\geq0,\ Q^{w}_{x_0,\la}(x)\geq0 \quad\ \text{in}\ B_{\lambda}(x_0)\backslash\{x_0\}.
\end{equation}
If the critical scale (defined in \eqref{eq443}) $\la_{x_0}<+\infty$, we must get $Q_{x_0,\lambda_{x_0}}^{u}(x)=Q_{x_0,\lambda_{x_0}}^{v}(x)=Q_{x_0,\lambda_{x_0}}^{w}(x)=0 \,\, \text{in}\ B_{\lambda_{x_0}}(x_0)\backslash\{x_0\}$. Otherwise, the sphere $S_\la$ can also be moved outward a bit further such that \eqref{d24d} still holds (see Case 1 or Case 2 in Step 2), which contradicts the definition \eqref{eq443} of $\la_{x_0}$. If the critical scale (defined in \eqref{eq443}) $\la_{x_0}=+\infty$, from \eqref{ddz2}, we obtain $Q^{u}_{x_0,\la}(x)=Q^{v}_{x_0,\la}(x)=Q^{w}_{x_0,\la}(x)=0\ \text{in}\ B_{\lambda}(x_0)\backslash\{x_0\}$ for $\la$ large enough. As a consequence, for arbitrary $x_{0}\in\mathbb{R}^{4}$, there exists a $\la>0$ depending on $x_0$ such that
$$Q^{u}_{x_0,\la}(x)=Q^{v}_{x_0,\la}(x)=Q^{w}_{x_0,\la}(x)=0,\qquad\forall x\in B_{\lambda}(x_0)\backslash\{x_0\}.$$

Therefore, we deduce from the Lemma \ref{le10} (i) that, for some $C_{1},C_{2}\in \mathbb{R}$, $\mu>0$ and $x_{0}\in \mathbb{R}^{3}$, $u,v$ must be of the form
\begin{equation}\label{553}u(x)=\left(\frac{C_{1}\mu}{1+\mu^{2}|x-x_{0}|^{2}}\right)^
\frac{3}{2},\qquad \forall\ x\in \mathbb{R}^{4}\end{equation}
and
\begin{equation}\label{554}v(x)=\frac{C_{2}\mu}{1+\mu^{2}|x-x_{0}|^{2}},\qquad \forall\ x\in \mathbb{R}^{4}.\end{equation}
Combining this with the system \eqref{eq41} and the asymptotic behavior \eqref{419} indicate
\begin{equation}\label{555}w(x)=\frac{5}{2p}\ln\left(\frac{C_{3}\mu}{1+\mu^{2}|x-x_{0}|^{2}}\right),\qquad \forall\ x\in \mathbb{R}^{4},\end{equation}
where $C_3>0$ is a constant. By direct calculations, \eqref{eq662} and \eqref{554}, we obtain
$C_2=\left(\frac{480}{p}\right)^{\frac{1}{4}}$ and hence
$$v(x)=\frac{\left(\frac{480}{p}\right)^{\frac{1}{4}}\mu}{1+\mu^{2}|x-x_{0}|^{2}},\qquad \forall\ x\in \mathbb{R}^{4}.$$
Combining the asymptotic behavior \eqref{542eq} with \eqref{553}, we get
$C_1=\left(\frac{2^{17}5}{p}\right)^{\frac{1}{12}}$ and then
\begin{equation*}
u(x)=\left(\frac{\left(\frac{2^{17}5}{p}\right)^{\frac{1}{12}}\mu}{1+\mu^{2}|x-x_{0}|^{2}}\right)^\frac{3}{2},\qquad \forall\ x\in \mathbb{R}^{4}.\end{equation*}
Similarly, by combining the asymptotic behavior \eqref{516} with \eqref{555}, we have
$C_3=\left(\frac{2^{49}3^85}{p}\right)^{\frac{1}{20}}$ and hence
$$w(x)=\frac{5}{2p}\ln\left(\frac{\left(\frac{2^{49}3^85}{p}\right)^{\frac{1}{20}}\mu}{1+\mu^{2}|x-x_{0}|^{2}}\right),\qquad \forall\ x\in \mathbb{R}^{4}.$$
This concludes our proof of Theorem \ref{thm1}.

\end{document}